\documentclass[11pt]{article}
\usepackage[utf8x]{inputenc}
\usepackage[a4paper]{geometry}
\usepackage{amsmath}
\usepackage{amsfonts}
\usepackage{amsthm}
\usepackage{amssymb}
\usepackage{tikz}
\usepackage{tikz-cd}
\usepackage{indentfirst}
\usepackage{authblk}
\usepackage{graphicx}
\usepackage{enumerate}

\theoremstyle{plain}
\newtheorem{thm}{Theorem}
\newtheorem{lem}[thm]{Lemma}
\newtheorem{prop}[thm]{Proposition}
\newtheorem{cor}[thm]{Corollary}
\newtheorem*{thm*}{Theorem}

\newtheorem{thmx}{Theorem}
\theoremstyle{definition}
\newtheorem{defn}[thm]{Definition}

\theoremstyle{remark}
\newtheorem{ex}[thm]{Example}
\newtheorem{rmk}[thm]{Remark}

\makeatletter
\DeclareFontFamily{U}{tipa}{}
\DeclareFontShape{U}{tipa}{m}{n}{<->tipa10}{}
\newcommand{\arc@char}{{\usefont{U}{tipa}{m}{n}\symbol{62}}}%

\newcommand{\arc}[1]{\mathpalette\arc@arc{#1}}

\newcommand{\arc@arc}[2]{%
  \sbox0{$\m@th#1#2$}%
  \vbox{
    \hbox{\resizebox{\wd0}{\height}{\arc@char}}
    \nointerlineskip
    \box0
  }%
}
\makeatother

\newcommand{\mc}{\mathcal}
\newcommand{\mbb}{\mathbb}
\newcommand{\mbf}{\mathbf}

\newcommand{\mf}{\mathfrak}
\newcommand{\ol}{\overline}
\newcommand{\vtheta}{\vartheta}
\newcommand{\vpi}{\varpi}
\renewcommand{\AA}{\mbb A}

\newcommand{\NN}{\mbb N}
\newcommand{\ZZ}{\mbb Z}
\newcommand{\QQ}{\mbb Q}
\newcommand{\RR}{\mbb R}

\newcommand{\GG}{\mbb G}
\newcommand{\oo}{\circ}
\newcommand{\interior}[1]{\overset{\circ}{\arc{#1}}}
\newcommand{\longto}{\longrightarrow}

\newcommand{\epi}{\twoheadrightarrow}

\renewcommand{\parallel}{/\!/}
\renewcommand{\tilde}{\widetilde}

\DeclareMathOperator{\nd}{nd}

\DeclareMathOperator{\red}{red}
\DeclareMathOperator{\inta}{int}
\DeclareMathOperator{\id}{id}

\DeclareMathOperator{\res}{res}

\DeclareMathOperator{\Gal}{Gal}
\DeclareMathOperator{\an}{an}
\DeclareMathOperator{\pr}{pr}
\DeclareMathOperator{\dom}{dom}

\DeclareMathOperator{\Lie}{Lie}

\DeclareMathOperator{\diag}{diag}
\DeclareMathOperator{\Spec}{Spec}
\DeclareMathOperator{\Stab}{Stab}

\DeclareMathOperator{\Hom}{Hom}

\DeclareMathOperator{\Grp}{Grp}
\DeclareMathOperator{\pred}{pred}
\DeclareMathOperator{\Span}{Span}
\DeclareMathOperator{\SL}{SL}
\DeclareMathOperator{\Sp}{Sp}
\DeclareMathOperator{\BC}{BC}
\DeclareMathOperator{\SU}{SU}
\DeclareMathOperator{\nr}{nr}
\DeclareMathOperator{\Fix}{Fix}
\DeclareMathOperator{\Aut}{Aut}

\usepackage{hyperref}

\title{Polyhedral compactifications of Bruhat-Tits buildings of quasi-reductive groups}
\author{Dorian Chanfi\footnote{Mathematisches Institut, Justus-Liebig-Universit\"at Gie\ss en\\Arndtstra\ss e 2, 35392 Gie\ss en\\E-mail: \href{mailto:dorian.chanfi@math.uni-giessen.de}{dorian.chanfi@math.uni-giessen.de}}}
\date{\today}

\begin{document}

\maketitle

\begin{abstract}
    Given a quasi-reductive group $G$ over a local field $k$, using Berkovich geometry, we exhibit a family of $G(k)$-equivariant compactifications of the Bruhat-Tits building $\mc B(G, k)$, constructed and investigated by Solleveld and Lourenço. The compactification procedure consists in mapping the building in the analytification $G^{\an}$ of $G$, then composing this map with the projections from $G^{\an}$ to its (in general non-compact) pseudo-flag varieties $(G/P)^{\an}$, for $P$ ranging among the pseudo-parabolic subgroups of $G$. This generalises previous constructions of Berkovich, then Rémy, Thuillier and Werner. 
    
    To define the embedding, we are led to giving a partial extension to the quasi-reductive context of results due to Rousseau on the functoriality of Bruhat-Tits buildings with respect to field extensions, which are of independent interest.

    Finally, we conclude by investigating the geometry at infinity of these compactifications. The boundaries are shown to be stratified, each stratum being equivariantly homeomorphic to the Bruhat-Tits building of the maximal quasi-reductive quotient of a pseudo-parabolic subgroup.
\end{abstract}

\newpage

\tableofcontents
 
\section*{Introduction}

This article is devoted to the construction of equivariant compactifications of Bruhat-Tits buildings of quasi-reductive groups. Pseudo-reductive and quasi-reductive groups were introduced with the aim of studying the combinatorics of rational points of arbitrary connected linear algebraic groups over arbitrary fields. They form a class of groups that generalises that of reductive groups, defined by Borel and Tits in \cite{BTCRAS} and \cite{TitsCDF}, then studied extensively by Conrad, Gabber and Prasad in \cite{CGP} and \cite{CP}. 

Recall that a smooth connected affine group $G$ over a field $k$ is said to be reductive if $G_{\ol k}$ contains no nontrivial smooth connected unipotent normal subgroups. More generally, the group $G$ is said to be \textit{pseudo-reductive} (resp. \textit{quasi-reductive}) if it contains no nontrivial smooth connected unipotent normal (resp. $k$-split unipotent) $k$-subgroups. 

When $k$ is a perfect field, the unipotent radical $\mc R_u(G_{\ol k})$ of $G_{\ol k}$, that is the maximal smooth connected unipotent normal subgroup of $G_{\ol k}$, is defined over $k$ and $k$-split, hence the notions of reductivity, pseudo-reductivity, and quasi-reductivity coincide in that case. Over imperfect fields, however, the classes are distinct \cite[Example 1.1.3]{CGP} and the role of pseudo-reductive groups in the study of linear algebraic groups is analogous to that of reductive groups in the case of perfect fields, in the sense that any smooth connected affine group over a field is an extension of a pseudo-reductive group by a unipotent group. Moreover, the structure theory of \cite{CGP} provides a good enough understanding of pseudo-reductive groups to allow one to extend results about smooth connected affine groups previously known over perfect fields to imperfect fields. For example, it allowed Conrad in \cite{Conrad} to extend previous results of Oesterlé's \cite[IV, 5.6]{Oesterle} and Borel's \cite[5.8]{Borel_Tamagawa} on the finiteness of Tamagawa numbers to connected groups over imperfect fields.

One upshot of the structure theory of pseudo-reductive groups developed in \cite{CGP}, notably the fact that many are \textit{generalised standard} \cite[Theorem 5.1.1]{CGP}, is that the rational points of quasi-reductive groups over a field $k$ are close to being direct products of groups of rational points of reductive groups (over extensions of $k$). This observation makes it plausible that, when $k$ is valued, the study of rational points of quasi-reductive groups over $k$ should be covered by some generalised version of Bruhat-Tits theory. 

Recall that, in the papers \cite{BT1} and \cite{BT2}, Bruhat and Tits associated to each reductive group $G$ over a valued field $k$ an affine \textit{building} $\mc B(G, k)$ endowed with an action of $G(k)$, commonly called its \textit{Bruhat-Tits building}. An affine building is a polysimplicial complex covered by subcomplexes called \textit{apartments}, each of which is isomorphic to a given Euclidean tiling, the apartment system being subject to rigid incidence relations \cite[IV.1]{Brown}. Bruhat-Tits buildings provide rich geometries that encode a lot of the structure of reductive groups over non-archimedean valued fields, analogously to symmetric spaces for real non-compact Lie groups. They are endowed with a natural non-positively curved metric and their vertices parameterise maximal compact subgroups.

In \cite{Sol}, Solleveld associates a building to any quasi-reductive group $G$ over a local field $k$ a Euclidean building endowed with an action of $G(k)$ that reduces to the Bruhat-Tits building when $G$ is reductive. The main observation is, as previously mentioned, the fact that the group $G(k)$ is, up to extension by a compact group and modification of Cartan subgroups, a finite direct product of groups of rational points of reductive groups $\prod_{i\in I} G'_i(k'_i)$ over finite extensions $k'_i$ of $k$. The building of $G$ over $k$ is then defined as the product $\prod_{i \in I} \mc B(G'_i, k'_i)$.
In \cite{Lou}, Lourenço extends Solleveld's construction to quasi-reductive groups over discretely valued fields with perfect residue fields. The article closely follows the layout of \cite{BT2} and one of its key new features is its extension of the arithmetic theory of Bruhat-Tits integral models to the quasi-reductive setting: To each bounded subset $\Omega \subset \mc B(G, k)$, we can associate integral models of $G$ whose integral points are stabilisers of $\Omega$ in $G(k)$ (pointwise, global, and other variants \cite[§§3.2, 4.2]{Lou}).

This brings us to the subject matter of our article, namely compactifications of Bruhat-Tits buildings of quasi-reductive groups over local fields. A motivating factor for compactifying Bruhat-Tits buildings is the analogy between affine buildings and symmetric spaces, whose compactifications have found applications in group theory -- see eg. \cite{BS1} for applications to finiteness properties of $S$-arithmetic groups, or Mostow's proof of his strong rigidity theorem \cite{Mostow} -- and arithmetic geometry -- where arithmetic quotients of symmetric spaces can sometimes be interpreted as moduli spaces of algebraic varieties. An example of this analogy coming to fruition is provided by \cite{Prasad_rigidity}, in which the author proves a rigidity result for lattices in semisimple $p$-adic groups by letting geodesic compactifications \cite{BH} of Bruhat-Tits buildings play the role that geodesic compactifications of symmetric spaces played in \cite{Mostow}.

In the following, we shall specifically concern ourselves with polyhedral compactifications of Bruhat-Tits buildings, studied under various guises in, among others, \cite{Landvogt}, \cite{Werner}, and \cite{RTW1}. These compactifications fit in an ordered family (for the domination order) indexed by conjugacy classes of pseudo-parabolic subgroups. Each space admits a polysimplicial structure and its boundary naturally decomposes as a union of Bruhat-Tits buildings of smaller rank, much like in Satake's construction of compactifications of symmetric spaces \cite{Satake}. The analogy between both constructions is given a thorough treatment in \cite{Werner} and \cite{RTW2}. In \cite{GR}, the authors establish that the vertices of the maximal compactification of $\mc B(G, k)$ parameterise closed amenable subgroups of $G(k)$ with Zariski-connected closure that are maximal for these properties. This suggests that a similar result should hold for an arbitrary quasi-reductive group $G$ over a local field.

In this article, we take the analytic-geometric point of view on compactifications of Bruhat-Tits buildings of quasi-reductive groups over local fields, as adopted in \cite{Werner2} and \cite{RTW1}.

Our main goal is to establish the following theorem, which synthesises most of the results of Sections 2 to 4.

\begin{thmx}
    Let $G$ be a quasi-reductive group over a local field $k$. Denote by $\mc B(G, k)$ the Bruhat-Tits building of $G$ over $k$. Let $P$ be a pseudo-parabolic subgroup of $G$.
    \begin{enumerate}
        \item There exists a continuous $G(k)$-equivariant map $$\vtheta_P: \mc B(G, k) \to (G/P)^{\an},$$ where $\an$ is the Berkovich analytic space associated to $G/P$.
        \item The map $\vtheta_P$ has relatively compact image and is an embedding if and only if $P$ is of non-degenerate type\footnote{See Definition \ref{def:non-degenerate_type}.}. The homeomorphism type of the closure of the image depends only on the conjugacy class of $P$.
        \item Whenever $P$ is of non-degenerate type, the closure of the image of $\vtheta_P$ admits a stratification by affine buildings. Precisely, the compactification $\ol{\vtheta_P(\mc B(G, k))}$ is $G(k)$-equivariantly homeomorphic to a disjoint union $$\ol{\vtheta_P(\mc B(G, k))} = \bigsqcup_{\substack{Q \subset G}} \mc B(Q/\mc R_{us,k}(Q), k),$$ where $Q$ ranges among the $J$-relevant pseudo-parabolic subgroups\footnote{See Definition \ref{def:relevant_pseudo-par}.} of $G$.
	For each $J$-relevant pseudo-parabolic subgroup $Q$ of $G$, the group $Q(k)$ is the stabiliser of the stratum $\mc B(Q/\mc R_{us,k}(Q), k)$.
    \end{enumerate}
\end{thmx}

The proof loosely follows the structure of \cite{RTW1}, but substantial new difficulties must be overcome.
The first difficulty is the construction of the map $\vtheta_P$, which relies on the existence of a canonical continuous $G(k)$-equivariant map $$\vtheta: \mc B(G, k) \to G^{\an}.$$

The idea in \cite{Ber} and \cite{RTW1} was to associate to each point $x \in \mc B(G, k)$ an affinoid subgroup $G_x$ of $G^{\an}$, which contains the information about the stabilisers $G(K)_{x_K}$ of $x_K \in \mc B(G, K)$ for each non-archimedean extension $K/k$. This makes extensive use of Rousseau's theory of the functoriality of Bruhat-Tits buildings with respect to field extensions, developed in \cite{Rou}.
The Shilov boundary of $G_x$ turns out to be a single point, which we denote by $\vtheta(x)$. In other words, each function $f\in \mc O(G_x)$ achieves its maximum at $\vtheta(x)$.
This mapping is reminiscent of the compactification in \cite{GR}, which associated to each special point its stabiliser in $G(k)$, viewed as a point in the Chabauty space of closed subgroups of $G$. By associating to a point the data of all stabilisers $G(K)_{x_K}$, for all non-archimedean extensions, the analytic construction can separate not just all special points, but in fact all points in the buildings.

In our situation, due to technical limitations, we can only carry out this construction for points that become special after a finite Galois extension of $k$, which we dub \textit{virtually special}. Indeed, the theory of Bruhat-Tits buildings for quasi-reductive groups over non-archimedean fields only associates buildings to groups over local fields for \cite{Sol} and discretely valued fields with perfect residue fields for \cite{Lou}, which prevents the use of arbitrarily ramified extensions as in \cite{Ber}. These virtually special points are then shown to be dense, which allows one to extend the map $\vtheta$ from the set of virtually special points to the whole building. This spans the entirety of the second section.

The second difficulty is to establish that $\vtheta_P$ is open with relatively compact image, and injective for suitable $P$. The proofs of these various assertions form the bulk of Section \ref{sect:Compactification}. The main observation is the fact that the restriction of $\vtheta_P$ to an apartment lies in the analytification of an affine scheme, which admits an interpretation in terms of seminorms over the algebra of regular functions. An explicit description of the restriction of the map $\vtheta_P$ to an apartment in terms of seminorms, proved in Proposition \ref{Formula_ThetaP}, plays a central role in most arguments.  Note that the relative compactness is somewhat surprising in this case as, in contrast with the reductive case, the target space $(G/P)^{\an}$ of $\vtheta_P$ is in general not compact. Indeed, by \cite[Theorem 3.4.8 (ii)]{Ber}, the analytification $(G/P)^{\an}$ is compact if and only if $G/P$ is proper, that is if and only if $P$ is a parabolic subgroup. However, quasi-reductive groups do not admit proper parabolic subgroups unless they are reductive \cite[Proposition 3.4.9]{CGP}.

The last point is the description of the boundary of $\vtheta_P(\mc B(G, k))$. Here, we rely on previous work by Charignon on stratified compactifications of affine buildings. In \cite{Char}, the author associates to each affine building $\mc I$, modeled on an affine Coxeter complex $\Sigma$, and each fan $\mc F$ in the direction $\vec \Sigma$ of $\Sigma$ subject to some compatibility relations \cite[2.3]{Char} a canonical compactification $\ol I^{\mc F}$ of $\mc I$, which we call a polyhedral compactification. We establish in Proposition \ref{prop:comparison_compactifications} that the compactification $\ol{\vtheta_P(\mc B(G, k))}$ is a polyhedral compactification in that sense, with respect to a fan that can be described explicitly in terms of the type of $P$. The remainder of Section \ref{sect:Boundary} is devoted to giving a group-theoretic description of the strata, which are constructed in terms of pure building geometry in \cite{Char}. This culminates in Theorem \ref{thm:stratification}.

To conclude, we give a brief description of Section \ref{sect:Functoriality} and its relation to the rest of the article. It is devoted to the proof of the following theorem:

\begin{thmx}
    The construction of the Bruhat-Tits building $\mc B(G, k)$  is functorial in the discretely valued field $k$ with perfect residue field.
\end{thmx}

This theorem plays a central role in the article as it allows one to make sense of the notion of virtually special points, which lies at the heart of the entire construction.

The precise statement can be found in Theorem \ref{Functoriality_inclusions} and generalises, in the setting of discretely valued fields with perfect residue fields, previous results of Rousseau's on the functoriality of Bruhat-Tits buildings of reductive groups with respect to field extensions \cite{Rou}. The proof relies heavily on the combinatorial theory of buildings associated to valuated root groop data developed in \cite{BT1}.

\subsection*{Structure of the article}

The present article splits into four sections. Section \ref{sect:Functoriality}, dealing with the functoriality of Bruhat-Tits buildings with respect to field extensions, may be taken as a blackbox for the remainder of the text. Section \ref{sect:Embedding} constructs a continuous map $$\vtheta: \mc B(G, k) \to G^{\an}.$$ Section \ref{sect:Compactification} establishes that, for suitable pseudo-parabolic subgroups $P$, the composition $$\vtheta_P: \mc B(G, k) \overset{\vtheta}{\to} G^{\an} \to (G/P)^{\an}$$ is a compactification map.
Section \ref{sect:Boundary} studies the geometry at infinity of the compactification and its strata. Finally, the Appendix studies the relation between the buildings associated to centralisers of tori in $G$ and the Bruhat-Tits building of $G$.

\subsection*{Notations and preliminaries}
	Unless otherwise specified, algebraic group over $k$ will mean smooth affine group scheme of finite type over $k$. We will usually adhere to the following notations.
	
	\subsubsection*{Algebraic groups}
	In the following, we denote by $k$ a field and $G$ a connected algebraic $k$-group.
	\begin{itemize}
	\item We denote by $\mc R_{u,k}(G)$ the $k$-unipotent radical of $G$, that is its maximal smooth connected $k$-unipotent normal subgroup. 
	\item Recall that a unipotent $k$-group is said to be \textit{split} if it admits a composition series whose successive factors are isomorphic to $\GG_{a,k}$. Denote by $\mc R_{us, k}(G)$ the $k$-split unipotent radical of $G$, that is its maximal smooth connected $k$-split unipotent normal subgroup.
	\item We say that a connected algebraic $k$-group is \textit{pseudo-reductive} if $\mc R_{u, k}(G) = 1$ and \textit{quasi-reductive} if $\mc R_{us, k}(G) = 1$.
	\item If $G$ is endowed with an action of split torus $S$, then we denote by $\Phi(G, S)$ the set of nonzero weights of the adjoint action of $S$ on the Lie algebra of $G$. 
	\end{itemize}
	In the following definitions, we assume that $G$ is quasi-reductive and that $S$ is a maximal split torus in $G$.
	\begin{itemize} 
	\item $\Phi = \Phi(G, S)$ is a root system called the (relative or $k$-) root system of $G$.
	\item $W(G, S) = N_G(S)/Z_G(S)$ is the Weyl group of $\Phi$, a finite constant group scheme over $k$.
	\item $X^*(S) = \Hom_{\Grp}(S, \mbb G_m)$ is the group of characters of $S$.
	\item $X_*(S) = \Hom_{\Grp}(\mbb G_m, S)$ is the group of cocharacters of $S$.
	\item $\langle \cdot, \cdot \rangle: X^*(S) \times X_*(S) \to \ZZ$ is the canonical pairing: For $\alpha \in X^*(S), \lambda \in X_*(S)$, we have $\alpha \circ \lambda: t \mapsto t^{\langle \alpha, \lambda\rangle}$. 		
	\item We set $$V(G, S) = \frac{X_*(S)\otimes_{\ZZ}\RR}{\Span_{\RR}(\Phi(G, S))^{\perp}} \simeq \frac{X_*(S)\otimes_{\ZZ}\RR}{X_*(Z(G))\otimes_{\ZZ}\RR}\simeq X_*(S') \otimes_{\ZZ}\RR,$$ where $S' = (S\cap \mc D(G))^0_{\red}$.
	\item We set $$V^*(G,S) = \Span_{\RR}(\Phi(G, S))\subset X^*(S)\otimes_{\ZZ}\RR	.$$
	\item The canonical pairing $\langle \cdot, \cdot \rangle$ induces a perfect pairing, which we still denote by $\langle \cdot, \cdot \rangle : V^*(G, S) \times V(G, S) \to \RR.$
	\item Given a cocharacter $\lambda$ of $G$, we let $P_G(\lambda)$, $U_G(\lambda)$, and $Z_G(\lambda)$ be the subgroups associated with $\lambda$ \cite[13.d]{Milne}, characterised functorially by the condition that, for each $k$-algebra $A$, we have $$\begin{array}{ccl}P_G(\lambda)(A) & = & \{g \in G(A), \lambda(t)g\lambda(t)^{-1} \text{ has a limit as } t\to 0\} \\ U_G(\lambda)(A) & = & \{g \in G(A), \lambda(t)g\lambda(t)^{-1} \underset{t \to 0}{\longto} 1 \} \\ Z_G(\lambda)(A) & = & \{g \in G(A), \forall t \in A^{\times}, \lambda(t)g\lambda(t)^{-1} = g \}\end{array}.$$
	\item  A \textit{pseudo-parabolic} subgroup of $G$ is a closed subgroup $P$ of the form $P = P_G(\lambda)\mc R_{u,k}(G)$. If $G$ is quasi-reductive, then we have $P = P_G(\lambda)$ and $\mc R_{us, k}(P) = U_G(\lambda)$.
	\end{itemize}
	\subsubsection*{Affine spaces}
	We denote by $A$ a real affine space with direction $V$.
	\begin{itemize}
	    \item If $F$ is an affine subspace of $A$, we denote by $\vec F$ its direction.
	    \item If $X$ is a subset of $V$, we denote by $X^{\vee} = \{\alpha \in V^*, \forall x \in X, \langle \alpha, x \rangle \ge 0\}$ its dual cone.
	    \item If $X$ is a subset of $V$, we denote by $X^0 = \{\alpha \in V^*, \forall x \in X, \langle \alpha, x \rangle = 0\}$ its annihilator.
	\end{itemize}
    \subsubsection*{Buildings and Bruhat-Tits theory}
    \begin{itemize}
    	\item We call a non-archimedean field any field that is complete with respect to a non-trivial non-archimedean absolute value.
        \item We call a local field any non-archimedean field that is locally compact.
    	\item If $G$ is quasi-reductive over a complete discretely valued field with a perfect residue field $k$, then we denote by $\mc B(G, k)$ its (reduced) Bruhat-Tits building. If $K$ is a finitely ramified extension of $k$, then we will also write $\mc B(G, K)$ for $\mc B(G_K, K)$.
		\item If $S$ is a maximal split torus in $G$, we denote by $A(G, S)$ the apartment of $\mc B(G, k)$ corresponding to $S$. It is an affine space under $V(G, S)$.
		\item If $\Omega \subset A(G, S)$ is a subset of an apartment, we denote by $$\hat P_{\Omega} = \{g \in G(k), \forall x \in \Omega, gx =x \}$$ its pointwise stabiliser.
		\item Recall that, given an affine building $X$ modeled on an affine Coxeter complex $\Sigma$ with direction $\vec{\Sigma}$ and vectorial Weyl group $W$, any $W$-invariant inner product on $\vec \Sigma$ extends to a unique metric on $X$ \cite[§3A]{Brown}. If $X= \mc B(G, k)$ is the Bruhat-Tits building of a group $G$, then the action of $G(k)$ on $X$ is isometric. The topology determined by this metric does not depend on the specific inner product on $\vec \Sigma$. All affine buildings in the text will be endowed with this topology and all metrics on buildings will be assumed to be of the type described above.
    \end{itemize}

\subsection*{Acknowledgements}

We thank Bertrand Rémy for his careful reading and comments, as well as Guy Rousseau for discussing with us on the functoriality of Bruhat-Tits buildings.
This work was partially supported by the project \textit{Geometric invariants of discrete and locally compact groups} of the \textit{DFG} (priority programme SPP 2026).
\section{Functoriality for Bruhat-Tits buildings} \label{sect:Functoriality}

In this section, we give an account of functoriality of Bruhat-Tits buildings of quasi-reductive groups with respect to field extensions, following the general outline of \cite{Rou}.

Specifically, our goal is to establish the following two results, which will be sufficient for our purposes:
\begin{enumerate}
	\item Given a quasi-reductive group $G$ over a discretely valued field $(k, \omega)$ with perfect residue field and given a discretely valued Galois extension $K/k$, there exists an action of $\Gal(K/k)$ on $\mc B(G, K)$ via simplicial automorphisms.
	\item There exists a system of equivariant maps $$p_{L/K}: \mc B(G, K) \to \mc B(G, L)$$ indexed by towers of finite separable extensions $k \subset K \subset L$ such that:
    \begin{enumerate}[(i)]
        \item The restriction of $p_{L/K}$ to any apartment $A(G_K, S) \subset \mc B(G, K)$ is an affine isomorphism onto an apartment $A(G_L, T)$ for some maximal split torus $T \subset G_L$ containing $S_L$.
        \item For any $k \subset K \subset L \subset M$, we have $$p_{M/K} = p_{M/L} \circ p_{L/K} .$$
        \item Whenever $L/K$ is Galois, there is an inclusion $$p_{L/K}(\mc B(G, K)) \subset (\mc B(G, L))^{\Gal(L/K)}.$$
    \end{enumerate}
\end{enumerate}

In order to prove the general theorem, we shall first need to prove some special cases and variants of it. The proofs will vary in length and complexity, but will always follow the same template.

\begin{enumerate}
    \item Observe that, under the assumption that $p_{L/K}$ is equivariant and satisfies (i), the linear part of the restriction of $p_{L/K}$ to an apartment $A(G_K, S)$ will be the map induced by base change on cocharacters.
    \item The assumptions of equivariance and affineness then reduce the proof of uniqueness to checking that a given special point $x$ can only have one image.
    \item Use characteristics of the special case at hand to exhibit the image $y$ of $x$. Usually, this will involve studying the action of $G(K)_x$ on $\mc B(G, L)$ and its fixed points.
    \item Prove that the affine map $A(G_K, S) \to A(G_L, T)$ mapping $x$ to $y$ with the linear part exhibited above extends to a $G(K)$-equivariant map. This will sometimes be proved explicitly, and sometimes be dealt with using Bruhat and Tits' theory of descent for valuated root data \cite[§9]{BT1}, which aims specifically at producing such statements.
\end{enumerate}

We fix a quasi-reductive group $G$ over a discretely valued field $(k, \omega)$ with perfect residue field, normalised so that $\omega(k^{\times}) = \ZZ$. We fix, once and for all, an algebraic closure $\ol k$ of $k$ which will contain all the extensions to be discussed below. Each algebraic extension $K/k$ will be endowed with the unique valuation that extends that of $k$.

\subsection{Galois action on the building}

We start off by proving that, for any Galois extension $K/k$ of discretely valued fields with perfect residue fields (eg. $K/k$ is finite or unramified), the Galois group $\Gal(K/k)$ acts on the building $\mc B(G, K)$.

\begin{thm}\label{Functoriality_iso}
	Let $K$ and $L$ be two discretely valued separable extensions of $k$ with perfect residue fields and $\sigma: K \to L$ be a morphism of $k$-algebras. Assume that $G_K$ and $G_L$ have the same relative rank. Then, there exists a unique map $$p_{\sigma}: \mc B(G, K) \to \mc B(G, L)$$ such that:
	\begin{enumerate}
		\item For each maximal $K$-split torus $S \subset G_K$, the map $p_{\sigma}$ restricts to an affine isomorphism of $A(G_K, S)$ onto $A(G_L, S^{\sigma})$.
		\item For each $x \in \mc B(G, K)$ and $g \in G(K)$, we have $p_{\sigma}(gx) = \sigma(g)p_{\sigma}(x)$.
	\end{enumerate}
	Moreover, affine roots in $A(G_K, S)$ are mapped to affine roots in $A(G_L, S^{\sigma})$.
	Lastly, if $K, L, M$ are three extensions of $k$ as above and $\sigma: K \to L$ and $\tau: L \to M$ are two $k$-morphisms, then, we have $p_{\tau \circ \sigma} = p_{\tau} \circ p_{\sigma}$.
\end{thm}

\begin{proof}
	We begin with the proof of uniqueness. The key point is the observation that the compatibility condition on $p_{\sigma}$ with respect to the $G(K)$-action together with the condition that $p_{\sigma}$ be affine on apartments imply that $p_{\sigma}$ is determined by the image of single special vertex in $\mc B(G, K)$. The uniqueness of the image of the latter then follows from condition 2.
	
	\textit{(Uniqueness)} Let a map $p_{\sigma}$ be given satisfying conditions 1 and 2. Fix a maximal $K$-split torus $S \subset G_K$. Let $x$ be a special vertex in $A(G_K, S)$. Because the Weyl group $W(G_K, S)(K)$ acts without nonzero fixed vectors on the direction $V(G_K, S)$, the point $x$ is the unique fixed point for the action of the stabiliser $N_{G_K}(S)(K)_x$ on $A(G_K, S)$. Moreover, because $N_{G_K}(S)(K)_x$ acts on $A(G_K, S)$ via its bounded quotient $N_{G_K}(S)(K)_x/Z(G)(K)$, condition 2 implies that its image $\sigma(N_{G_K}(S)(K)_x)$ act on $A(G_L, S^{\sigma})$ via its bounded quotient $\sigma(N_{G_K}(S)(K)_x)/Z(G)(L)$. Therefore, there exists $x'\in A(G_L, S^{\sigma})$ that is fixed under $\sigma(N_{G_K}(S)(K)_x)$. Lastly, because the image of the subgroup $N_{G_K}(S)(K)_x$ in the linear group of $V(G_K, S)$ is the full Weyl group $W(G_K, S)$, condition 2 implies that the image of $\sigma(N_{G_K}(S)(K)_{x})$ in $GL(V(G_L, S^{\sigma}))$ is the full Weyl group $W(G_L, S^{\sigma})(L)$ and therefore acts on $V(G_L, S^{\sigma})$ without fixing nonzero vectors. The point $x'$ is thus a special vertex and the unique fixed point of the apartment $A(G_L, S^{\sigma})$ under the action of $\sigma(N_{G_K}(S)(K)_x)$.
	
	We now characterise the linear part of the restriction of $p_{\sigma}$ to the apartment $A(G_K, S)$. Let $\vpi$ be a uniformiser of the valuation ring $k^{\circ}$. For each cocharacter $\lambda \in \Phi(G_K, S)^\vee$, the element $s = \lambda(\vpi^{-1}) \in S(K)$ acts on $A(G_K, S)$ by translation by the vector $\nu(s) \in V(G_K, S)$ satisfying $$\forall a\in \Phi(G_K, S), \langle a, \nu(s) \rangle = -\omega(\vpi^{-\langle a, \lambda \rangle}) = \langle a, \lambda \rangle.$$ In other words, the vector $\nu(s)$ is the image of $\lambda$ in $V(G_K, S)$.
	By assumption, we have $$p_{\sigma}(s \cdot x) = \sigma(s) \cdot x' = \sigma(\lambda)(\vpi) \cdot x',$$
	denoting by $\sigma: X_*(S) \overset{\sim}{\to} X_*(S^{\sigma})$ the isomorphism given by pulling back cocharacters along the map $\sigma^*: \Spec L \to \Spec K.$ This map extends $\RR$-linearly and passes to the quotient to a linear isomorphism $$\sigma: V(G_K, S) \to V(G_L, S^{\sigma}).$$
	The same argument as before then implies that $\sigma(s)$ acts on $A(G_L, S^{\sigma})$ by translation by $\sigma(\lambda)$.
	In other words, for each cocharacter $\lambda \in X_*(S)$, we have $$p_{\sigma}(x+\lambda) = x' + \sigma(\lambda).$$ We deduce from condition 1 and the fact that the images of the cocharacters span $V(G_K, S)$ that, for each $v \in V(G_K, S)$, we have $$p_{\sigma}(x+v) = x' + \sigma(v).$$
	The linear part of $p_{\sigma}$ is therefore the isomorphism $\sigma: V(G_K, S) \to V(G_L, S^{\sigma})$ induced by pulling back cocharacters.
	
	We deduce from the above considerations the uniqueness of the map $p_{\sigma}$. Indeed, if $p_{\sigma}$ and $p'_{\sigma}$ satisfy conditions 1 and 2, then they coincide over $A(G_K, S)$. Because $\mc B(G, K) = G(K) \cdot A(G_K, S)$, condition 2 implies that $p_{\sigma}$ and $p'_{\sigma}$ coincide everywhere.

	\textit{(Existence)} We now prove that the conditions brought forth above do define a map from $\mc B(G, K)$ to $\mc B(G, L)$. The idea is to check that the map $$\begin{array}{rcl} G(K) \times A(G_K, S) & \longto & G(L) \times A(G_L, S^{\sigma}) \\ (g, x+v) & \longmapsto & (\sigma(g), x'+\sigma(v))\end{array}$$ passes to the quotient, which amounts to proving that, for each $z \in A(G_K, S)$, the stabliser $\hat P_z$ is mapped to $\hat P_{p_{\sigma}(z)}$ by $\sigma$.
	Fix an apartment $A(G_K, S)$ as before and a special vertex $x \in A(G_K, S)$. Denote by $x'$ the unique fixed point of $A(G_L, S^{\sigma})$ under the subgroup $\sigma(N_{G_K}(S)(K)_x)$. We define $p_{\sigma}$ on $A(G_K, S)$ as the affine isomorphism $$p_{\sigma}: \begin{array}{rcl} A(G_K, S) & \longto & A(G_L, S^{\sigma})\\x+v & \longmapsto & x'+\sigma(v)\end{array}.$$
	Then, for each $n \in N_{G_K}(S)(K)$ and each $z = x + v \in A(G_K, S)$, we have \begin{equation} 
	p_{\sigma}(nz) = \sigma(n)p_{\sigma}(z). \label{eq: equivariance_psigma}
	\end{equation}
	Indeed, by construction, if $n \in N_{G_K}(S)(K)_x$, we have $$p_{\sigma}(nz) = p_{\sigma}(x + nvn^{-1}) = x' + \sigma(n)\sigma(v) \sigma(n)^{-1} = \sigma(n)\cdot (x' + \sigma(v)) = \sigma(n)p_{\sigma}(z).$$
	Moreover, if $n \in Z_{G_K}(S)(K)$, then $nz = x + v + \nu(n)$, where $\nu(n) \in V(G_K, S)$ satisfies $\langle a, \nu(n) \rangle = - \omega(a(n))$ for each $a \in \Phi$. Then, we have $$p_{\sigma}(nz) = p_{\sigma}(x+v+\nu(n)) = x' + \sigma(v) + \sigma(\nu(n)).$$
	Now, by a previous observation, we have $$\langle \sigma(a), \sigma(\nu(n)) \rangle = \langle a, \nu(n)\rangle = -\omega(a(n)) = - \omega(\sigma(a)(\sigma(n))) $$ for each $a \in \Phi$, hence $\sigma(\nu(n)) = \nu(\sigma(n))$ and finally $$p_{\sigma}(nz) = \sigma(n)(x'+\sigma(v)) = \sigma(n) p_{\sigma}(z).$$
	Because $N_{G_K}(S)(K)$ is spanned by $N_{G_K}(S)(K)_x$ and $Z_{G_K}(S)(K)$, we get the result we claimed.
	
	To conclude, we now only have to prove that, denoting by $$\varphi^x = (\varphi^x_a)_{a \in \Phi} \text{ (resp. } \varphi^{x'} = (\varphi^{x'}_a)_{a \in \Phi^\sigma}\text{)}$$ the valuation of the root data $$(Z_{G_K}(S)(K), (U_a(K))_{a \in \Phi}) \text{ (resp. } (Z_{G_K}(S^\sigma)(K), (U_{\sigma(a)}(K))_{a\in \Phi} ) \text{ )},$$ associated to $x$ (resp. $x'$) we have for each $a \in \Phi$ and $u \in U_a(K) \setminus \{1\}$ $$\varphi_a^x(u) = \varphi^{x'}_{\sigma(a)}(\sigma(u)).$$
	Let $a \in \Phi$ and $u \in U_a(K) \setminus \{1\}$. Recall from \cite[6.2.12 b)]{BT1} that, if $(u', u'') \in U_{-a}(K)^2$ is the unique pair \cite[6.1.2 (2)]{BT1} such that $m(u) = u'uu'' \in N_{G_K}(S)(K)$, then $\varphi_a^x(u)$ is the only $k \in \RR$ such that the fixed locus of $m(u)$ in $A(G_K, S)$ is the hyperplane $$\partial \alpha_{a, k}^x = \{z \in A(G_K, S), a(z-x) + k = 0\}.$$
	Now observe that $$\sigma(m(u)) = \sigma(u') \sigma(u) \sigma(u'') \in N_{G_L}(S^{\sigma})(L).$$
	Because $\sigma(u)$ lies in $U_{\sigma(a)}(L) \setminus\{1\}$ whereas $\sigma(u')$ and $\sigma(u'')$ lie in $U_{-\sigma(a)}(L)$, the aforementioned uniqueness property implies $$\sigma(m(u)) = m(\sigma(u)).$$
	Finally, recall that the equivariance property (\ref{eq: equivariance_psigma}) implies that the fixed locus of $m(\sigma(u))$ is the image under $p_{\sigma}$ of the fixed locus of $m(u)$, that is the hyperplane $$\partial \alpha _{\sigma(a), \varphi_a^x(u)}^{x'} = \{z \in A(G_L, S^\sigma), \sigma(a)(z-x') + \varphi_a^x(u) = 0\}.$$ In other words, we have $$\varphi_{\sigma(a)}^{x'}(\sigma(u)) = \varphi_a^x(u).$$
	Note that this proves that $p_{\sigma}$ maps the affine roots of $A(G_K, S)$ to those of $A(G_L, S^{\sigma})$.\\ 	
	We deduce that, for any bounded subset $\Omega \subset A(G_K, S)$, we have $$\sigma(U_{\Omega}) = U_{p_\sigma(\Omega)}.$$ Moreover, by (\ref{eq: equivariance_psigma}), we have $$\sigma(\hat N_{\Omega}) = \hat N_{\sigma(\Omega)}, $$ where $\hat N_{\Omega} = \{n \in N_{G_K}(S)(K), \forall x \in \Omega, n\cdot x = x\}$ is the pointwise stabiliser of $\Omega$. Because $\hat P_{\Omega} = \hat N_{\Omega} U_{\Omega}$, by \cite[7.1.8, 7.4.4]{BT1}, we conclude that $$\sigma(\hat P_{\Omega}) = \hat P_{p_\sigma(\Omega)}.$$
	Consequently, for all pairs $(g,x), (h,y) \in G(K) \times A(G_K, S)$, there exists $n \in N_{G_K}(S)(K)$ such that $nx = y$ and $n^{-1}h^{-1}g \in \hat P_x$ if and only if there exists $n^\sigma \in N_{G_L}(S^{\sigma})(L)$ such that $n^\sigma p_{\sigma}(x) = p_{\sigma}(y)$ and $(n^\sigma)^{-1}\sigma(h^{-1})\sigma(g) \in \hat P_{p_\sigma(x)}$. In other words, the injection $$\begin{array}{rcl} G(K) \times A(G_K, S) & \longto & G(L) \times A(G_L, S^{\sigma}) \\ (g, x) & \longmapsto & (\sigma(g), p_{\sigma}(x))\end{array}$$
	descends to an injection from $\mc B(G, K)$ to $\mc B(G, L)$, which we still denote by $p_{\sigma}$, such that, for each $g \in G(K)$ and $x \in \mc B(G, K)$, we have $$p_{\sigma}(gx) = \sigma(g)p_{\sigma}(x).$$
\end{proof}

Note that the proof of Theorem \ref{Functoriality_iso} yields the following relation between valuated root data.

\begin{prop}\label{Galois_action_valuated_root_data}
	Let $K$ and $L$ be two discretely valued separable extensions of $k$ with perfect residue fields and $\sigma: K \to L$ be a morphism of $k$-algebras such that $G_K$ and $G_L$ have the same relative rank.
	Let $S \subset G_K$ be a maximal $K$-split torus and $x$ be a point in the apartment $A(G_K, S)$. Then, for each $a \in \Phi(G_K, S)$, we have $$\varphi_a^x \circ \sigma^{-1} = \varphi_{\sigma(a)}^{p_{\sigma(x)}}.$$
\end{prop}

We also deduce the following property relating the polysimplicial structures of $\mc B(G, K)$ and $\mc B(G, L)$ under the assumptions of Theorem \ref{Functoriality_iso}.

\begin{prop}\label{prop:pseudo_split_vertices}
	Let $K$ and $L$ be two discretely valued separable extensions of $k$ with perfect residue fields and $\sigma: K \to L$ be a morphism of $k$-algebras such that $G_K$ and $G_L$ have the same relative rank (eg. $G_K$ is pseudo-split).
	Let $S \subset G_K$ be a maximal $K$-split torus. Then, the embedding $p_{\sigma}$ maps vertices (resp. special vertices) in $\mc B(G, K)$ to vertices (resp. special vertices) in $\mc B(G, L)$.
\end{prop}

\begin{proof}
	This statement follows directly from the fact that $p_{\sigma}$ maps affine roots in $A(G_K, S)$ to affine roots in $A(G_L, S^{\sigma})$ and the equality between the relative ranks of $G_K$ and $G_L$.
\end{proof}

\begin{rmk}
	Note that, in general, the image of a higher-dimensional facet in $A(G_K, S)$ will not be a facet in $A(G_L, S^{\sigma})$.
\end{rmk}

Finally, Theorem \ref{Functoriality_iso} also immediately yields the existence of a Galois action on the Bruhat-Tits building.

\begin{thm}\label{Galois_action_building}
	Let $K/k$ be a Galois extension of discretely valued fields with perfect residue fields and $\sigma \in \Gal(K/k)$. Then, there exists a unique map $$p_{\sigma}: \mc B(G, K) \to \mc B(G, K)$$ such that: 
	\begin{enumerate}
		\item For each maximal $K$-split torus $S \subset G_K$, the map $p_{\sigma}$ restricts to an affine isomorphism of $A(G_K, T)$ onto $A(G_K, S^{\sigma})$ in such a way that affine roots in $A(G_K, S)$ are mapped to affine roots in $A(G_K, S^{\sigma})$. In particular, the map $p_{\sigma}$ is simplicial.
		\item For each $x \in \mc B(G, K)$ and $g \in G(K)$, we have $p_{\sigma}(gx) = \sigma(g)p_{\sigma}(x)$.
	\end{enumerate}
	
	Moreover, for each $\sigma, \tau \in \Gal(K/k)$, we have $p_{\sigma \tau} = p_{\sigma} \circ p_{\tau}$.
\end{thm}

\subsection{Functorial system of embeddings}

Our second, more substantial, task is to establish the existence of a canonical system of embeddings indexed by discretely valued separable extensions of $k$ with perfect residue fields.

We first prove the existence of canonical embeddings $p_{L/K}$ for two classes of extensions, mirroring the two steps of the general construction of the Bruhat-Tits building. First, we treat the case where $L$ is the strict Henselisation of $K$, then we treat the case where $G_K$ is quasi-split and $G_L$ is pseudo-split. 

Both proofs rely on descent theory for valuations of root data, as exposed in \cite[§9]{BT1}, which provides a general framework for constructing equivariant embeddings of buildings.

\subsubsection*{The case of the strict henselisation}

The first case that we treat is that of a strict Henselisation $K^{\nr}/K$. The situation is as nice as it could be, in the sense that the canonical embedding $\mc B(G, K) \to \mc B(G, K^{\nr})$ is a simplicial map and identifies $\mc B(G, K)$ with the fixed locus of $\mc B(G, K^{\nr})$ under the action of $\Gal(K^{\nr}/K)$.

The proof is a direct consequence of the construction of Bruhat-Tits buildings which relies on unramified descent in the general case \cite[§5]{BT2}, \cite[Section 4.4]{Lou}.

\begin{prop}\label{strict_henselization_embedding}
	Let $K$ be a discretely valued separable extension of $k$ with perfect residue field and $K^{\nr}$ be its strict henselisation. There exists a unique map $$p_{K^{\nr}/K}: \mc B(G, K) \to \mc B(G, K^{\nr})$$ such that
	\begin{enumerate}
		\item For each maximal $K$-split torus $S \subset G_K$, there exists a maximal $K^{\nr}$-split torus $S^{\nr} \subset G_{K^{\nr}}$ such that $S_{K^{\nr}} \subset S^{\nr}$ and such that the map $p_{K^{\nr}/K}$ restricts to an affine isomorphism of $A(G_K, S)$ onto an affine subspace of $A(G_{K^{\nr}}, S^{\nr})$.
		\item The map $p_{K^{\nr}/K}$ is $G(K)$-equivariant.
		\item We have the inclusion $p_{K^{\nr}/K}(\mc B(G, K)) \subset \mc B(G, K^{\nr})^{\Gal(K^{\nr}/K)}$.
	\end{enumerate}
	Moreover, for each maximal $K$-split torus $S \subset G_K$ and each maximal $K^{\nr}$-split torus $S^{\nr} \subset G_{K^{\nr}}$ such that $S_{K^{\nr}} \subset S^{\nr}$, the affine roots in $A(G_K, S)$ correspond to the restrictions of the affine roots in $A(G_{K^{\nr}}, S^{\nr})$, and we have the equality $p_{{K^{\nr}}/K}(\mc B(G, K)) = \mc B(G, {K^{\nr}})^{\Gal({K^{\nr}}/K)}$.
\end{prop}

\begin{proof}
	Let $S$ be a maximal $K$-split torus in $G_K$. Let $\Gamma$ be the Galois group $\Gal(K^{\nr}/K)$. By \cite[Lemme 4.8]{Lou}, there exists a maximal $K^{\nr}$-split torus $S^{\nr}$ in $G_{K^{\nr}}$ that contains $S_{K^{\nr}}$ and is defined over $K$. Denote by $\Phi = \Phi(G_K, S)$ and $\Phi^{\nr} = \Phi(G_{K^{\nr}}, S^{\nr})$ the corresponding root systems. The apartment $A(G_{K^{\nr}}, S^{\nr})$ is then stable under the action of $\Gamma$ and admits a fixed point $y$ corresponding to a $\Gamma$-invariant valuation $\varphi^{\nr}$ of the root group datum $(Z_{G_{K^{\nr}}}(S^{\nr})(K^{\nr}), (U_{a^{\nr}}(K^{\nr}))_{a^{\nr} \in \Phi^{\nr}})$. 
	It then follows from \cite[Théorème 4.9]{Lou} that the valuation $\varphi^{\nr}$ descends to a valuation $\varphi$ of the root group datum $(Z_{G_K}(S)(K), (U_a(K))_{a\in \Phi})$, corresponding to a point $x \in A(G_K, S)$. By \cite[9.1.17]{BT1}, there is then a $G(K)$-equivariant isometry $$p_{K^{\nr}/K}: \mc B(G, K) \to \mc B(G, K^{\nr})$$ such that $p_{K^{\nr}/K}(x) = y$. We then have $p_{K^{\nr}/K}(A(G_K, S)) = A(G_{K^{\nr}}, S^{\nr})^{\Gamma}$ and, by $G(K)$-equivariance, the inclusion $p_{K^{\nr}/K}(\mc B(G, K)) \subset \mc B(G, K^{\nr})^{\Gamma}$, which completes the proof of the existence statement. By \cite[9.2.14 and 9.2.15]{BT1}, the map $p_{K^{\nr}/K}$ is in fact the unique $G(K)$-equivariant isometry from $\mc B(G, K)$ to $\mc B(G, K^{\nr})$ that maps into $\mc B(G, K^{\nr})^{\Gamma}$, which establishes the uniqueness statement.
	
	The additional statements follow directly from \cite[Théorème 4.9]{Lou}.
\end{proof}

For technical reasons to be detailed below, we also establish a similar result for embedded buildings -- or inner façades in the terminology of \cite{Rou2} -- associated to centralisers of tori.

Recall that, if $S$ is a $K$-split torus in $G_K$, then the union $\mc B(Z_{G_K}(S), G_K, K)$ of the apartments $A(G_K, T)$, where $T$ ranges among the maximal $K$-split tori of $G_K$ that contain $S$, is closely related to the Bruhat-Tits building $\mc B(Z_{G_K}(S), K)$ (see the Appendix for precise statements). In particular, Proposition \ref{strict_henselization_embedding} implies the following result.

\begin{prop}\label{strict_henselisation_embedding_levi}
	Let $K$ be a discretely valued extension of $k$ with perfect residue field and $S$ be a $K$-split torus in $G_K$. Then, the action of $\Gal(K^{\nr}/K)$ on $\mc B(G, K^{\nr})$ stabilises the subspace $\mc B(Z_{G_K}(S), G, K^{\nr})$ and we have $$(\mc B(Z_{G_K}(S), G, K^{\nr}))^{\Gal(K^{\nr}/K)} = p_{K^{\nr}/K}(\mc B(Z_{G_K}(S), G, K))$$
\end{prop}

\begin{proof}
	First of all, note that Conditions 1 and 3 of Proposition \ref{strict_henselization_embedding} imply that $$p_{K^{\nr}/K}(\mc B(Z_{G_K}(S), G, K)) \subset \mc B(Z_{G_K}(S), G, K^{\nr})^{\Gal(K^{\nr}/K)}.$$
	For the reverse inclusion, we apply Proposition \ref{strict_henselization_embedding} to the group $Z_{G_K}(S)$. To pass from the façade $\mc B(Z_{G_K}(S), G, K^{\nr})$ to the building $\mc B(Z_{G_K}(S), K^{\nr})$, we need to mod out, which we do below.
	Let $T$ be a maximal $K$-split torus of $G_K$ containing $S$. By \cite[Lemme 4.8]{Lou}, we have at our disposal a maximal $K^{\nr}$-split torus $T^{\nr}$ of $G_{K^{\nr}}$ containing $T_{K^{\nr}}$ that is defined over $K$, and we have $$p_{K^{\nr}/K}(A(G_K, T)) = A(G_{K^{\nr}}, T^{\nr})^{\Gal(K^{\nr}/K)}.$$ Consider the vector spaces $$\begin{array}{ccc}
	V^{\nr} & = & \ker(V(G_{K^{\nr}}, T^{\nr}) \to V((Z_{G_K}(S))_{K^{\nr}}, T^{\nr})) \\
	V & = & \ker(V(G_K, T) \to V(Z_{G_K}(S), T))
	\end{array}.$$
	The map $X_*(T) \to X_*(T^{\nr})$ given by extension of scalars induces a linear injection $\begin{array}{ccc}V & \hookrightarrow & V^{\nr}\\v & \mapsto & v_{K^{\nr}}\end{array}$ that identifies $V$ with the set of $\Gal(K^{\nr}/K)$-fixed points of $V^{\nr}$. Indeed, the proof of Theorem \ref{Functoriality_iso} establishes that the action of $\Gal(K^{\nr}/K)$ on $V(G_{K^{\nr}}, T^{\nr})$ is induced by the pullback action on cocharacters.
	
	Then, we claim that the restriction of $p_{K^{\nr}/K}$ to $\mc B(Z_{G_K}(S), G, K)$ is $V$-equivariant, ie. that for each $x \in \mc B(Z_{G_K}(S), G, K)$ and each $v \in V$, we have $$p_{K^{\nr}/K}(x+v) = p_{K^{\nr}/K}(x) + v_{K^{\nr}}.$$
	Because the action of $V$ commutes with that of $Z_{G_K}(S)(K)$ (Property 2 of the Appendix), it is sufficient to prove the above for $x \in A(G_K, T)$. 
	
	Recall that the centraliser $Z_{G_K}(T)(K)$ acts on $A(G_K, T)$ by translation via the map $$\nu: Z_{G_K}(T)(K) \longto V(G_K, 
	T) $$ such that, for each $z \in Z_{G_K}(T)(K)$ and $\alpha \in \Phi(G_K, T)$, we have $$\langle \alpha, \nu(z) \rangle = - \omega(\alpha(z)).$$
	Likewise, we have an analogously defined map $$\nu^{\nr}: Z_{G_{K^{\nr}}}(T^{\nr})(K^{\nr}) \longto V(G_{K^{\nr}}, T^{\nr}).$$
	In particular, the center $Z(Z_{G_K}(S))(K)$ acts on $A(G_K, T)$ by translations by elements of $V$ and on $A(G_{K^{\nr}}, T^{\nr})$ by elements of $V^{\nr}$ and we have, for each $z \in Z(Z_{G_K}(S))(K)$ $$\nu^{\nr}(z) = \nu(z)_{K^{\nr}}.$$ Because $p_{K^{\nr}/K}$ is $G(K)$-equivariant, we have $$p_{K^{\nr}/K}(x + \nu(z)) = p_{K^{\nr}/K}(x) + \nu(z)_{K^{\nr}}$$ for each $x \in A(G_K, T)$ and $z \in Z(Z_{G_K}(S))(K)$. Because both sides of the equation are affine and because the convex hull of $\{\nu(z), z \in  Z(Z_{G_K}(S))(K)\}$ is $V$, we get the result.
	
	Now let $$\tilde \pi^{\nr}: \mc B(Z_{G_K}(S), G, K^{\nr}) \to \mc B(Z_{G_K}(S), K^{\nr})$$ be the essentialisation map described in the Appendix. Because it is invariant under the action of $V^{\nr}$, the map $\tilde \pi^{\nr} \circ p_{K^{\nr}/K}$ is invariant under the action of $V$. Therefore, there exists a unique map $j$ that makes the following diagram commute 
	\begin{center}
		\begin{tikzcd}
			\mc B(Z_{G_K}(S), G, K^{\nr}) \arrow[r, "\tilde \pi^{\nr}"] & \mc B(Z_{G_K}(S), K^{\nr}) \\
			\mc B(Z_{G_K}(S), G, K) \arrow[u, "p_{K^{\nr}/K}"] \arrow[r, "\tilde \pi"] & \mc B(Z_{G_K}(S), K) \arrow[u, dotted, "\exists! j"]
		\end{tikzcd}
	\end{center}
	Moreover:
	\begin{enumerate}
		\item The restriction of $j$ to $A(Z_{G_K}(S), T)$ fits into the following commutative diagram
		\begin{center}
			\begin{tikzcd}
				A(G_{K^{\nr}}, T^{\nr}) \arrow[r, twoheadrightarrow, "\tilde \pi^{\nr}"] & A((Z_{G_K}(S))_{K^{\nr}}, T^{\nr}) \\
				A(G_K, T) \arrow[u, hookrightarrow, "p_{K^{\nr}/K}"] \arrow[r, twoheadrightarrow, "\tilde \pi"] & A(Z_{G_K}(S), T) \arrow[u, "j"]
			\end{tikzcd}.
		\end{center}
		In particular, the restriction of $j$ to $A(Z_{G_K}(S), T)$ is an affine injection.
		\item The map $j$ is $G(K)$-equivariant. In particular, the restriction of $j$ to each apartment $A(G_K, T')$ is an affine injection into an apartment $A(G_{K^{\nr}}, (T')^{\nr})$ for some maximal $K^{\nr}$-split torus $(T')^{\nr}$ containing $(T')_{K^{\nr}}$.
		\item Because $\tilde \pi^{\nr}$ is $\Gal(K^{\nr}/K)$-equivariant (Proposition \ref{Galois_action_inner_facade}), we have $$\tilde \pi^{\nr}(p_{K^{\nr}/K}(\mc B(Z_{G_K}(S), G, K))) \subset \mc B(Z_{G_K}(S), K^{\nr})^{\Gal(K^{\nr}/K)}$$ and therefore $$j(\mc B(Z_{G_K}(S), K)) \subset \mc B(Z_{G_K}(S), K^{\nr})^{\Gal(K^{\nr}/K)}.$$
	\end{enumerate}
	In other words, the map $j$ is none other than the map $$p^1_{K^{\nr}/K}: \mc B(Z_{G_K}(S), K) \to \mc B(Z_{G_K}(S), K^{\nr})$$ given by Proposition \ref{strict_henselization_embedding}.
	
	Now, let $y^{\nr} \in \mc B(Z_{G_K}(S), G, K^{\nr})^{\Gal(K^{\nr}/K)}$. Because $\tilde \pi^{\nr}$ is $\Gal(K^{\nr}/K)$-equivariant, we have $$\tilde \pi^{\nr}(y^{\nr}) \in \mc B(Z_{G_K}(S), K^{\nr})^{\Gal(K^{\nr}/K)}.$$ By Proposition \ref{strict_henselization_embedding}, there exists $x \in \mc B(Z_{G_K}(S), K)$ such that $$\tilde \pi^{\nr}(y^{\nr}) = p^1_{K^{\nr}/K}(x).$$ Because $\tilde \pi$ is onto, there exists $z\in \mc B(Z_{G_K}(S), G, K)$ such that $\tilde \pi(z) = x$. Then, we have $$\tilde \pi^{\nr}(p_{K^{\nr}/K}(z)) = p^1_{K^{\nr}/K}(\tilde \pi(z)) = \tilde \pi^{\nr}(y^{\nr}).$$
	Therefore, there exists $v^{\nr} \in V^{\nr}$ such that $y^{\nr} = p_{K^{\nr}/K}(z) + v^{\nr}$. Because both $y^{\nr}$ and $p^1_{K^{\nr}/K}(z)$ are fixed under the action of the Galois group, so must $v^{\nr}$ be (Proposition \ref{Galois_action_inner_facade}, 3). Consequently, $v^{\nr}$ can be written $v^{\nr} = v_{K^{\nr}}$ for some $v \in V$ and we thus have $$y^{\nr} = p_{K^{\nr}/K}(z + v) \in p_{K^{\nr}/K}(\mc B(Z_{G_K}(S), G, K)),$$ which completes the proof.
\end{proof}

\subsubsection*{The quasi-split to pseudo-split case}

The second special case is that of a finite extension $L/K$ such that $G_K$ is quasi-split and $G_L$ is pseudo-split. The main difficulty here is to exhibit the image of a given special point. Because this image will be constructed as a valuation of a root data, which will involve concrete manipulations of parameterisations of root groups, this proof will be noticeably longer than the other ones.

\begin{prop}\label{quasisplit_to_split_embedding}
	Let $K$ and $L$ be finite separable extensions of $k$, with $K \subset L$. Assume that $L/K$ is Galois, that $G_K$ is quasi-split and that $G_L$ is pseudo-split. Then, there exists a unique map $$p_{L/K}: \mc B(G, K) \to \mc B(G, L)$$ such that 
	\begin{enumerate}
		\item For each maximal $K$-split torus $S \subset G_K$, there exists a maximal $L$-split torus $T \subset G_L$ such that $S_L \subset T$ and such that the map $p_{L/K}$ restricts to an affine isomorphism of $A(G, S)$ onto an affine subspace of $A(G_L, T)$.
		\item The map $p_{L/K}$ is $G(K)$-equivariant.
		\item We have the inclusion $p_{L/K}(\mc B(G, K)) \subset \mc B(G, L)^{\Gal(L/K)}$.
	\end{enumerate}
\end{prop}

\begin{lem}\label{lem:maximal_tori_quasisplit}
	With the notations of Proposition \ref{quasisplit_to_split_embedding}, if $S \subset G_K$ is a maximal $K$-split torus, then there exists a unique maximal torus $T \subset G_K$ that contains $S$. The torus $T$ splits over $L$ and the set of fixed points of the apartment $A(G_L, T_L)$ under the induced action of $\Gal(L/K)$ is an affine subspace with direction $X_*(S_L)_{\RR}/X_*(Z(G_L))_{\RR}$. Moreover, we have $Z_{G_K}(S) = Z_{G_K}(T)$ as well as $N_{G_K}(S) \subset N_{G_K}(T)$.
\end{lem}

\begin{proof}
	Because $G_K$ is quasi-split, the centraliser $Z_{G_K}(S)$ is nilpotent and therefore contains a unique maximal torus $T$ that lies in the center $Z(Z_{G_K}(S))$ \cite[Theorem 16.47]{Milne}. It immediately follows that $Z_{G_K}(S) = Z_{G_K}(T)$. Moreover, for each $K$-algebra $R$, if $n \in N_{G_K}(S)(R)$, then $n$ normalises the subgroup $Z_{G_R}(S_R)$ and, because $T_R$ is characteristic in $Z_{G_R}(S_R)$, $n$ must also normalise $T_R$. Hence $N_{G_K}(S) \subset N_{G_K}(T)$. 
	
	Applying the same theorem to $Z_{G_L}(S_L) = (Z_{G_K}(S))_L$, we note that $T_L$ is the only maximal torus of $G_L$ that contains $S_L$. But $S_L$ is split and therefore contained in a maximal $L$-split torus of $G_L$ which, because $G_L$ is pseudo-split, must be a maximal torus and can thus only be $T_L$. Hence, the torus $T_L$ must be split. 
	
	Because $T_L$ is defined over $K$, the action of $\Gal(L/K)$ on $\mc B(G, L)$ stabilises its associated apartment $A(G_L, T_L)$ and induces an action by affine automorphisms. In particular, because $\Gal(L/K)$ is compact, the action of $\Gal(L/K)$ admits a fixed point. Setting this point as origin defines an isomorphism of affine spaces between $A(G_L, T_L)$ and $V(G_L, T_L)$ that identifies the action of $\Gal(L/K)$ on $A(G_L, T_L)$ with the action on $V(G_L, T_L)$ induced by the pullback action on cocharacters of $T_L$. Because the fixed locus of the $\Gal(L/K)$-action on $V(G_L, T_L) = X_*(T_L)_{\RR}/X_*(Z(G_L))_{\RR}$ is precisely $X_*(S_L)_{\RR}/X_*(Z(G_L))_{\RR}$, we get the expected result.
\end{proof}

\begin{proof}[Proof of Proposition \ref{quasisplit_to_split_embedding}]
The spirit of the proof is very similar to that of Proposition \ref{strict_henselization_embedding}. As in the previous case, we prove that, letting $\Phi = \Phi(G, S)$ be the relative root system of $G$ and $\tilde \Phi = \Phi(G_L, T_L)$ be its absolute root system, the valuation $(\varphi_{a}^y)_{a\in \tilde \Phi}$ associated to $y$ descends to the valuation $(\varphi_a^x)_{a\in \Phi}$ associated to $x$ in the sense of \cite[9.1.11]{BT1}. The strategy goes as follows: Letting $x$ be a special vertex in $A(G, S)$, we give an explicit construction of a valuation of the root data $(Z(L), (U_a(L))_{a\in \tilde\Phi})$ that descends to $(\varphi_a^x)_{a\in \Phi}$ and prove that it is fixed under the Galois action. Proposition \cite[9.1.17]{BT1} then directly yields a map satisfying 1-3 and Propositions \cite[9.2.14 and 9.2.15]{BT1} then ensure the uniqueness.
	
Let $x$ be a special vertex in $A(G, S)$. We may assume without loss of generality that there exists a Chevalley quasi-system in $G^{\pred}$ along $S$ \cite[Définition 2.4]{Lou}, ie. a family of isomorphisms $(\zeta_a)_{a\in \Phi}$ of Weil restrictions of groups of the form $\SL_2, \SU_3$, or $\BC_1$ onto the universal covers $\tilde G_a$ of the groups of $K$-rank 1 $G_a = \langle U_{a}, U_{-a}\rangle$ associated to each non-divisible root $a \in \Phi_{\nd}$, such that the valuation $(\varphi_a^x)_{a\in \Phi}$ associated to $x$ is given by the formulas in \cite[Définition 3.1]{Lou}.

We construct an associated Chevalley quasi-system in $G_L$ along $T_L$. Note that, under the assumption that $G_L$ is pseudo-split, each $(G_a)_L$ is pseudo-split with $(T \cap G_a)_L$ as a split maximal torus \cite[Corollary A.2.7]{CGP}. Because $G$ is pseudo-split over $K$, there exists a pseudo-parabolic subgroup $B \subset G$ such that $B_L$ is a minimal pseudo-parabolic $L$-subgroup of $G_L$. Denote by $\Delta \subset \Phi(G, S)$ the system of relative simple roots associated to $B$ and by $\tilde \Delta \subset \Phi(G_L, T_L)$ the system of absolute simple roots associated to $B_L$. Recall \cite[Proof of Theorem C.2.15, Step 5]{CGP} that the restriction map $\alpha \mapsto \alpha_{|S}$ establishes a bijection between the set of $\Gal(L/K)$-orbits in $\tilde \Delta$ and the set $\Delta$.
Let $a \in \Delta$ be a relative simple root, let $S_a = (S \cap G_a)^0_{\red}$, and let $\tilde S_a$ be the unique maximal $L$-split torus of $\tilde G_a$ that maps onto $S_a$.

\begin{enumerate}
	\item Assume that the isomorphism $\zeta_a$ has as domain a Weil restriction $R_{K'/K}(\SL_{2, K'})$. The base change $R_{K'/K}(\SL_{2, K'})_L$ decomposes into a direct product $$R_{K'/K}(\SL_{2, K'})_L = R_{K'\otimes_KL/L}(\SL_{2,K'\otimes L}) \simeq \prod_{i=1}^r R_{L'_i/L}(\SL_{2,L'_i}),$$ where $K' \otimes_K L = \prod_{i=1}^r L'_i$ is the canonical decomposition of the finite reduced $L$-algebra $K'\otimes_KL$ into a product of extensions of $L$ and we have a commutative diagram 
	\begin{center}
		\begin{tikzcd}
			\prod_{i=1}^r R_{L'_i/L}(\SL_{2,L'_i}) \arrow[r, "\sim"] & R_{K'\otimes_KL/L}(\SL_{2,K'\otimes_K L}) \arrow[r, "(\zeta_a)_L"] & (\tilde G_a)_L \\
			\diag\left(\begin{pmatrix}* & \\ & *\end{pmatrix}_L\right) \arrow[u, "\subset"] \arrow[r, "\sim"] & \begin{pmatrix} * & \\ & *\end{pmatrix}_L \arrow[u, "\subset"] \arrow[r, "(\zeta_a)_L"] & (\tilde S_a)_L \arrow[u, "\subset"]
		\end{tikzcd}.
	\end{center}
	Here, the notation $\diag\left(\begin{pmatrix}* & \\ & *\end{pmatrix}_L\right)$ refers to the diagonal embedding of the torus $\begin{pmatrix}* & \\ & *\end{pmatrix}_L$ of $\SL_{2,L}$ in the product $\prod_i R_{L'_i/L}(\SL_{2,L'_i})$.
	Because $(G_a)_L$ is pseudo-split, so is $\prod_i R_{L'_i/L}(\SL_{2,L'_i})$, which implies that each $R_{L'_i/L}(\SL_{2,L'_i})$ is pseudo-split and therefore \cite[Proposition A.5.15 (2)]{CGP} that each extension $L'_i/L$ is purely inseparable. It follows that the torus $\prod_i \begin{pmatrix}* & \\ & *\end{pmatrix}_{L}$ is the unique maximal torus of $\prod_i R_{L'_i/L}(\SL_{2,L'_i})$ containing the split torus $\diag\left(\begin{pmatrix}* & \\ & *\end{pmatrix}_L\right)$. The assumption that $\zeta_a$ maps the subgroup $R_{K' \otimes_K L/L}\begin{pmatrix}1 & \mathbb G_{a,K'\otimes_K L} \\ & 1 \end{pmatrix}$ to $U_a$ implies that $a$ pulls back along $\zeta_a$ to the character $\begin{array}{ccc}
	\mathbb G_{m,L} & \to & \mathbb G_{m,L} \\ t & \mapsto & t^2
	\end{array}$. Consequently, pulling back along $(\zeta_a)_L$ establishes a Galois-equivariant one-to-one correspondence between the roots $\alpha \in \tilde \Delta$ such that $\alpha_{|S_L} = a_L$ and the roots of $\prod_{i=1}^r R_{L'_i/L} (\SL_{2, L'_i})$ whose restriction to $\diag \begin{pmatrix}* & \\ & * \end{pmatrix}_L$ is the character $\begin{pmatrix} t & \\ & t^{-1} \end{pmatrix} \mapsto t^2$, namely the maps $\begin{array}{ccc} \mathbb G_{m,L}^r & \to & \mathbb G_{m,L}\\(t_1, \dots, t_r)& \mapsto & t_i^2\end{array}$ for $i \in [\![1, r]\!]$. We relabel the factor fields of $K' \otimes_KL$ accordingly: $K' \otimes_K L = \prod_{\substack{\alpha \in \tilde \Delta\\\alpha_{|S_L} = a_L}} L'_{\alpha}$.
	
	It follows that the factors $R_{L'_\alpha/L}(\SL_{2,L'_\alpha})$ are the subgroups of $L$-rank 1 of $R_{K'\otimes_K L/L} (\SL_{2,K' \otimes_K L})$ associated to the roots relative to $\prod_i \begin{pmatrix}* & \\ & * \end{pmatrix}_L$ whose restriction to $\diag \begin{pmatrix}* & \\ & * \end{pmatrix}_L$ is the character $\begin{pmatrix} t & \\ & t^{-1} \end{pmatrix} \mapsto t^2$. Accordingly, we have a direct product decomposition $$(\tilde {G_a})_L = \prod_{\substack{\alpha \in \tilde \Delta\\\alpha_{|S_L} = a_L}} \tilde{(G_L)_{\alpha}}$$ and the map $(\zeta_a)_L$ decomposes into a product $$(\zeta_a)_L = \prod_{\alpha_{|S_L} = a_L} \zeta_{\alpha},$$ where each $\zeta_{\alpha}: R_{L'_\alpha/L}(\SL_{2, L'_\alpha}) \to \tilde{(G_L)_{\alpha}}$ is an isomorphism that maps the upper triangular matrices to the root group $U_{\alpha}$ and whose restriction to the torus $\begin{pmatrix} * & \\ & * \end{pmatrix}$ is the coroot $\alpha^{\vee}$.
	
	Observe also that, for each $\sigma \in \Gal(L/K)$, the automorphism induced by $\sigma$ on $K' \otimes_K L$ restricts to an isomorphism of valued fields $L'_{\alpha} \overset{\sim}{\to} L'_{\sigma(\alpha)}$ for each $\alpha \in \tilde \Delta$ such that $\alpha_{|S_L} = a_L$. Because $(\zeta_a)_L$ is defined over $K$, we have $$\sigma \circ \zeta_{\alpha} \circ \sigma^{-1} = \zeta_{\sigma( \alpha)}$$ for each $\alpha \in \tilde \Delta$ such that $\alpha_{|S_L} = a_L$.
	
	Consequently, if for each $\alpha \in \tilde \Delta$ such that $\alpha_{|S_L} = a_L$ we define a function $\varphi_{\alpha}: U_{\alpha}(L) \to \RR$ by setting, for each $x \in L'_{\alpha}$,  $$\varphi_{\alpha}\left(\zeta_{\alpha}\begin{pmatrix}1 & x \\ & 1\end{pmatrix}\right) = \omega(x),$$ then we have $$\varphi_{\alpha} \circ \sigma^{-1} = \varphi_{\sigma(\alpha)}$$ for each $\sigma \in \Gal(L/K)$.
	
	Moreover, under the identification $\zeta_a$, the inclusion $U_a(K) \subset U_a(L) = \prod_{\substack{\alpha \in \tilde \Delta\\\alpha_{|S_L} = a_L}} U_{\alpha}(L)$ pulls back to the diagonal map $K' \to K' \otimes_K L = \prod_{\substack{\alpha \in \tilde \Delta\\\alpha_{|S_L} = a_L}} L'_{\alpha}$. Therefore, in the decomposition of an element $u \in U_{a}(K)\setminus\{1\}$ as $u = \prod_{\substack{\alpha \in \tilde \Delta\\\alpha_{|S_L} = a_L}} u_{\alpha}$, with $u_{\alpha} \in U_{\alpha}(L)$, we have $\varphi^x_a(u) = \varphi_{\alpha}(u_{\alpha})$ for each $\alpha \in \tilde \Delta$ such that $\alpha_{|S_L} = a_L$.
	
	\item Next, assume that the isomorphism $\zeta_a$ has as domain a Weil restriction $R_{K'/K}(\BC_{1, K'})$. The exact same reductions as in step 1 apply and we may write $K' \otimes_K L = \prod_{\substack{\alpha \in \tilde \Delta\\\alpha_{|S_L} = a_L}} L'_{\alpha}$ and $$R_{K' \otimes_K L/L} (\BC_{1, K' \otimes_K L}) = \prod_{\substack{\alpha \in \tilde \Delta\\\alpha_{|S_L} = a_L}} R_{L'_{\alpha}/L}(\BC_{1, L'_{\alpha}}).$$ 
	Likewise, we have a direct product decomposition $$(\tilde {G_a})_L = \prod_{\substack{\alpha \in \tilde \Delta\\\alpha_{|S_L} = a_L}} \tilde{(G_L)_{\alpha}}$$ and the map $(\zeta_a)_L$ decomposes into a product $$(\zeta_a)_L = \prod_{\alpha_{|S_L} = a_L} \zeta_{\alpha},$$ where each $\zeta_{\alpha}: R_{L'_\alpha/L}(\BC_{1, L'_\alpha}) \to \tilde{(G_L)_{\alpha}}$ is an isomorphism that maps the "upper triangular unipotent matrices"\footnote{That is, the root group of $R_{L'_{\alpha}/L}(\BC_{1, L'_{\alpha}})$ that lifts the group of upper triangular unipotent matrices in $R_{(L'_{\alpha})^{1/2}/L}(\Sp_{2, (L'_{\alpha})^{1/2}})$} to the root group $U_{\alpha}$ and whose restriction to the torus $\begin{pmatrix} * & \\ & * \end{pmatrix}_L$ is the coroot $\alpha^{\vee}$. Like before, the fact that $(\zeta_a)_L$ is defined over $K$ implies that $$\sigma \circ \zeta_{\alpha} \circ \sigma^{-1} = \zeta_{\sigma(\alpha)}$$ for each $\sigma \in \Gal(L/K)$ and $\alpha \in \tilde \Delta$ such that $\alpha_{|S_L} = a_L$.
	
	Let $\eta \in K'^{1/2} \setminus K'$ such that the function $\varphi_a^x: U_a(K) \to \RR$ is given by $$\varphi_a^x(x_a(u, v)) = \frac{1}{2} \omega(\eta u^2 + v),$$ where $x_a: R_{K'^{1/2}/K}(\mbb G_{a,K'^{1/2}}) \times R_{K'/K}(\mbb G_{a, K'}) \to U_a$ is the restriction of $\zeta_a$ to the subgroup of unipotent upper triangular matrices in $R_{K'/K}(\BC_{1, K'})$. Then, if for each $\alpha \in \tilde \Delta$ such that $\alpha_{|S_L} = a_L$ we define a function $\varphi_{\alpha}: U_{\alpha}(L) \to \RR$ by setting for each $(u, v) \in (L'_{\alpha})^{1/2}\times L'_{\alpha}$ $$\varphi_{\alpha}(x_{\alpha}(u, v)) = \frac{1}{2}\omega(\eta u^2 + v),$$ we have for each such $\alpha$ and each $\sigma \in \Gal(L/K)$ $$\varphi_{\alpha} \circ \sigma^{-1} = \varphi_{\sigma(\alpha)}.$$
	
	Moreover, under the identification $\zeta_a$, the inclusion $U_a(K) \subset U_a(L) = \prod_{\substack{\alpha \in \tilde \Delta\\\alpha_{|S_L} = a_L}} U_{\alpha}(L)$ pulls back to the diagonal map $(K')^{1/2} \times K' \to  \prod_{\substack{\alpha \in \tilde \Delta\\\alpha_{|S_L} = a_L}} (L'_{\alpha})^{1/2} \times L'_{\alpha}$. Therefore, in the decomposition of an element $u \in U_{a}(K)\setminus\{1\}$ as $u = \prod_{\substack{\alpha \in \tilde \Delta\\\alpha_{|S_L} = a_L}} u_{\alpha}$, with $u_{\alpha} \in U_{\alpha}(L)$, we have $\varphi^x_a(u) = \varphi_{\alpha}(u_{\alpha})$ for each $\alpha \in \tilde \Delta$ such that $\alpha_{|S_L} = a_L$.
	
	\item Finally, assume that the isomorphism $\zeta_a$ has as domain a Weil restriction $R_{K'_2/K}(\SU_{3, K'/K'_2})$. Again, we decompose $K'_2 \otimes_KL = \prod_{i\in I_a} L'_i$, where the $L'_i$ are fields and because $R_{K'_2/K}(\SU_{3, K'/K'_2})_L$ is pseudo-split, we have $K' \subset L'_i$ for each $i$, each $L'_i/L$ is purely inseparable, and $$R_{K'_2/K}(\SU_{3, K'/K'_2})_L = \prod_{i\in I_a} R_{L'_i/L}(\SL_{3, L'_i}).$$
	Likewise, we have a commutative diagram 
	\begin{center}
		\begin{tikzcd}
			\prod_{i\in I_a} R_{L'_i/L}(\SL_{3,L'_i}) \arrow[r, "\sim"] & R_{K'\otimes_KL/L}(\SL_{3,K'\otimes L}) \arrow[r, "(\zeta_a)_L"] & (\tilde G_a)_L \\
			\diag\left(\begin{pmatrix}* & & \\  & 1 & \\& & *\end{pmatrix}_L\right) \arrow[u, "\subset"] \arrow[r, "\sim"] & \begin{pmatrix} * & & \\ & 1 & \\ & & *\end{pmatrix}_L \arrow[u, "\subset"] \arrow[r, "(\zeta_a)_L"] & (\tilde S_a)_L \arrow[u, "\subset"]
		\end{tikzcd}.
	\end{center}
	The torus $\prod_{i \in I_a} \begin{pmatrix}* & & \\  & * & \\& & *\end{pmatrix}_L$ is then the unique maximal torus of $\prod_{i\in I_a} R_{L'_i/L}(\SL_{3,L'_i})$ containing the split torus $\diag\left(\begin{pmatrix}* & & \\  & 1 & \\& & *\end{pmatrix}_L\right)$.
	
	For each $i \in I_a$, denote by $\alpha_i$ (resp. $\beta_i$) the character of $T_L$ that pulls back to $$ \left(\begin{pmatrix}s_j & & \\  & s_j^{-1}t_j & \\& & t_j^{-1}\end{pmatrix}\right)_{j \in I_a} \mapsto s_i^2 t_i^{-1} \left( \text{resp.}  \left(\begin{pmatrix}s_j & & \\  & s_j^{-1}t_j & \\& & t_j^{-1}\end{pmatrix}\right)_{j \in I_a} \mapsto t_i^2 s_i^{-1} \right)$$ along $(\zeta_a)_L$.
	The latter characters are the only roots of $\prod_{i \in I_a} R_{L'_i/L}(SL_{3, L'_i})$ that restrict to the character $\begin{pmatrix}t & & \\ & 1 & \\ & & t^{-1}\end{pmatrix}\mapsto t$ on $\diag \left(\begin{pmatrix}* & &\\ & 1 & \\ & & *\end{pmatrix}_L\right)$, therefore we have $$\{\alpha \in \tilde \Delta, \alpha_{|S_L} = a_L\} = \bigcup_{i \in I_a}\{\alpha_i, \beta_i\} \text{ and } \{\alpha \in \tilde \Phi, \alpha_{|S_L} = 2a_L\} = \{\alpha_i+\beta_i, i \in I_a\}$$
	
	For each $i \in I_a$, set $$\zeta_{\alpha_i}: \begin{array}{ccc} R_{L'_i/L}(SL_{2, L'_i}) & \longto & G_{\alpha_i} \\ \begin{pmatrix} a & b\\ c & d \end{pmatrix} & \longmapsto & \zeta_a\left(\begin{pmatrix} a & b & \\ c & d & \\ & & 1\end{pmatrix}\right)\end{array},$$
	$$\zeta_{\beta_i}: \begin{array}{ccc} R_{L'_i/L}(SL_{2, L'_i}) & \longto & G_{\beta_i} \\ \begin{pmatrix} a & b\\ c & d \end{pmatrix} & \longmapsto & \zeta_a\left(\begin{pmatrix} 1 & & \\ & a & b \\ & c & d \end{pmatrix}\right)\end{array},$$
	and $$ \zeta_{\alpha_i+\beta_i}: \begin{array}{ccc} R_{L'_i/L}(SL_{2, L'_i}) & \longto & G_{\alpha_i  +\beta_i} \\ \begin{pmatrix} a & b\\ c & d \end{pmatrix} & \longmapsto & \zeta_a\left(\begin{pmatrix} a & & b \\ & 1 &  \\ c & & d \end{pmatrix}\right)\end{array}.$$
	Here, we think of $R_{L'_i/L}(SL_{3, L'_i})$ as a direct factor of $R_{K'_2 \otimes_K L}(SL_{3, K'_2\otimes_KL})$.
	
	Then, for each $i \in I_a$ and $\sigma \in \Gal(L/K)$, we have $$\sigma \circ \zeta_{\alpha_i} \circ \sigma^{-1} =  \zeta_{\alpha_{\sigma(i)}} = \zeta_{\sigma(\alpha_i)},$$
	$$\sigma \circ \zeta_{\beta_i} \circ \sigma^{-1} =  \zeta_{\beta_{\sigma(i)}} = \zeta_{\sigma(\beta_i)},$$
	and $$\sigma \circ \zeta_{\alpha_i+\beta_i} \circ \sigma^{-1} =  \zeta_{\sigma(\alpha_i+\beta_i)}.$$
	
	Once again, if, for each $\alpha \in \tilde \Delta$ such that $\alpha_{|S_L} \in \{a_L, 2a_L\}$, we define $\varphi_{\alpha}: U_{\alpha} \to \RR$ by $$\varphi_{\alpha}\left(\zeta_{\alpha}\begin{pmatrix}1 & x\\ & 1\end{pmatrix} \right) = \omega(x),$$ then for each $\sigma \in \Gal(L/K)$ and each such $\alpha$, we have $$\varphi_{\alpha} \circ \sigma^{-1} = \varphi_{\sigma(\alpha)}.$$
	Recall also that, denoting by $H_0(K', K_2')$ the $K'_2$-variety $$H_0(K', K_2') = \{(x, y) \in R_{K'/K'_2}(\mbb A_{K'}^2), y+\tau(y) = x\tau(x)\},$$ where $\tau$ is the nontrivial element of $\Gal(K'/K'_2)$, the group of upper-triangular unipotent matrices in $R_{K_2'/K	}SU_{3, K'/K'_2}$ is identified with $R_{K'_2/K}(H_0(K', K'_2))$ via the map $$\mu: (x, y) \mapsto \begin{pmatrix} 1 & -\tau(x) & -y \\ & 1 & x \\ & & 1\end{pmatrix}.$$
	Then, the inclusion $U_a(K) \subset U_a(L) = \prod_{i \in I_a} U_{\alpha_i}(L) U_{\alpha_i+\beta_i}(L) U_{\beta_i}(L)$ spells out explicitly as the map $$\zeta_a(\mu(x,y)) \mapsto \prod_{i \in I_a} \zeta_{\alpha_i}(-\tau(x))\zeta_{\alpha_i+\beta_i}(\tau(y))\zeta_{\beta_i}(x).$$
	Then, if $u = \zeta_a(\mu(x,y)) \in U_a(K) \setminus \{1\}$ decomposes as $u = \prod_{i \in I_a} u_{\alpha_i} u_{\alpha_i+\beta_i}u_{\beta_i}$, we have for each $i \in I_a$: $$\varphi_{\alpha_i}(u_{\alpha_i}) = \varphi_{\beta_i}(u_{\beta_i}) = \omega(x) \ge \varphi_a^x(u)$$ and $$\varphi_{\alpha_i+\beta_i}(u_{\alpha_i+\beta_i}) = \omega(y) = 2\varphi_a^x(u).$$

\end{enumerate}	 
	
	We have thus constructed a quasi-épinglage $(\zeta_{\alpha})_{\alpha \in \tilde \Delta}$ \cite[Définition 2.4]{Lou} in $G_L$ along $T_L$. By \cite[Proof of Proposition 2.5]{Lou}, this quasi-épinglage extends to a Chevalley quasi-system $(\zeta_{\alpha})_{\alpha \in \tilde \Phi}$ in $G^{\pred}_L$ along $T_L$. Let us now extend our family $(\varphi_{\alpha})_{\alpha \in \tilde \Delta}$ to a valuation $(\varphi_{\alpha})_{\alpha \in \tilde \Phi_{\nd}}$ of the root data $(Z_G(T)(L), (U_{\alpha}(L))_{\alpha \in \tilde \Phi})$. Our main tool is the fact that, for all non-divisible roots $\alpha$ and $\beta$ in $\tilde \Phi_{\nd}$, the parametrisation $\inta(m_{\alpha}) \circ \zeta_{\beta}$ is \textit{similar} to $\zeta_{s_{\alpha}(\beta)}$, in the sense that is differs from $\zeta_{s_{\alpha}(\beta)}$ by precomposition with a field $L$-automorphism of $\bar L$ that maps the domain of $\zeta_{\beta}$ onto that of $\zeta_{s_{\alpha}(\beta)}$, and by a possible sign change. 
	
	Recall from \cite[Définition 2.4]{Lou} that, if $(\zeta_a)_{a\in \Phi}$ is a Chevalley quasi-system, we let $$m_a = \left\{\begin{array}{cl} \zeta_a\begin{pmatrix} & 1 \\ -1 & \end{pmatrix} & \text{if } \dom(\zeta_a) \text{ is of the form } R_{K'/K}(SL_{2, K'}) \\ \zeta_a\begin{pmatrix}& & -1 \\ & - 1 & \\ -1 & &\end{pmatrix} & \text{if } \dom(\zeta_a) \text{ is of the form } R_{K_2'/K}(SU_{3, K'/K'_2}) \\ x_{-a}(0, -1)x_a(0, 1)x_{-a}(0, -1) & \text{if } \dom(\zeta_a) \text{ is of the form } R_{K'/K}(BC_{1, K'}) \end{array}\right. .$$ The computations we made above allow us to relate the $(m_a)_{a\in \Phi_{\nd}}$ and the $(m_{\alpha})_{\alpha \in \tilde \Phi_{\nd}}$ as follows
	\begin{enumerate}
	    \item If $a \in \Delta$ is such that $\zeta_a$ has as domain a Weil restriction $R_{K'/K}(SL_{2, K'})$ or $R_{K'/K}(BC_{1, K'})$, then we have $m_a = \prod_{\substack{\alpha \in \tilde \Phi\\ \alpha_{|S} = a}} m_{\alpha}.$
	    \item If $a \in \Delta$ is such that $\zeta_a$ has as domain a Weil restriction $R_{K'_2/K}(SU_{3, K'/K'_2})$, then for some order on the roots, we have $m_a = \lambda(-1) \left(\prod_{\substack{\alpha \in \tilde \Phi\\ \alpha_{|S} \in \{a, 2a\}}}m_{\alpha}\right)$, for some cocharacter $\lambda \in X_*(T)$.
	\end{enumerate}
	This allows one to define valuation functions by precomposing those already constructed by products of $\inta(m_{\alpha})$, and to transfer properties the general case that were proved above for the $\varphi_{\alpha}$, where $\tilde \alpha \in \tilde \Phi$ with $\alpha_{|S} \in \Delta \cup (2\Delta \cap \Phi)$.
	The reason we expect this transfer to work is that, for valuations  of root data, we do have such a compatibility with conjugacy by specific elements of $N_G(T)(L)$, see \cite[Proposition 6.2.7]{BT1} -- with the notations of this proposition, we have $m_{\alpha} \in M_{\alpha, 0}$ for each $\alpha$.
	
	Observe \cite[Ch. VI, §1, Proposition 15]{Bou} that each non-divisible root $a \in \Phi_{\nd}$ can be mapped to $\Delta$ by some element $w\in W(G, S)(K)$. For each $a\in \Phi_{\nd} \setminus \Delta$, fix once and for all a simple root $b_a \in \Delta$ and an element $w_a \in W(G, S)(K)$ such that $w_a(a) = b_a$.
	
	Finally, recall that, as observed in the proof of Lemma \ref{lem:maximal_tori_quasisplit}, we have an inclusion $$W(G, S)(K)\hookrightarrow W(G_L, T_L)(L).$$ This inclusion has the property that, for any character $\chi$ and any $w \in W(G, S)(K)$, we have $w(\chi)_{|S} = w(\chi_{|S})$.
	
	Now, let $\alpha \in \tilde \Phi_{\nd} \setminus \tilde \Delta$ whose restriction to $S$ is a root $ja$, with $j \in \{1,2\}$ and $a \in \Phi_{\nd} \setminus \Delta$. Set $$\varphi_{\alpha} = \varphi_{\beta_{\alpha}} \circ \inta(m_{\alpha_1}) \circ\dots \inta(m_{\alpha_r}),$$
	where $\alpha_1, \dots, \alpha_r, \beta_{\alpha} \in \tilde \Delta$, with $\beta_{\alpha} = w_a(\alpha)$ and $w_a = s_{\alpha_1}\dots s_{\alpha_r}$.
	This function does not depend on the decomposition of $w_a$ as a product of simple reflections. Indeed, if $\zeta_{\alpha}$ has as domain a Weil restriction $R_{L'/L}(SL_{2, L'})$, then because $(\zeta_{\alpha})_{\alpha \in \tilde \Phi_{\nd}}$ is a Chevalley quasi-system, there exists a $\tau \in\Aut_L(\ol L)$ such that, for each $x \in L'$, we have $$ \inta(m_{\alpha_1}) \circ\dots \inta(m_{\alpha_r}) \circ \zeta_{\alpha}\begin{pmatrix} 1 & x \\ & 1\end{pmatrix} = \zeta_{\beta_{\alpha}}\begin{pmatrix} 1 & \pm \tau(x) \\ & 1 \end{pmatrix},$$
	and therefore $$\varphi_{\alpha}\left( \zeta_{\alpha}\begin{pmatrix} 1 & x \\ & 1\end{pmatrix}\right)  = \varphi_{\beta_{\alpha}}\left(\zeta_{\beta_{\alpha}}\begin{pmatrix} 1 & \pm \tau(x) \\ & 1 \end{pmatrix}\right) = \omega(x).$$
	On the other hand, if $\zeta_{\alpha}$ has as domain a Weil restriction $R_{L'/L}(BC_{1, L'})$, then likewise, there exists $\tau \in \Aut_L(\ol L)$ such that, for each $(u, v) \in L'^{1/2}\times L'$, we have $$ \inta(m_{\alpha_1}) \circ\dots \inta(m_{\alpha_r}) \circ \zeta_{\alpha}(x_{\alpha}(u,v)) = \zeta_{\beta_{\alpha}}(x_{\beta_{\alpha}}( \tau(u), \tau(v))).$$ Therefore, there exists $\eta \in L'^{1/2} \setminus L'$ such that $$\varphi_{\alpha}(x_{\alpha}(u,v)) = \varphi_{\beta_{\alpha}}(x_{\beta_{\alpha}}(\tau(u), \tau(v))) = \frac{1}{2}\omega(\eta u^2 + v)$$ for all $(u, v) \in L'^{1/2} \times L'$.
	
	It now follows from \cite[Proposition 3.2]{Lou} that the collection $(\varphi_{\alpha})_{\alpha \in \tilde \Phi_{\nd}}$ defined above constitutes a valuation of the root data $(Z_G(T)(L), (U_{\alpha}(L))_{\alpha \in \tilde \Phi})$ corresponding to a point $z \in A(G_L, T)$. Furthermore, we claim that, if $a \in \Phi_{\nd}$ is a non-divisible root and $u \in U_a(K)\setminus \{1\}$, we have the equivalence $$u \in U_{a, r}^x \iff u \in \prod_{\substack{\alpha \in \tilde \Phi_{\nd}\\ \alpha_{|S_L} = a_L}} U_{\alpha, r}^z \times \prod_{\substack{\alpha \in \tilde \Phi_{\nd}\\ \alpha_{|S_L} = 2a_L}} U_{\alpha, 2r}^z,$$
	that is $$\varphi_a^x(u) = \sup\left\{r \in \RR,   u \in \prod_{\substack{\alpha \in \tilde \Phi_{\nd}\\ \alpha_{|S_L} = a_L}} U_{\alpha, r}^z \times \prod_{\substack{\alpha \in \tilde \Phi_{\nd}\\ \alpha_{|S_L} = 2a_L}} U_{\alpha, 2r}^z \right\}.$$
	This follows from earlier computations in the case where $a \in \Delta$. The general case will be reduced to that of simple roots. Let $a \in \Phi_{\nd}\setminus \Delta$.
	Let $a_1, \dots, a_p \in \Delta$ be such that $w_a$ admits the following decomposition as a product of simple reflections $$w_a = s_{a_1} \dots s_{a_p}.$$
	Let $u \in U_a(K) \setminus \{1\}$. On the one hand, by \cite[Proposition 6.2.7]{BT1}, we have $$\varphi_{a}^x(u) = \varphi_{b_a}^x ( \inta(m_{a_1})\circ \cdots \inta(m_{a_p})(u)).$$
	On the other hand, if $\alpha \in \tilde \Phi$ is such that $\alpha_{|S} \in \{ a, 2a\}$, then because $$w_a = \left(\prod_{\substack{\alpha' \in \tilde \Phi\\ \alpha'_{|S} \{a_1, 2a_1\}}}s_{\alpha'}\right)  \cdots \left(\prod_{\substack{\alpha' \in \tilde \Phi\\ \alpha'_{|S} \in \{a_p, 2a_p\}}}s_{\alpha'}\right)$$ is a decomposition of $w_a$ as a product of reflections in $W(G, T)(L)$, the component $u_{\alpha}$ of $u$ in $U_{\alpha}(L)$ satisfies  $$\begin{array}{ccl} \varphi_{\alpha}(u_{\alpha}) & = & \varphi_{\beta_{\alpha}}\left(\left(\prod_{\substack{\alpha' \in \tilde \Phi\\ \alpha'_{|S} \in \{a_1, 2a_1\}}}\inta(m_{\alpha'})\right) \circ  \cdots \left(\prod_{\substack{\alpha' \in \tilde \Phi\\ \alpha'_{|S} \in\{ a_p, 2a_p\}}}\inta(m_{\alpha'})\right)(u_{\alpha})\right)\\ & = & \varphi_{\beta_{\alpha}}(\inta(\lambda(-1))\circ \inta(m_{a_1})\circ \cdots \inta(m_{a_p})(u_{\alpha}))\end{array}$$ for some $\lambda \in X_*(T)$. 
	We can then check explicitly using our formulas for the épinglage $(\varphi_{\alpha})_{\alpha \in \tilde \Delta}$ that $\varphi_{\beta_{\alpha}}\circ \inta(\lambda(-1))  = \varphi_{\beta_{\alpha}}$.
    Now, because $\beta_{\alpha} = w_a(\alpha)$ lifts $b_a = w_a(a)$, and because $\inta(m_{a_1})\circ \cdots \inta(m_{a_p})(u_{\alpha})$ is the component of $\inta(m_{a_1})\circ \cdots \inta(m_{a_p})(u)$ in $U_{\beta_{\alpha}}(L)$, we deduce from previous computations that $$\begin{array}{cc} \varphi_{\alpha}(u_{\alpha}) \ge  \varphi_a^x(u) & \text{if } \alpha_{|S} = a\\ \varphi_{\alpha}(u_{\alpha}) = 2\varphi_a^x(u) & \text{if } \alpha_{|S} = 2a \end{array}.$$
	
	This completes the proof of the fact that the valuation $(\varphi_a^x)_{a \in \Phi}$ is induced by the valuation $(\varphi_{\alpha})_{\alpha \in \tilde \Phi}$ in the sense of \cite[Définition 9.1.11]{BT1}. By \cite[Proposition 9.1.17]{BT1}, there exists a $G(K)$-equivariant map $$j: \mc B(G, K) \to \mc B(G, L)$$ that maps $x$ to $z$ and is an isometric embedding if the metric on $V(G_K, S)$ is chosen to be induced by a choice of inner product on $V(G_L, T)$. In particular, it is affine on apartments, which establishes conditions 1. 
	
	Our last task is to check that $z \in \mc B(G, L)^{\Gal(L/K)}$, which amounts to proving that the valuation $(\varphi_{\alpha})_{\alpha \in \tilde \Phi_{\nd}}$ is $\Gal(L/K)$-invariant. 
	We proved that, for all $\alpha \in \tilde \Phi$ and $\sigma \in \Gal(L/K)$, we have $$\varphi_{\alpha} \circ \sigma^{-1} = \varphi_{\sigma(\alpha)}.$$
	Now, if $\alpha \in \tilde \Phi_{\nd}$ satisfies $\alpha_{|S} \in \{a, 2a\}$ with $a \in \Phi_{\nd} \setminus \Delta$, then writing $w_a = s_{\alpha_1} \circ \cdots s_{\alpha_r}$ as a product of simple reflections in the absolute Weyl group $W(G,T)(L)$, we have for each $\sigma \in \Gal(L/K)$: $$\begin{array}{ccl}\varphi_{\alpha}\circ \sigma^{-1} & = & \varphi_{\beta_{\alpha}}\circ \inta(m_{\alpha_1}) \circ \cdots \circ \inta(m_{\alpha_r})\circ \sigma^{-1}\\ &= & \varphi_{\beta_{\alpha}} \circ \sigma^{-1} \circ \inta(\sigma(m_{\alpha_1})) \circ \dots \inta(\sigma(m_{\alpha_r}))\\ 
	& = & \varphi_{\sigma(\beta_{\alpha})} \circ \inta(\sigma(m_{\alpha_1})) \circ \dots \inta(\sigma(m_{\alpha_r})) \end{array}.$$
	Because $w_a \in W(G, S)(K)$ is $\Gal(L/K)$-invariant, we have $$\sigma(\beta_{\alpha}) = \sigma(w_a(\alpha)) = w_a(\sigma(\alpha)) = \beta_{\sigma(\alpha)}.$$
	Moreover, given the compatibility between the parameterisations $\zeta_{\alpha}$ for $\alpha \in \tilde \Delta$ and the action of $\Gal(L/K)$, we have $\sigma(m_{\alpha_i}) = m_{\sigma(\alpha_i)}$ for each $i \in [\![1,r]\!]$. 
	We deduce that $$\begin{array}{ccl} \varphi_{\alpha} \circ \sigma^{-1} & = & \varphi_{\beta_{\sigma(\alpha)}}\circ \inta(m_{\sigma(\alpha_1)}) \circ \cdots \inta(m_{\sigma(\alpha_r)})\\ & = & \varphi_{\sigma(\alpha)}\end{array},$$
	as $w_a = \sigma(w_a) = s_{\sigma(\alpha_1)} \cdots s_{\sigma(\alpha_r)}$. This completes the proof of the $\Gal(L/K)$-invariance of $(\varphi_{\alpha})_{\alpha \in \tilde \Phi_{\nd}}$, and therefore $z$ is fixed under the action of $\Gal(L/K)$ on $\mc B(G, L)$.
	
	By $G(K)$-equivariance of $j$, we deduce that $j(G(K) \cdot x) \subset \mc B(G, L)^{\Gal(L/K)}$. Finally, because $\mc B(G, K)$ is the convex hull of $G(K) \cdot x$, we get $j(\mc B(G, K)) \subset \mc B(G, L)^{\Gal(L/K)}$.
	
	The above map $j$ is therefore a map from $\mc B(G, K)$ to $\mc B(G, L)$ that satisfies conditions 1-3, which completes the proof of Proposition \ref{quasisplit_to_split_embedding}.
	
\end{proof}

\subsubsection*{The general case}

We are now in position to prove the general functoriality theorem. Let us begin with the following technical corollary of Theorem \ref{Functoriality_iso} as well as Propositions \ref{strict_henselization_embedding} and \ref{quasisplit_to_split_embedding}, which will prove very useful below.

\begin{prop}\label{prop:special_galois_eq}
	Let $K$ and $L$ be finite separable extensions of $k$.
	\begin{enumerate}
		\item Assume that $G_K$ and $G_L$ have the same relative rank. Let $K_0$ and $L_0$ be extensions of $k$ such that the following inclusions hold $$\begin{matrix} K & \subset & L \\ \cup & & \cup \\ K_0 & \subset & L_0 \end{matrix}.$$ Assume that $K$ is Galois over $K_0$ and that $L$ is Galois over $L_0$. Then, the map $p_{L/K}$ is $\Gal(L/L_0)$-equivariant, if we let $\Gal(L/L_0)$ act on $\mc B(G, K)$ via the restriction map $\res_{L/K}: \Gal(L/L_0) \to \Gal(K/K_0)$.
		\item Assume that $K$ and $L$ satisfy the conditions of Proposition \ref{strict_henselization_embedding} or \ref{quasisplit_to_split_embedding}. Let $K_0 \subset K$ be a subextension of $k$ such that $K$ and $L$ are Galois over $K_0$. Then, the map $p_{L/K}$ is $\Gal(L/K_0)$-equivariant, if we let $\Gal(L/K_0)$ act on $\mc B(G, K)$ via the restriction map $\res_{L/K}: \Gal(L/K_0) \to \Gal(K/K_0)$.
	\end{enumerate}
\end{prop}

\begin{proof}
	\begin{enumerate}
		\item Let $\gamma \in \Gal(L/L_0)$ and $\gamma_K = \res_{L/K}(\gamma)$. Then, by Theorems \ref{Functoriality_iso} and \ref{Galois_action_building}, the map $$p_{\gamma} \circ p_{L/K} \circ p_{\gamma_K^{-1}} : \mc B(G, K) \to \mc B(G, L)$$ satisfies conditions 1 and 2 of Theorem \ref{Functoriality_iso} with $\sigma:K \to L$ being the inclusion, hence we must have $p_{\gamma} \circ p_{L/K} \circ p_{\gamma_K^{-1}} = p_{L/K}$.
		\item Let $\gamma \in \Gal(L/K_0)$ and $\gamma_K = \res_{L/K}(\gamma)$. To prove that $p_{\gamma} \circ p_{L/K} \circ p_{\gamma_K^{-1}} = p_{L/K}$, it will be sufficient, because of the uniqueness properties of Propositions \ref{strict_henselization_embedding} and \ref{quasisplit_to_split_embedding} to prove that the map $p_{\gamma} \circ p_{L/K} \circ p_{\gamma_K^{-1}}$ satisfy conditions 1-3.
		As previously observed, conditions 1 and 2 follow directly from Theorem \ref{Galois_action_building}. Lastly, we need to check that $$p_{\gamma} \circ p_{L/K} \circ p_{\gamma_K^{-1}}(\mc B(G, K)) \subset \mc B(G, L)^{\Gal(L/K)},$$ that is $$p_{\gamma} \circ p_{L/K}(\mc B(G, K)) \subset \mc B(G, L)^{\Gal(L/K)}.$$
		It is sufficient to check that $p_{\gamma}$ maps the $\Gal(L/K)$-fixed locus of $\mc B(G, L)$ to itself, which is immediate because $\Gal(L/K)$ is normal in $\Gal(L/K_0)$.
	\end{enumerate}
\end{proof}

We now state and prove the main theorem.

\begin{thm}\label{Functoriality_inclusions}
	There exists a unique family $(p_{L/K}: \mc B(G, K) \to \mc B(G, L))$ indexed by pairs of finite separable extensions of $k$ such that $K \subset L$ satisfying the following properties:
	\begin{enumerate}
		\item For each $K \subset L$, each $S \subset G_K$, there exists $T \subset G_L$ such that $S_L \subset T$ and such that the map $p_{L/K}$ restricts to an affine isomorphism of $A(G_K, S)$ onto an affine subspace of $A(G_L, T)$.
		\item For each $K \subset L$, the map $p_{L/K}$ is $G(K)$-equivariant.
		\item Whenever $L/K$ is Galois, we have $p_{L/K}(\mc B(G, K)) \subset (\mc B(G, L))^{\Gal(L/K)}$.
		\item For each finite separable extensions $K, L, M$ of $k$ such that $K \subset L \subset M$, we have $$p_{M/K} = p_{M/L} \circ p_{L/K}.$$
	\end{enumerate}
\end{thm}

\begin{proof}
	We follow the proof of \cite[Théorème 5.2]{Rou}, with minimal adjustments. We first establish the uniqueness assertion.
	
	\textit{(Uniqueness)} Let $(p_{L/K})_{k \subset K \subset L}$ be a system of maps satisfying conditions 1-4. Observe that, if $K \subset L$, with $L = K^{\nr}$, then $p_{L/K}$ must be the map constructed in Proposition \ref{strict_henselization_embedding}. Likewise, if $K \subset L$, with $G_K$ quasi-split and $G_L$ pseudo-split, then $p_{L/K}$ must be the map constructed in Proposition \ref{quasisplit_to_split_embedding}. Now, let $K$ and $L$ be arbitrary finite separable extensions of $k$ with $K \subset L$.
	
	Consider the strict henselisation $K^{\nr}$ (resp. $L^{\nr}$) of $K$ (resp. $L$). Then, the groups $G_{K^{\nr}}$ and $G_{L^{\nr}}$ are quasi-split by \cite[Corollaire 2.3]{Lou}. There is thus a unique minimal finite Galois extension $\tilde{K^{\nr}}/K^{\nr}$ (resp. $\tilde{L^{\nr}}/L^{\nr}$) such that $G_{\tilde{K^{\nr}}}$ (resp. $G_{\tilde{L^{\nr}}}$) is pseudo-split. Then, the extensions $K^{\nr}$ and $\tilde{K^{\nr}}$ (resp. $L^{\nr}$ and $\tilde{L^{\nr}}$) are Galois over $K$ (resp. $L$), the group $G_{K^{\nr}}$ (resp. $G_{L^{\nr}}$) is quasi-split and the group $G_{\tilde{K^{\nr}}}$ (resp. $G_{\tilde{L^{\nr}}}$) is pseudo-split. 
	
	It then follows from condition 4 and Propositions \ref{strict_henselization_embedding} and \ref{quasisplit_to_split_embedding} that the following diagram commutes
	\begin{center}
		\begin{tikzcd}
			\mc B(G, K) \arrow[r, "p_{K^{\nr}/K}"] \arrow[d, "p_{L/K}"] & \mc B(G, K^{\nr}) \arrow[r, "p_{\tilde{K^{\nr}}/K^{\nr}}"] & \mc B(G, \tilde{K^{\nr}}) \arrow[d, "p_{\tilde{L^{\nr}}/\tilde{K^{\nr}}}"] \\
			\mc B(G, L) \arrow[r, "p_{L^{\nr}/L}"] & \mc B(G, L^{\nr}) \arrow[r, "p_{\tilde{L^{\nr}}/L^{\nr}}"] & \mc B(G, \tilde{L^{\nr}})
		\end{tikzcd}.
	\end{center}
	
	Because all maps in the diagram are injective, the map $p_{L/K}$ is uniquely determined in terms of the maps $p_{K^{\nr}/K}, p_{\tilde{K^{\nr}}/K^{\nr}}, p_{\tilde{L^{\nr}}/\tilde{K^{\nr}}}, p_{\tilde{L^{\nr}}/L^{\nr}}$, and $p_{L^{\nr}/L}$, each of which is uniquely determined in terms of $G$, $K$, and $L$ as a consequence of either Proposition \ref{strict_henselization_embedding} or Proposition \ref{quasisplit_to_split_embedding}.
	
	\textit{(Existence)} Let $K$ and $L$ be arbitrary finite separable extensions of $k$ with $K \subset L$. We consider the extensions -- with the same notations as above $$K_1 = K^{\nr}, K_2 = \tilde{K^{\nr}}, L_1 = L^{\nr}, L_2 = \tilde{L^{\nr}}.$$ Consider $$\begin{array}{ccc}\Gamma = \Gal(K_2/K),& \Gamma_1 = \Gal(K_1/K), & \Gamma_2 = \Gal(K_2/K_1)\\ \Gamma' = \Gal(L_2/L), & \Gamma'_1 = \Gal(L_1/L), & \Gamma'_2 = \Gal(L_2/L_1) \end{array}.$$ Restriction maps fit in a commutative diagram 
	\begin{center}
		\begin{tikzcd}
			\Gamma_2 \arrow[rr, "\subset"] & & \Gamma \arrow[rr, twoheadrightarrow, "\res_{K_2/K_1}"] & & \Gamma_1 \\
			\Gamma'_2 \arrow[rr, "\subset"] \arrow[u, hookrightarrow, "\res_{L_2/K_2}"] & & \Gamma' \arrow[rr, twoheadrightarrow, "\res_{L_2/L_1}"] \arrow[u, hookrightarrow, "\res_{L_2/K_2}"] & & \Gamma'_1 \arrow[u, hookrightarrow, "\res_{L_1/K_1}"]
		\end{tikzcd}.
	\end{center}
	By Theorem \ref{Galois_action_building} and Propositions \ref{strict_henselization_embedding} and \ref{quasisplit_to_split_embedding}, we may then form the following diagram.
	
	\begin{center}
		\begin{tikzcd}
			\mc B(G, K) \arrow[r, "p_{K_1/K}"]  & \mc B(G, K_1)  \arrow[loop above, "\Gamma_1"]  \arrow[r, "p_{K_2/K_1}"] & \mc B(G, K_2) \arrow[loop above, "\Gamma"]  \arrow[d, "p_{L_2/K_2}"] \\
			\mc B(G, L) \arrow[r, "p_{L_1/L}"] & \mc B(G, L_1) \arrow[loop below, "\Gamma'_1"] \arrow[r, "p_{L_2/L_1}"] & \mc B(G, L_2) \arrow[loop below, "\Gamma'"]
		\end{tikzcd}.
	\end{center}
	
	Because all maps in the diagram are injective, it is sufficient, in order to construct a map $\mc B(G, K) \to \mc B(G, L)$ that makes the rectangle commute, to prove that $$p_{L_2/K_2} \circ p_{K_2/K_1} \circ p_{K_1/K}(\mc B(G, K)) \subset p_{L_2/L_1} \circ p_{L_1/L}(\mc B(G, L)).$$
	We proceed in two steps in order to exploit the features of the two special cases: First, we construct a map $p_{L_1/K_1}: \mc B(G, K_1) \to \mc B(G, L_1)$ making the right-hand square commute. Then, we exploit the uniqueness of this map to construct the map $p_{L/K}$.
	
	\textit{Step 1:} By the same argument as before, it is sufficient to prove that $$p_{L_2/K_2} \circ p_{K_2/K_1}(\mc B(G, K_1)) \subset p_{L_2/L_1}(\mc B(G, L_1)).$$
	Furthermore, by $G(K_1)$-equivariance, it is sufficient to check that $$p_{L_2/K_2} \circ p_{K_2/K_1}(A(G_{K_1}, S_1)) \subset p_{L_2/L_1}(\mc B(G, L_1))$$ for some maximal $K_1$-split torus $S_1 \subset G_{K_1}$. Let $S_1$ be such a maximal $K_1$-split torus in $G_{K_1}$. By Proposition \ref{quasisplit_to_split_embedding} and Lemma \ref{lem:maximal_tori_quasisplit}, the image $p_{K_2/K_1}(A(G_{K_1}, S_1))$ lies in the apartment $A(G_{K_2}, S_2)$ associated to the unique maximal $K_2$-split torus $S_2 \subset G_{K_2}$ that contains $(S_1)_{K_2}$. Now, by Theorem \ref{Functoriality_iso}, the image $p_{L_2/K_2}(A(G_{K_2}, S_2))$ is precisely $A(G_{L_2}, T_2)$, with $T_2 = (S_2)_{L_2}$. As $T_2$ is defined over $K_1$ (Lemma \ref{lem:maximal_tori_quasisplit}), it is in particular defined over $L_1$. Hence, the action of $\Gamma'_2$ on $B(G, L_2)$ leaves $A(G_{L_2}, T_2)$ invariant. Because $p_{L_2/K_2}$ is $\Gamma'_2$-equivariant, by Proposition \ref{prop:special_galois_eq}, we have the inclusion $$p_{L_2/K_2} \circ p_{K_2/K_1}(A(G_{K_1}, S_1)) \subset A(G_{L_2}, T_2)^{\Gamma'_2}.$$
	To conclude, let $T_1$ be a maximal $L_1$-split torus in $G_{L_1}$ containing $(S_1)_{L_1}$. By the same reasoning as above, the image $p_{L_2/L_1}(A(G_{L_1}, T_1))$ lies in the apartment $A(G_{L_2}, T'_2)$ associated to the unique maximal $L_2$-split torus of $G_{L_2}$ containing $(T_1)_{L_2}$. Furthermore, by construction of $p_{L_2/L_1}$, we must then have $$p_{L_2/L_1}(A(G_{L_1}, S_1)) = A(G_{L_2}, T'_2)^{\Gamma'_2}.$$ Observe that the torus $T'_2$ must then contain $(S_1)_{L_2}$. By uniqueness of the maximal torus in $(Z_{G_{K_1}}(S_1))_{L_2}$, we must then have $T'_2 = T_2$. Hence, we have the inclusion $$p_{L_2/K_2} \circ p_{K_2/K_1}(A(G_{K_1}, S_1)) \subset p_{L_2/L_1}(A(G_{L_1}, T_1)),$$ which completes the proof. 
	
	We have thus proved the existence of a map $p_{L_1/K_1}: \mc B(G, K_1) \to \mc B(G, L_1)$ such that the square
	\begin{center}
		\begin{tikzcd}
			\mc B(G, K_1) \arrow[d, "p_{L_1/K_1}"] \arrow[r, "p_{K_2/K_1}"] & \mc B(G, K_2) \arrow[d, "p_{L_2/K_2}"] \\
			\mc B(G, L_1) \arrow[r, "p_{L_2/L_1}"] & \mc B(G, L_2)
		\end{tikzcd}
	\end{center}
	commutes. By $G(K_1)$-equivariance and injectivity of $p_{L_2/K_2} \circ p_{K_2/K_1}$ and $p_{L_2/L_1}$, the map $p_{L_1/K_1}$ must be $G(K_1)$-equivariant. Moreover, it follows from the construction that it maps each apartment $A(G_{K_1}, S_1)$ to the apartment $A(G_{L_1}, T_1)$ associated to the unique maximal $L_1$-split torus such that $(S_1)_{L_1} \subset T_1$. The restriction map	$$p_{L_1/K_1}: A(G_{K_1}, S_1) \to A(G_{L_1}, T_1)$$ is then affine because, letting $S_2$ be the unique maximal $K_2$-split torus of $G_{K_2}$ containing $S_1$ and $T_2 = (S_2)_{L_2}$, we have the commutative diagram 
	\begin{center}
		\begin{tikzcd}
			A(G_{K_1}, S_1) \arrow[d, "p_{L_1/K_1}"] \arrow[r, "p_{K_2/K_1}"] & A(G_{K_2}, S_2) \arrow[d, "p_{L_2/K_2}"] \\
			A(G_{L_1}, T_1) \arrow[r, "p_{L_2/L_1}"] & A(G_{L_2}, T_2)
		\end{tikzcd},
	\end{center}
	where $p_{L_2/K_2} \circ p_{K_2/K_1}$ and $p_{L_2/L_1}$ are affine and injective.
	
	We also prove that $p_{L_1/K_1}$ is $\Gamma'_1$-equivariant, which will also follow formally from Galois-equivariance properties of the other maps in the square. Let $\gamma'_1 \in \Gamma'_1$, a lift $\gamma'$ of $\gamma'_1$ to $L_2$, the restriction $\gamma = \res_{L_2/K_2}(\gamma')$, and the restriction $\gamma_1 = \res_{K_2/K_1}(\gamma)$. Then, Proposition \ref{prop:special_galois_eq} implies that $$\begin{array}{ccll} p_{L_2/L_1} \circ p_{\gamma'_1} \circ p_{L_1/K_1} \circ p_{\gamma_1}^{-1} & = & p_{\gamma'} \circ p_{L_2/L_1} \circ j \circ p_{\gamma_1}^{-1} & (\Gamma'\text{-equivariance of } p_{L_2/L_1}) \\  & = & p_{\gamma'} \circ p_{L_2/K_2} \circ p_{K_2/K_1} \circ p_{\gamma_1}^{-1} & (\text{commutativity of the diagram})\\ & = & p_{L_2/K_2} \circ p_{\gamma} \circ p_{K_2/K_1} \circ p_{\gamma_1}^{-1} & (\Gamma'\text{- equivariance of } p_{L_2/K_2})\\ & = & p_{L_2/K_2} \circ p_{K_2/K_1} & (\Gamma\text{-equivariance of } p_{K_2/K_1})\\ & = & p_{L_2/L_1} \circ p_{L_1/K_1}\end{array},$$
	and therefore $$p_{\gamma'_1} \circ p_{L_1/K_1} \circ p_{\gamma_1}^{-1} = p_{L_1/K_1}.$$
	
	\textit{Step 2:} Our second and final step is to deduce that, given an apartment $A(G_K, S)$ in $\mc B(G, K)$, we have $$p_{L_1/K_1} \circ p_{K_1/K}(A(G_K, S)) \subset p_{L_1/L}(A(G_L, T)),$$ for some maximal $L$-split torus $T$ of $G_L$ such that $S_L \subset T$.
	First of all, it follows from Proposition \ref{strict_henselization_embedding} and the construction of $p_{L_1/K_1}$ that the image $p_{L_1/K_1} \circ p_{K_1/K}(A(G_K, S))$ lies in the apartment $A(G_{L_1}, T_1)$ associated a certain maximal $L_1$-split torus $T_1$ containing $S$ and is $\Gamma'_1$-equivariant.
	Consequently, we have the inclusion $$p_{L_1/K_1} \circ p_{K_1/K}(A(G_K, S)) \subset (\mc B(Z_{G_K}(S), G, L_1))^{\Gamma'_1},$$ with the notations of the Appendix.
	By Proposition \ref{strict_henselisation_embedding_levi}, we deduce $$p_{L_1/K_1} \circ p_{K_1/K}(A(G_K, S)) \subset p_{L_1/L}(\mc B(Z_{G_K}(S), G, L)).$$ Let $x \in A(G_K, S)$. The above inclusion implies that there exists a maximal $L$-split torus $T$ of $G_L$ containing $S_L$ such that $$p_{L_1/K_1} \circ p_{K_1/K}(x) \in p_{L_1/L}(A(G_L, T))$$
	Because the right-hand side is stable under the action of $S(K) \subset T(L)$, we get that $$p_{L_1/K_1} \circ p_{K_1/K}(S(K) \cdot x) \subset p_{L_1/L}(A(G_L, T))$$ and, passing to the convex hull, $$p_{L_1/K_1} \circ p_{K_1/K}(A(G_K, S)) \subset p_{L_1/L}(A(G_L, T)).$$
	We then deduce by the same arguments as before that there must exist a $G(K)$-equivariant map $$p_{L/K}: \mc B(G, K) \to \mc B(G, L)$$ such that $p_{L_1/K_1}\circ p_{K_1/K} = p_{L_1/L} \circ p_{L/K}$. Moreover, we have $p_{L/K}(A(G_K, S)) \subset A(G_L, T)$ and, by $G(K)$-equivariance, we find that for each $S$, there must exist a maximal $L$-split torus $T$ containing $S_L$ such that $p_{L/K}(A(G_K, S)) \subset A(G_L, T)$. The same argument as before then proves that the restriction of $p_{L/K}$ to $A(G_K, S)$ is affine, which establishes that $p_{L/K}$ satisfies condition 1. 
	To complete the proof of Theorem \ref{Functoriality_inclusions}, we now need to establish conditions 3 and 4.
	
	Condition 4 is an easy consequence of the uniqueness conditions. Indeed, if $K \subset L \subset M$ with $M/k$ a finite separable extension, then we have a commutative diagram 
	\begin{center}
		\begin{tikzcd}
			\mc B(G, K) \arrow[rr, "p_{K_2/K_1} \circ p_{K_1/K}"]  \arrow[d, "p_{L/K}"] &  & \mc B(G, K_2) \arrow[d, "p_{L_2/K_2}"]  \\
			\mc B(G, L) \arrow[rr, "p_{L_2/L_1} \circ p_{L_1/L}"] \arrow[d, "p_{M/L}"] & & \mc B(G, L_2) \arrow[d, "p_{M_2/L_2}"] \\
			\mc B(G, M) \arrow[rr, "p_{M_2/M_1} \circ p_{M_1/M}"]  & & \mc B(G, M_2)
		\end{tikzcd}.
	\end{center}
	Furthermore, because $G_{K_2}$ is pseudo-split, we know by Theorem \ref{Functoriality_iso} that $p_{M_2/L_2} \circ p_{L_2/K_2} = p_{M_2/K_2}$.
	Therefore, we have the commutative diagram 
	\begin{center}
		\begin{tikzcd}
			\mc B(G, K) \arrow[rr, "p_{K_2/K_1} \circ p_{K_1/K}"]  \arrow[d, "p_{M/L} \circ p_{L/K}"]  & & \mc B(G, K_2) \arrow[d, "p_{M_2/K_2}"] \\
			\mc B(G, M) \arrow[rr, "p_{M_2/M_1} \circ p_{M_1/M}"]  & & \mc B(G, M_2)
		\end{tikzcd}
	\end{center}
	Because $p_{M/K}$ is, by construction, the only map $\mc B(G, K) \to \mc B(G, M)$ that makes the above rectangle commute, we must have $$p_{M/K} = p_{M/L} \circ p_{L/K}.$$
	Lastly, establishing condition 3 merely requires adapting the Galois-equivariance arguments we previously brought forth. Indeed, if $L/K$ is Galois, then $L_1$ and $L_2$ are Galois over $K$. Furthermore, setting $$\tilde \Gamma' = \Gal(L_2/K), \tilde \Gamma_1 = \Gal(L_1/K), \Sigma = \Gal(L/K),$$ we have a diagram 
	\begin{center}
		\begin{tikzcd}
			\Gamma \arrow[rr, twoheadrightarrow, "\res_{K_2/K_1}"] & & \Gamma_1 & & \\
			\tilde \Gamma' \arrow[rr, twoheadrightarrow, "\res_{L_2/L_1}"] \arrow[u, twoheadrightarrow, "\res_{L_2/K_2}"] & & \tilde \Gamma'_1 \arrow[u, twoheadrightarrow, "\res_{L_1/K_1}"] \arrow[rr, twoheadrightarrow, "\res_{L_1/L}"] & & \Sigma
		\end{tikzcd}.
	\end{center}
	Recall that, by Theorem \ref{Galois_action_building}, we have the following actions 
	\begin{center}
		\begin{tikzcd}
			\mc B(G, K) \arrow[d, "p_{L/K}"] \arrow[r, "p_{K_1/K}"] & \mc B(G, K_1) \arrow[r, "p_{K_2/K_1}"] \arrow[loop above, "\Gamma_1"] & \mc B(G, K_2) \arrow[d, "p_{L_2/K_2}"] \arrow[loop above, "\Gamma"]\\
			\mc B(G, L) \arrow[loop below, "\Sigma"] \arrow[r, "p_{L_1/L}"] & \mc B(G, L_1) \arrow[loop below, "\tilde \Gamma'_1"] \arrow[r, "p_{L_2/L_1}"] & \mc B(G, L_2) \arrow[loop below, "\tilde \Gamma'"] 
		\end{tikzcd}.
	\end{center}
	Then, repeated applications of Proposition \ref{prop:special_galois_eq} yield $$p_{L_2/K_2} \circ p_{K_2/K_1} \circ p_{K_1/K}(\mc B(G, K)) \subset \mc B(G, L_2)^{\tilde \Gamma'}.$$
	By commutativity of the diagram, we get $$p_{L_2/L_1} \circ p_{L_1/L} \circ p_{L/K}(\mc B(G, K)) \subset \mc B(G, L_2)^{\tilde \Gamma'}.$$
	More precisely, we may write $$p_{L_2/L_1} \circ p_{L_1/L} \circ p_{L/K}(\mc B(G, K)) \subset \mc B(G, L_2)^{\tilde \Gamma'}\cap p_{L_2/L_1} \circ p_{L_1/L}(\mc B(G, L)).$$
	Then, using Proposition \ref{prop:special_galois_eq} and the injectivity of $p_{L_2/L_1} \circ p_{L_1/L}$, we get $$p_{L_2/L_1} \circ p_{L_1/L} \circ p_{L/K}(\mc B(G, K)) \subset p_{L_2/L_1} \circ p_{L_1/L}(\mc B(G, L)^{\Sigma}),$$ and finally by injectivity $$p_{L/K}(\mc B(G, K)) \subset \mc B(G, L)^{\Sigma},$$ which completes the proof.
\end{proof}

\begin{rmk}
    We used Proposition \ref{strict_henselisation_embedding_levi} in Step 2 of the existence proof of Theorem \ref{Functoriality_inclusions} to ensure that $p_{L/K}$ maps the apartment into an apartment $A(G_L, T)$ for a torus $T \subset G_L$ \textit{that contains $S_L$}. At that stage of the proof, we could have argued directly from Proposition \ref{strict_henselization_embedding} that $p_{L_1/K_1} \circ p_{K_1/K}(A(G_K, S))$ lies in $\mc B(G, L_1)^{\Gamma'_1} = p_{L_1/L}(\mc B(G, L))$, and thus that, by convexity, the image lies in the image of some apartment $p_{L_1/L}(A(G_L, T))$. However, the torus $T$ need not contain $S_L$.
\end{rmk}

We conclude this section by generalising two results that held in the two special cases of Theorem \ref{Functoriality_inclusions} we studied earlier.

First of all, we prove that, given an apartment $A$ in a building $\mc B(G, K)$, there exists a Galois extension $L/K$ such that $G_L$ is pseudo-split and such that $A$ is the set of fixed points of an apartment $A_L \subset \mc B(G, L)$ that is stable under the action of $\Gal(L/K)$.

\begin{prop}\label{prop:special_apartments}
	Let $K$ be a finite separable extension of $k$ and let $A(G, S)$ be an apartment in $\mc B(G, K)$. There exists a maximal torus $T$ of $G$ containing $S$ and a finite Galois extension $L/K$ satisfying the following:
	\begin{enumerate}
		\item The torus $T_L$ is split.
		\item The injection $p_{L/K}: \mc B(G, K) \to \mc B(G, L)$ maps $A(G, S)$ isomorphically onto the set of $\Gal(L/K)$-fixed points of $A(G_L, T_L)$.
	\end{enumerate}
\end{prop}

\begin{proof}
	Let $K^{\nr}$ be the strict henselisation of $K$ and $\tilde K^{\nr}$ be the minimal Galois extension that splits $G_{K^{\nr}}$. By \cite[Lemme 4.8]{Lou}, there exists a maximal $K^{\nr}$-split torus $S^{\nr}$ that contains $S_{K^{\nr}}$ and is defined over $K$. Because $G_{K^{\nr}}$ is quasi-split, there exists a unique maximal torus $T$ of $G_{K^{\nr}}$ that contains $S^{\nr}$. The torus $T$ is defined over $K$, because $S^{\nr}$ is, and splits over $\tilde K^{\nr}$. Moreover, Proposition \ref{strict_henselization_embedding} yields $$p_{K^{\nr}/K}(A(G, S)) \subset A(G_{K^{\nr}}, S^{\nr})$$ and Proposition \ref{quasisplit_to_split_embedding} yields $$p_{\tilde K^{\nr}/K^{\nr}}(A(G_{K^{\nr}}, S^{\nr})) \subset A(G_{\tilde K^{\nr}}, T_{\tilde K^{\nr}}).$$
	In particular, $$p_{\tilde K^{\nr}/K}(A(G, S)) \subset A(G_{\tilde K^{\nr}}, T_{\tilde K^{\nr}}).$$
	Now, let $L \subset \tilde K^{\nr}$ be a finite Galois extension of $K$ such that $T_L$ is split. By Theorem \ref{Functoriality_iso}, we have $$p_{\tilde K^{\nr}/L}(A(G_L, T_L)) = A(G_{\tilde K^{\nr}}, T_{\tilde K^{\nr}}).$$
	In other words, $$p_{\tilde K^{\nr}/L} \circ p_{L/K}((A(G, S)) \subset p_{\tilde K^{\nr}/L}(A(G_L, T_L))$$ and, by injectivity, we have $$p_{L/K}(A(G, S)) \subset A(G_L, T_L).$$
	Because $T$ is defined over $K$, the apartment $A(G_L, T_L)$ is stable under the action of $\Gal(L/K)$ and the set of $\Gal(L/K)$-fixed points is a non-empty affine space under $V(G_L, T_L)^{\Gal(L/K)} = V(G, S)$. By considering dimensions, we thus see that $$p_{L/K}(A(G, S)) = A(G_L, T_L)^{\Gal(L/K)}.$$
\end{proof}

Secondly, we prove a general version of Proposition \ref{prop:special_galois_eq}.

\begin{thm}\label{thm:full_galois_eq}
	Let $K, L, M$ be finite separable extensions of $k$ such that $K \subset L \subset M$. Assume that $L$ and $M$ are Galois over $K$. Then, the map $p_{M/L}$ is $\Gal(M/K)$-equivariant if we let $\Gal(M/K)$ act on $\mc B(G, L)$ through the restriction map $\res_{M/L}: \Gal(M/K) \to \Gal(L/K)$.
\end{thm}

\begin{proof}
	The argument is a combination of Proposition \ref{prop:special_galois_eq} and some ideas of the proof of Theorem \ref{thm:full_galois_eq}. Recall that, by construction, the map $p_{M/L}$ is the only map that makes the following diagram commute 	
	
	\begin{center}
		\begin{tikzcd}
			\mc B(G, L) \arrow[d, "p_{M/L}"] \arrow[r, "p_{L_1/L}"]  & \mc B(G, L_1)  \arrow[r, "p_{L_2/L_1}"] & \mc B(G, L_2)  \arrow[d, "p_{M_2/L_2}"] \\
			\mc B(G, M) \arrow[r, "p_{M_1/M}"] & \mc B(G, M_1)  \arrow[r, "p_{M_2/M_1}"] & \mc B(G, M_2) 
		\end{tikzcd},
	\end{center} where, as in the proof of Theorem \ref{Functoriality_inclusions}, we let $$L_1 = L^{\nr}, L_2 = \tilde{L^{\nr}}, M_1 = M^{\nr}, M_2 = \tilde{M^{\nr}}.$$ 
	Note that all extensions in question are Galois over $K$ under the assumption that $L$ and $M$ are Galois over $K$.
	Then, letting $\gamma' \in \Gal(M/K)$ and denoting by $\gamma'_1  \in \Gal(M_1/K)$ a lift of $\gamma'$, and by $\gamma'_2 \in \Gal(M_2/K)$ a lift of $\gamma'_1$, and letting $\gamma_2 = \res^{M_2}_{L_2} (\gamma'_2) \in \Gal(L_2/K)$, as well as $\gamma_1 = \res^{L_2}_{L_1}(\gamma_2) \in \Gal(L_1/K)$, and finally $\gamma = \res^{L^{\nr}}_L(\gamma) \in \Gal(L/K)$, repeated applications of Proposition \ref{prop:special_galois_eq} yield: $$\begin{array}{ccl} p_{M_2/M_1} \circ p_{M_1/M} \circ p_{\gamma'} \circ  p_{M/L} & = & p_{M_2/M_1} \circ p_{\gamma'_1} \circ p_{M_1/M} \circ p_{M/L}  \\  & = & p_{\gamma'_2} \circ p_{M_2/M_1} \circ p_{M_1/M} \circ p_{M/L} \\ & = & p_{\gamma'_2} \circ p_{M_2/L_2} \circ p_{L_2/L_1} \circ p_{L_1/L} \\ & = & p_{M_2/L_2} \circ p_{\gamma_2} \circ p_{L_2/L_1} \circ p_{L_1/L}\\ & = & p_{M_2/L_2}\circ p_{L_2/L_1} \circ p_{\gamma_1} \circ p_{L_1/L} \\ & = & p_{M_2/L_2} \circ p_{L_2/L_1} \circ p_{L_1/L} \circ p_{\gamma}\end{array}.$$ 
	It follows that $p_{\gamma'} \circ p_{M/L} \circ p_{\gamma}$ makes the above rectangle commute, hence we must have $$p_{\gamma'} \circ p_{M/L} \circ p_{\gamma} = p_{M/L}.$$
\end{proof}
\section{Mapping the building into \texorpdfstring{$G^{\an}$}{Gan}} \label{sect:Embedding}

Let $G$ be a quasi-reductive group over a local field $k$. We denote its valuation by $\omega$ and by $|\cdot|$ the absolute value $| \cdot | : x \mapsto e^{-\omega(x)}$. In this section, we construct an injection $$\vtheta: \mc B(G, k) \to G^{\an}$$ following the general philosophy of \cite{RTW1}. Recall that the idea is to associate to each special vertex $x$ the unique point $\vtheta(x)$ in the Shilov boundary of $(\mc G_x^{\circ})^{\an}$, where $\mc G_x^{\circ}$ is the parahoric subgroup associated to $x$ in \cite[Théorème 4.5]{Lou}. In the case where the point $x$ is not a special vertex, its image is $\pr_{K/k}(\vtheta_K(x_K))$, where $K$ is a suitably chosen non-archimedean extension of $k$ such that $x_K$ becomes a special vertex in $\mc B(G, K)$. 

In our case, because we are reliant on Bruhat-Tits theory for quasi-reductive groups as exposed in \cite{Lou}, which requires discretely valued base fields with perfect residue fields, we cannot make sense of the above definition unless the point $x$ becomes a special vertex in $\mc B(G, K)$ for a \textit{finite} Galois extension $K/k$ (for the simple reason that we do not have a working definition of $\mc B(G, K)$ for an arbitrary non-archimedean extension $K/k$).

The main observation we make here is that there is a large enough amount of such points that embedding them determines a unique embedding of the whole building by a density argument.

\subsection{Virtually special points in the building}

We first introduce the notion of a virtually special point of the building, which formalises the aforementioned condition, and collect some basic facts. We conclude by proving that virtually special points are dense in the building.

\begin{defn}
	Let $x \in \mc B(G, k)$. We say that $x$ is \textit{virtually special} if there exists a finite Galois extension $K/k$ such that $G_K$ is pseudo-split and the canonical injection $$p_{K/k}: \mc B(G, k) \to \mc B(G, K)$$ defined in Theorem \ref{Functoriality_inclusions}  maps $x$ to a special vertex. We denote by $\mc B_{\QQ}(G, k)$ the set of virtually special points in $\mc B(G, k)$. If $S \subset G$ is a maximal split torus, we denote by $A_{\QQ}(G, S)$ the set of virtually special points in the apartment $A(G, S)$.
\end{defn}

For ease of notation, whenever $K$ is a finite Galois extension of $k$ and $x\in \mc B(G, k)$, we will write $x_K$ instead of $p_{K/k}(x)$ in all the following.

The following proposition is a direct consequence of the definition.

\begin{prop}\label{prop:virtually_special_and_extensions}
	If $k'/k$ is a finite Galois extension, then we have $$\mc B_{\QQ}(G, k) = p_{k'/k}^{-1}(\mc B_{\QQ}(G, k')).$$
\end{prop}

We also check that the set of virtually special points is preserved under the action of $G(k)$.

\begin{prop}
	The action of $G(k)$ on $\mc B(G, k)$ stabilises $\mc B_{\QQ}(G,k)$.
\end{prop}

\begin{proof}
	Let $x \in \mc B_{\QQ}(G,k)$ and $g \in G(k)$. Let $K/k$ be a finite Galois extension such that $x_K$ is a special vertex. Because $G(K)$ acts on $\mc B(G, K)$ by simplicial automorphisms, the point $gx_K$ is a special vertex. Moreover, the map $p_{K/k}$ is $G(k)$-equivariant, hence $gx_K = p_{K/k}(gx)$, which concludes.
\end{proof}

To conclude these introductory observations, we check that the set of virtually special points is preserved under the action of a Galois group.

\begin{prop}\label{Galois_rational}
	If $k'/k$ is a finite Galois extension, then the action of $\Gal(k'/k)$ stabilises $\mc B_{\QQ}(G, k')$.
\end{prop}

\begin{proof}
	Let $x  \in \mc B_{\QQ}(G, k')$ and $\sigma \in \Gal(k'/k)$. Let $K$ be a finite Galois extension of $k$ containing $k'$ such that $G_K$ is pseudo-split and $x_K$ is a special vertex. Let $\tilde\sigma \in \Gal(K/k)$ be an automorphism of $K$ that extends $\sigma$. Then, because the action of $\Gal(K/k)$ on $\mc B(G,K)$ preserves the simplicial structure, the point $\tilde\sigma(x_K)$ is special. Moreover, by Theorem \ref{thm:full_galois_eq}, we have $\tilde \sigma(x_K) = \sigma(x)_K$.  Hence, the point $\sigma(x)$ is virtually special.
\end{proof}

Our goal is now to prove the following proposition.
\begin{prop}\label{Density_rational}
	The subset $\mc B_{\QQ}(G, k)$ is dense in $\mc B(G, k)$. More precisely, for any facet $\mbf f \subset \mc B(G, k)$, the subset $\mbf f \cap \mc B_{\QQ}(G, k)$ is dense in $\mbf f$.
\end{prop}

This key result will be established in two steps. We first treat the case where $G$ is pseudo-split over $k$. This relies on proving that, for an appropriate choice of coordinate system, the set of virtually special points in an apartment $A(G, T)$ is precisely the set of points with rational coordinates.
Then, we treat the general case by exploiting the pseudo-split case. Letting $K/k$ be a suitable finite Galois extension such that $G_K$ is pseudo-split, we prove that, if $x$ is a vertex in $\mc B(G, k)$, then $x_K$ is the centroid of the facet $\mbf f$ of $\mc B(G, K)$ that contains it. In particular, the point $x_K$ is a convex combination of vertices of $\mc B(G, K)$ with rational coefficients, hence virtually special.

\begin{proof}[Step 1: The pseudo-split case] Assume that $G$ is pseudo-split over $k$. Let $A = A(G, T)$ be an apartment in $\mc B(G, k)$ and $o$ be a special vertex in $A$. Denote by $\Phi$ the root system $\Phi(G,T)$ associated to $T$. Recall from Proposition \ref{prop:pseudo_split_vertices} that, if $K/k$ is a finite Galois extension, then the canonical injection $p_{K/k}: \mc B(G, k) \to \mc B(G, K)$ identifies $A$ with the apartment $A_K = A(G_K,T_K)$ in such a way that special vertices in $A$ are mapped to special vertices in $A_K$. In particular, the point $o_K$ is a special vertex of $A_K$ and the affine roots in $A_K$ are the $$\alpha_{a_K, \gamma}^{o_K} =  \{x \in A_K, \langle a_K, x - o_K \rangle + \gamma \ge 0\}$$ for $a \in \Phi$ and $\gamma \in \Gamma'_{a_K}$. It then follows from \cite[1.3.7]{BT1} that a special vertex of $A_K$ is a point $x$ such that, for each root $a \in \Phi$, there exists either $\gamma \in \Gamma'_{a_K}$ such that $x \in \partial \alpha^{o_K}_{a_K,\gamma}$ or a $\gamma \in \Gamma'_{2a_K}$ such that $x \in \alpha^{o_K}_{2a_K, \gamma}$. 
	
	Note also that, because $G$ is pseudo-split over $k$, then for each root $a \in \Phi$, we have $$\tilde G^{a_K} = \left\{ \begin{array}{cl} R_{K \otimes_k k_a/K}(SL_{2, K\otimes_k k_a}) & \text{if } a \text{ is not multipliable}\\ R_{K \otimes_k k_{2a}/K}(BC_{1, K \otimes_k k_{2a}}) & \text{if } a \text{ is multipliable}\end{array}\right. ,$$ where $k_a/k$ (resp. $k_{2a}/k$) is purely inseparable. Then, by the proof of \cite[Lemme 4.4]{Lou}, for each $a \in \Phi$, we have $$\Gamma'_{a_K} = \left\{ \begin{array}{cl} \omega((K \otimes_k k_a)^{\times}) & \text{if } a \text{ is not multipliable} \\ \frac{1}{4} \omega((K \otimes_k k_{2a})^{\times}) \setminus \frac{1}{2}\omega((K \otimes_k k_{2a})^{\times}) & \text{if } a \text{ is multipliable} \end{array}\right. .$$
	Note in particular that, letting $e(K/k)$ be the ramification index of the extension $K/k$, then $$\Gamma'_{a_K} = \frac{1}{e(K/k)}\Gamma'_a$$ and that, if $a$ is not multipliable, then $$\Gamma'_a = \omega(k_a^{\times}) = \left\{\begin{array}{cl} \ZZ & \text{if } [k_a:k] = 1 \\ \frac{1}{p^{m_a}}\ZZ & \text{otherwise}\end{array}\right. ,$$ where $p = \mathrm{char}(k)$ and $m_a$ is the least positive integer $m$ such that $ x^{p^m} \in k$ for each $x \in k_a$.
	
	We are now in position to prove that virtually special points are exactly those points $x \in A$ such that $\langle a, x-o \rangle \in \QQ$ for each $a \in \Phi$. First, by inspection, we see that for each finite Galois extension $K/k$ and each root $a \in \Phi$, we have $\Gamma'_{a_K} \cup \Gamma'_{2a_K} \subset \QQ$. In particular, if $x \in A$ is mapped by $p_{K/k}$ to a special vertex in $A_K$ for some finite Galois extension $K/k$, then we have $$\langle a, x-o \rangle = \langle a_K, x_K - o_K \rangle \in \QQ $$ for each root $a \in \Phi$. Conversely, if $\langle a, x-o \rangle \in \QQ$ for each $a \in \Phi$, then there exists $N$ such that, for each non-multipliable $a$, $\langle a, x-o \rangle \in \frac{1}{N} \Gamma'_a$ (the key point being that, when $a$ is non-multipliable, the subset $\Gamma'_a$ is a nontrivial subgroup of $\QQ$). Then, letting $K/k$ be a totally ramified extension of degree $N$, we immediately check that $x_K$ is a special vertex in $A_K$.

	Letting $\Delta$ be a basis for $\Phi$, it follows that the virtually special points in $A$ are exactly the points $x$ with rational coordinates in $A$, setting $o$ as an origin and within the coordinate system given by the dual basis of $\Delta$. It is clear from this characterisation that the set of virtually special points in $A$ is dense in $A$ and closed under convex combinations with rational weights. 
	
	We conclude this discussion by proving that furthermore the virtually special points in any facet are dense in that facet. First of all, any vertex $x$ is a virtually special point. Indeed, it must satisfy $\langle a, x-o \rangle \in \Gamma'_a$ for $a$ ranging in a subset of $\Phi$ that forms a basis for $\Span_{\QQ}(\Phi)$, and therefore we must have $\langle a, x - o\rangle \in \QQ$ for all $a \in \Delta$. Next, if $\mbf f$ is an arbitrary facet of $A$, the extremal points $\{x_1, \dots, x_r\}$ of $\ol{\mbf f}$ are vertices of $A$, and thus virtually special points. Then, the convex combinations of the $x_i$ with positive rational coefficients form a dense subset of $\mbf f$ comprised of virtually special points.
\end{proof}

\begin{proof}[Step 2: The general case]
    We now let $G$ be arbitrary. As in the first step, we prove that vertices are virtually special. Fix a vertex $x$ in an apartment $A(G, S)$ of $\mc B(G, k)$. Recall from Proposition \ref{prop:special_apartments} that there exists a Galois extension $K/k$ and a maximal torus $T$ of $G$ containing $S$ such that $$p_{K/k}(A(G, S)) = A(G_K, T_K)^{\Gal(K/k)}.$$
    Denote by $\mbf f \subset A(G_K, T_K)$ the facet of $\mc B(G, K)$ that contains $x_K$. Because the action of the group $G(K)$ preserves the polysimplicial structure and because $p_{K/k}$ is $G(k)$-equivariant, the stabiliser $G(k)_x$ stabilises the facet $\mbf f$. Further, the group $G(k)_x$ stabilises the set $\{z_1, \dots, z_r\}$ of vertices of $\mbf f$. In particular, the centroid $\gamma = \frac{1}{r}\sum_{i=1}^r z_i$ is fixed by $G(k)_x$. Indeed, the group $G(k)_x$ acts isometrically on $\mc B(G, K)$ hence, for any $g \in G(k)_x$, the map $$g: A(G_K, T_K) \to gA(G_K, T_K)$$ induced by $g$ is an isometry and thus affine, hence preserves convex combinations.
    
    On the other hand, observe that the Galois group $\Gamma = \Gal(K/k)$ acts on $A(G_K, T_K)$ by affine automorphisms that preserve the polysimplicial structure, and fixes $x_K$. Consequently, the group $\Gamma$ preserves the facet $\mbf f$ and, by similar arguments, also fixes the centroid $\gamma$. Hence, there exists $y \in A(G, S)$ such that $\gamma = y_K$. By $G(k)$-equivariance and injectivity of $p_{K/k}$, the group $G(k)_x$ must then fix the point $y$. Because $x$ is a vertex, it is the only point of $\mc B(G, k)$ fixed by $G(k)_x$ \cite[Proposition 3.10 (a)]{Sol}, hence $x = y$ and we have $$x_K = \frac{1}{r}\sum_{i=1}^r z_i.$$
    By Step 1, each of the $z_i$ is virtually special, hence, as a convex combination of the $z_i$ with rational weights, so is $x_K$. By Proposition \ref{prop:virtually_special_and_extensions}, we deduce that $x$ is virtually special. We then deduce that all convex combinations of vertices with rational coefficients are virtually special points, which proves that virtually special points are dense in every facet (and in particular in $A(G, S)$).
\end{proof}

We conclude with this transitivity property:

\begin{prop}\label{Transitivity}
	Let $x, y \in \mc B_{\QQ}(G, k)$. Then, there exists a finite Galois extension $K/k$, a split maximal torus $T$ in $G_K$, and $t \in T(K)$ such that $y_K = tx_K$.
\end{prop}

\begin{proof}
	We may and do assume that $G$ is pseudo-split over $k$.
	
	Let $A(G, T)$ be an apartment of $\mc B(G, k)$ that contains $x$ and $y$. Let $\Delta$ be a basis for $\Phi(G, T)$. Then, it follows from the proof of Proposition \ref{Density_rational} that we have $\langle a, y-x \rangle \in \QQ$ for each $a \in \Delta$. Let $N \in \mathbb N$ such that $$\forall a \in \Delta, \langle a, y-x \rangle \in \frac{1}{N} \omega(k^{\times}).$$ Then, if $K/k$ is a totally ramified Galois extension of degree $ND$, where $D$ is the determinant of the Cartan matrix $(\langle a, (a')^{\vee}\rangle)_{a, a' \in \Delta}$ of $\Phi(G,T)$, there exists a $t \in T(K)$ such that $$\forall a \in \Delta, \langle a, y_K-x_K \rangle = - \omega(a(t))$$ and therefore $y_K = tx_K$.
	Indeed, letting $a^{\vee} \in X_*(T)$ denote the coroot associated to $a$ for each $a \in \Delta$, we may look for $t$ of the form $\prod_{a \in \Delta} a^{\vee}(\varpi_{K})^{n_a} $ for $n_a \in \ZZ$.
	Then, the condition $y_K = tx_K$ is equivalent to the condition $$\forall a \in \Delta, \sum_{a' \in \Delta} \langle a, (a')^{\vee} \rangle n_{a'} = \frac{ND}{\omega(\varpi)} \langle a, y_K - x_K \rangle \in D\ZZ .$$
	This system has a unique solution $(n_a)_{a \in \Delta} \in \ZZ^{\Delta}$, which completes the proof.
	
\end{proof}

Note that, in the case when $G$ is pseudo-split over $k$, the proof yields a better result.

\begin{cor}\label{Transitivity_split}
	Assume that $G$ is pseudo-split over $k$. Let $x$ and $y$ in $\mc B_{\QQ}(G, k)$. Then, for any split maximal torus $T$ such that the apartment $A(G, T)$ contains $x$ and $y$, there exists a finite Galois extension $K/k$ and an element $t \in T(K)$ such that $y_K = tx_K$.
\end{cor}

\subsection{Mapping the virtually special points}

We now construct a canonical $G(k)$-equivariant map $$\vtheta: \mc B_{\QQ}(G, k) \to G^{\an}$$

The idea is to associate to each point its "analytic stabiliser", defined as follows: Let $K/k$ be a finite Galois extension such that $G_K$ is pseudo-split and such that $x_K$ is special. Denote by $\mc G_{x_K}^{\circ}$ the parahoric $K^{\circ}$-model of $G_K$ associated to the facet $\{x_K\}$, as defined in \cite[4.2]{Lou}. Then, the Berkovich analytification $(\mc G_{x_K}^{\circ})^{\an}$ constructed in \cite[1.2.4]{RTW1} is a strictly $K$-affinoid subgroup of $G^{\an}$ and, by \cite[Proposition 1.3]{RTW1}, the projection $G_x := \pr_{K/k}((\mc G_{x_K}^{\circ})^{\an})$ is a $k$-affinoid subgroup of $G$. This subgroup does not depend on the choice of the extension $K$. Indeed, if $K$ and $K'$ are two such extensions and $L$ a common finite Galois extension, then it follows from \cite[Théorème 4.5]{Lou} that $$\mc G_{x_L}^{\circ} = \mc G_{x_K}^{\circ} \otimes_{K^{\circ}}{L^{\circ}} = \mc G_{x_{K'}}^{\circ}\otimes_{K^{\circ}}{L^{\circ}}$$ and therefore $$(\mc G_{x_L}^{\circ})^{\an} = (\mc G_{x_K}^{\circ})^{\an} \hat\otimes_K L = (\mc G_{x_{K'}}^{\circ})^{\an} \hat \otimes_{K'} L$$ and finally $$\pr_{L/k} ((\mc G_{x_L}^{\circ})^{\an}) = \pr_{K/k}(\pr_{L/K}((\mc G_{x_L}^{\circ})^{\an})) = \pr_{K/k}((\mc G_{x_K}^{\circ})^{\an})$$ and, by the same argument $$\pr_{L/k} ((\mc G_{x_L}^{\circ})^{\an}) = \pr_{K'/k} ((\mc G_{x_{K'}}^{\circ})^{\an}).$$

Then, we have the following proposition:

\begin{prop}\label{Shilov}
	Let $x$ be a point in $\mc B_{\QQ}(G, k)$.  The $k$-affinoid subgroup $G_x$ has a unique Shilov boundary point, which we denote by $\vtheta(x)$.
\end{prop}

\begin{proof}
	Consider a finite Galois extension such that $G_K$ is pseudo-split and $x_K$ is a special vertex.
	
	By construction, the $K$-analytic subgroup $G_x \hat\otimes_k K$ is the analytification of the $K^{\circ}$-group scheme $\mc G = \mc G_{x_K}^{\circ}$ and, as such, is strictly $K$-affinoid. Because $\mc G$ is smooth, it follows from \cite[1.2.4]{RTW1} and \cite[2.4.4]{Ber} that the reduction map $$\rho: \mc G^{\an} \to \mc G_{\tilde K}$$ establishes a one-to-one correspondence between the Shilov boundary of $G_x \hat\otimes_k K$ and the generic points of the irreducible components of the reduction $\mc G_{\tilde K}$.
	
	As $\mc G_{\tilde K}$ is connected, the strictly $K$-affinoid domain $G_x \hat\otimes_k K$ has only one Shilov boundary point and the same goes for $G_x$ as the projection $\pr_{K/k}$ maps the Shilov boundary of $G_x\hat \otimes_k K$ onto that of $G_x$ \cite[Proof of Proposition 2.4.5]{Ber}.
	
\end{proof}

\begin{prop}\label{Equiv_theta}
	For any $x \in B_{\QQ}(G, k)$ and $g \in G(k)$, we have $$G_{gx} = gG_x g^{-1} \text{ as well as } \vtheta(gx) = g\vtheta(x) g^{-1}.$$
\end{prop}

\begin{proof}
	Let $K/k$ be a finite Galois extension such that $G_K$ is pseudo-split and $x_K$ is a special vertex of $\mc B(G,k)$. Denote by $\mc G_{x_K}$ the parahoric $K^{\oo}$-model of $G$ associated to $x_K$. Then, the automorphism $\inta(g)$ of $G$ induces an isomorphism from $\mc G_{x_K}$ to a $K^{\circ}$-model $\mc H$ of $G$. We prove that $\mc H = \mc G_{g x_K}$, which will establish that the analytic automorphism $\inta(g)^{\an}$ of $G^{\an}$ maps $G_x$ to $G_{gx}$.

	The group $\mc H$ is smooth, connected, affine and, for each finite unramified extension $K'/K$, we have $$\mc H(K'^{\circ}) = g P_{x_{K'}}^{\circ} g^{-1}.$$
	By \cite[Théorème 4.5]{Lou}, it is sufficient to prove that $gP_{x_{K'}}^{\circ}g^{-1} = P_{gx_{K'}}^{\circ}$ for each unramified extension $K'/K$ to conclude. Let $T$ be a split maximal $K$-torus such that the apartment $A(G_K, T)$ contains $x_K$ and $gx_K$. The Bruhat decomposition \cite[Theorem 3.14 (a)]{Sol} rewrites as $$G(K) = P_{gx_K}N_G(T)(K) P_{x_K}.$$ By construction, we have $$N_G(T)(K)P_{x_K} = N_G(T)(K) P^{\circ}_{x_K} = P_{x_K}^{\circ}N_G(T)(K)$$ and therefore we also have $G(K) = P_{gx_K}^{\circ}N_G(T)(K) P_{x_K}^{\circ}$. Further, because $x_K$ is special, we may write $G(K) =  P_{gx_K}^{\circ}Z_G(T)(K) P_{x_K}^{\circ}$. Write $g = pzp'$, with $p \in P_{gx_K}^{\circ}$, $p'\in P_{x_K}^{\circ}$, $z\in Z_G(T)(K)$.
	We then have $gx = zx$ and it is sufficient to check that $zP_{x_K}^{\circ}z^{-1} = P^{\circ}_{zx_K}$. 
	
	First of all, for any $a \in \Phi(G_K,T)$ and $r \in \RR$, we have $zU_{a,r}z^{-1} = U_{a, r - \omega(a(z))}$\footnote{We think of $X^*(T)$ as included in $X^*(Z_G(T)) \otimes_{\ZZ} \QQ$}, and thus $$zU_{a, x_K}z^{-1} = U_{a,z\cdot x_K} \text{ as well as } zU_{x_K} z^{-1} = U_{z\cdot x_K}.$$ Consequently, $$z \mc Z^{\circ}(K^{\circ})(N_G(T)(K) \cap U_{x_K})z^{-1} = \mc Z^{\circ}(K^{\circ})(N_G(T)(K) \cap U_{z\cdot x_K}).$$ Because $P_{x_K}^{\circ}$ (resp. $P_{z\cdot x_K}^{\circ}$) is spanned by the $ U_{a,x_K}$ (resp. $U_{a, z\cdot x_K}$) and $\mc Z^{\circ}(K^{\circ})(N_G(T)(K) \cap U_{x_K})$ (resp.
	$\mc Z^{\circ}(K^{\circ})(N_G(T)(K) \cap U_{z\cdot x_K})$ ), we do have $zP_{x_K}^{\circ}z^{-1} = P^{\circ}_{z\cdot x_K}$.
	
	Because the same reasoning applies if we substitute any finite unramified extension $K'$ of $K$ for $K$ in the above proof, we find $\mc H = \mc G_{g \cdot x_K}^{\circ}$. Passing to Berkovich analytifications, the automorphism $\inta(g)^{\an}$ maps $G_x$ onto $G_{gx}$. As it is an analytic isomorphism, it also maps the unique Shilov boundary point of $G_x$ to that of $G_{gx}$, whence $$\vtheta(gx) = g\vtheta(x)g^{-1}.$$
\end{proof}

\begin{lem}\label{lem:Compatibility_extensions}
	Let $K/k$ be a finite Galois extension. We have the commutative diagram 
	\begin{center}
		\begin{tikzcd}
			\mc B_{\QQ}(G, K) \arrow[r, "\vtheta"] & G_K^{\an}\arrow[d, "\pr_{K/k}"] \\ \mc B_{\QQ}(G, k)\arrow[u, "p_{K/k}"] \arrow[r, "\vtheta"] & G^{\an}
		\end{tikzcd}.
	\end{center}
\end{lem}

\begin{proof}
	Let $x \in \mc B_{\QQ}(G, k)$ and $L/K$ be a finite Galois extension such that $G_L$ is pseudo-split and $x_L$ is a special point. Then, we have $$G_{x_K} = \pr_{L/K}((\mc G_{x_L}^{\circ})^{\an})$$ as well as $$G_x = \pr_{L/k}((\mc G_{x_L}^{\circ})^{\an}) = \pr_{K/k}(G_{x_K}).$$ Because the projection $\pr_{K/k}$ maps Shilov boundaries to Shilov boundaries, we have $\vtheta(x) = \pr_{K/k}(\vtheta(x_K))$.
\end{proof}

\subsection{Extending to the building}

Our goal is now to extend the map $\vtheta$ to the whole of $\mc B(G, k)$. Given that we have the commutative diagram established in Lemma \ref{lem:Compatibility_extensions}
\begin{center}
	\begin{tikzcd}
		\mc B_{\QQ}(G, K) \arrow[r, "\vtheta"] & G_K^{\an}\arrow[d, "\pr_{K/k}"] \\ \mc B_{\QQ}(G, k)\arrow[u, "p_{K/k}"] \arrow[r, "\vtheta"] & G^{\an}
	\end{tikzcd}
\end{center}
for each finite Galois extension $K/k$, it is sufficient to work in the case where $G$ is pseudo-split, which we henceforth assume.

We begin by establishing a formula for the norm $\vtheta(x)$ associated to a point of the building.

Let $T$ be a split maximal torus and let $\Phi = \Phi(G, T)$ be the root system corresponding to $T$. Consider a minimal pseudo-parabolic subgroup $P$ containing $T$ and let $\Phi^+ = \Phi(P, T)$ be the set of positive roots corresponding to $P$. Denote by $\Phi_{\nd}$ the set of non divisible roots in $\Phi$ and by $\Phi^+_{\nd} = \Phi^+ \cap \Phi_{\nd}$. Then the multiplication map yields an open immersion $$\prod_{\alpha \in \Phi^+_{\nd}} U_{-\alpha} \times Z_G(T) \times \prod_{\alpha \in \Phi^+_{\nd}} U_{\alpha} \hookrightarrow G$$
whose image is independent of the order on $\Phi^+_{\nd}$. We denote this image by $\Omega(T, P)$ and call it the open Bruhat cell corresponding to the pair $(T, P)$.

If $o$ is a special point in $A(G, T)$, then, $o$ determines integral models $\mc U_{a, o}$ for each $U_{a}$ and an integral model $\mc G$ for $G$ defined in \cite[4.2]{Lou}.  Letting $\mc Z^{\circ}$ be the open subgroup of the Néron model of $Z = Z_G(T)$ containing its generic fibre and the connected component of its special fibre, it follows from \cite[Théorème 4.5]{Lou} that the above open immersion descends to an open immersion of $k^{\circ}$-schemes $$\prod_{\alpha \in \Phi^+_{\nd}} \mc U_{-\alpha, o} \times \mc Z^{\circ} \times \prod_{\alpha \in \Phi^+_{\nd}} \mc U_{\alpha, o} \hookrightarrow \mc G.$$
We denote by $\Omega(\mc T, \mc P)$ the image of this immersion.

For each non-divisible root $\alpha \in\, \Phi_{\nd}$, we have an $\mc T$-equivariant isomorphism of $k^{\circ}$-schemes: $$\mc U_{\alpha} \simeq \mc U_{2{\alpha}} \times \mc U_{\alpha}/\mc U_{2{\alpha}},$$
where $\mc T = \Spec(k^{\oo}[X^*(T)]) \subset \mc Z^{\circ}$. Moreover, each $\mc U_{2{\alpha}}$ (resp. $\mc U_{\alpha}/\mc U_{2{\alpha}}$) is an affine space over $k^{\circ}$ on which $\mc T$ acts linearly via the character $2{\alpha}$ (resp. ${\alpha}$). 
Fix $\mc T$-equivariant isomorphisms $$u_{\alpha}: \mbb A^{N_{\alpha}}_{k^{\oo}} \overset{\sim}{\to} \mc U_{\alpha}/\mc U_{2{\alpha}} \text{ and } u_{2{\alpha}}: \mbb A^{N_{2{\alpha}}}_{k^{\oo}} \overset{\sim}{\to} \mc U_{2{\alpha}}$$ for each non-divisible root ${\alpha}$. These yield a $\mc T$-equivariant open immersion of $k^{\oo}$-schemes $$\prod_{{\alpha} \in\ \Phi^+}\mbb A^{N_{-\alpha}}_{k^{\oo}}  \times \mc Z^{\circ} \times \prod_{{\alpha} \in\ \Phi^{+}}\mbb A^{N_{\alpha}}_{k^{\oo}} \overset{\sim}{\to} \Omega(\mc T, \mc P) \subset \mc G.$$

Lastly, we denote by $\vtheta_Z(\ast)$ the only point in the Shilov boundary of $\mc Z^{\an}$, a multiplicative norm on $k[Z]$. The uniqueness of that point follows from Proposition \ref{Shilov} applied to $Z$. Note that the Bruhat-Tits building of the quasi-reductive group $Z$ \cite[Remark C.2.12 (1)]{CGP} reduces to a point, which makes this notation consistent with the above.

\begin{prop}\label{Formule_theta}
	Assume that $G$ is pseudo-split. We keep the above notation. Let $${\mbf \Phi} = \{(a, i) \in\, \Phi \times \NN, 1 \le i \le N_a\}.$$
	
	\begin{enumerate}
		\item The point $\vtheta(o)$ lies in $\Omega(T, P)^{\an}$ and satisfies for each $f = \sum_{\nu \in \NN^{\mbf \Phi}} f_{\nu}\xi^{\nu} \in k[\Omega(T,P)] = k[Z][(\xi_{\alpha, i})_{(\alpha, i) \in {\mbf \Phi}}]$, $$|f|(\vtheta(o)) = \max_{\nu \in \NN^{\mbf \Phi}} |f_{\nu}|(\vtheta_Z(\ast)).$$
		
		\item Given the identification between $A(G, T)$ and $V(G, T)$ given by $o$, we have for each $x \in A_{\QQ}(G, T)$  and $f \in k[\Omega(T,P)]$ : 
		\begin{equation}\label{eq:Formule_theta}
		|f|(\vtheta(x)) = \max_{\nu \in \NN^{\mbf \Phi}}|f_{\nu}|(\vtheta_Z(\ast)) \prod_{\substack{\alpha \in \Phi\\1 \le i \le N_{\alpha}}} e^{\nu(\alpha, i)\langle \alpha, x\rangle}. 
		\end{equation}
	\end{enumerate}
\end{prop}

\begin{proof}
	\begin{enumerate}
		\item The proof is analogous to that of \cite[Proposition 2.6]{RTW1}.
		
		Recall that $\vtheta(o)$ is the only point of the Shilov boundary of the affinoid subgroup $\mc G^{\an}$ of $G^{\an}$, that is, by \cite[1.2.4]{RTW1}, the only point in $\mc G^{\an}$ that is mapped to the generic point of $\mc G_{\tilde k}$ by the reduction map $$\rho: \mc G^{\an} \to \mc G_{\tilde k}.$$
		Because $\Omega(\mc T, \mc P)$ is an open subset of $\mc G$ that meets the special fibre, its special fibre $\Omega(\mc T, \mc P)_{\tilde k}$ contains the generic point of $\mc G$. As a consequence of the commutativity of the following diagram
		\begin{equation*}
		\begin{tikzcd}
		\Omega(\mc T, \mc P)^{\an} \arrow[d, "\subset"] \arrow[r, "\rho"] & \Omega(\mc T, \mc P)_{\tilde k} \arrow[d, "\subset"] \\ \mc G^{\an} \arrow[r, "\rho"] & \mc G_{\tilde k}
		\end{tikzcd},
		\end{equation*}
		the open subset $\Omega(\mc T, \mc P)^{\an}$ of $\mc G^{\an}$ must contain $\vtheta(o)$.
		Moreover, the condition that $\vtheta(o)$ reduces to the generic point of $\Omega(\mc T, \mc P)_{\tilde k}$ may be rewritten as the claim that, for any function $f \in k[\Omega(T, P)] = k[Z][(\xi_{a, i})_{a\in\, _k\Phi, 1 \le i \le N_a}]$, we have the equivalences:
		$$|f|(\vtheta(o)) \le 1 \Longleftrightarrow f \in k^{\oo}[\mc Z][(\xi_{a,i})_{(\alpha, i) \in {\mbf \Phi}}] $$ and $$|f|(\vtheta(o)) < 1 \Longleftrightarrow f \text{ maps to zero in } (\tilde k\otimes_{k^{\oo}} k^{\oo}[\mc Z])[(\xi_{\alpha, i})_{(\alpha, i)\in {\mbf \Phi}}].$$
		It follows that the restriction of $\vtheta(o)$ to $k[Z]$ is the Shilov boundary of the affinoid subgroup $\mc Z^{\an} \subset Z^{\an}$.
		
		\item By Proposition \ref{Transitivity} there exist a finite Galois extension $K$ and $t \in T(K)$ such that $x_K = t\cdot o_K$. By Proposition \ref{Equiv_theta}, we have $\vtheta(x_K) = t\vtheta(o_K)t^{-1}$. In particular, because $\inta(t)$ preserves $\Omega(T, P)$, we must have $\vtheta(x_K) \in \Omega(T, P)^{\an}$. Moreover, we have $$|f|(\vtheta(x)) = |f|(\vtheta(x_K)) = |\tau^*(f)|(\vtheta(o_K)),$$  where $\tau$ is the only map that makes the following diagram commute 
		\begin{center}
			\begin{tikzcd}
				\prod_{{\alpha} \in\ \Phi^+}\mbb A^{N_{-\alpha}}_K  \times \mc Z \times \prod_{{\alpha} \in\ \Phi^{+}}\mbb A^{N_{\alpha}}_K \arrow[r, "\tau"] \arrow[d, "\sim"] & \prod_{{\alpha} \in\ \Phi^+}\mbb A^{N_{-\alpha}}_K  \times \mc Z \times \prod_{{\alpha} \in\ \Phi^{+}}\mbb A^{N_{\alpha}}_K \arrow[d, "\sim"]\\ \Omega(T, P)_K \arrow[r, "\mathrm{int}(t_K)_{|\Omega(T, P)_K}"] & \Omega(T, P)_K
			\end{tikzcd}.
		\end{center}
		Now, because $T_K$ normalises $U_{\alpha, K}$ for each $\alpha \in \Phi_{\nd}$ and, further, preserves the weight space decomposition $U_{\alpha, K} \simeq U_{\alpha, K}/U_{2\alpha, K} \times U_{2\alpha, K}$, it follows that $\tau$ acts on each $\AA_K^{N_{\alpha}}$ by multiplication by the character $\alpha$. 
		In other words, $$\tau^*: K[Z][(\xi_{\alpha, i})_{(\alpha,i)\in {\mbf \Phi}}] \to K[Z][(\xi_{\alpha, i})_{(\alpha,i)\in {\mbf \Phi}}] $$ is the only morphism of $K[Z]$-algebras that maps $\xi_{\alpha, i}$ to $\alpha(t) \xi_{\alpha, i}$ for each $\alpha \in \Phi$ and $i \in [\![1, N_{\alpha}]\!]$. Consequently, if $f = \sum_{\nu \in \NN^{\mbf \Phi}} f_{\nu}\xi^{\nu} $, then $$\tau^*(f) = \sum_{\nu \in \NN^{\mbf \Phi}} f_{\nu}\prod_{(\alpha, i) \in {\mbf \Phi}} \alpha(t)^{\nu(a,i)} \xi^{\nu}$$ and therefore $$|f|(\vtheta(x)) = \max_{\nu \in \NN^{\mbf \Phi}} |f_{\nu}|(\vtheta_Z(\ast)) \prod_{(\alpha, i) \in {\mbf \Phi}} |\alpha(t)|^{\nu(a,i)}.$$
		Now, given that, $o_K$ is taken as origin of $A(G_K, T_K)$ and that, by definition of the apartment \cite[Définition 3.3]{Lou}, $t$ acts on $A(G_K, T_K)$ by translation by the vector $\nu(t) \in V(G_K, T_K)$ such that $\langle \alpha, \nu(t) \rangle = -\omega(\alpha(t))$, we have $$|f|(\vtheta(x)) = \max_{\nu \in \NN^{\mbf \Phi}} |f_{\nu}|(\vtheta_Z(\ast))\prod_{(\alpha,i)\in {\mbf \Phi}}e^{\nu(\alpha,i)\langle \alpha, x \rangle}.$$
	\end{enumerate}
\end{proof}

\begin{thm}\label{thm:construction_theta}
	The map $\vtheta$ extends uniquely to a continuous $G(k)$-equivariant map $$\vtheta: \mc B(G, k) \to G^{\an}.$$
\end{thm}

\begin{proof}
	Note that, because $\vtheta$ is already $G(k)$-equivariant, the continuity of the $G(k)$-action on $\mc B(G,k)$ and $G^{\an}$ will immediately imply that the extension is equivariant.
	Consequently, all we need to do is prove that, for any point $x \in \mc B(G, k)$ and any sequence $(x_n) \in \mc B_{\QQ}(G, k)^{\NN}$ converging to $x$, the sequence $\vtheta(x_n)$ converges in $G^{\an}$ and that the limit depends only on $x$.
	
	The proof comprises two steps, one mainly using Berkovich, and one using Bruhat-Tits theory. First of all, we treat the case of a sequence of virtually special points in an apartment converging to a point in that apartment, which follows easily from Proposition \ref{Formule_theta}. Next, we reduce the general case to the first step using the geometry of the building. The main observation is the fact that, from any sequence of points in the building converging to a point $x$, we may extract a subsequence lying in a single apartment containing $x$, a consequence of the local finiteness of the building.
	
	\textit{Step 1: }Assume that $x \in A(G, T)$. Then, it follows from Proposition \ref{Formule_theta} that, for any sequence $(x_n) \in A_{\QQ}(G, T)^{\NN}$ converging to $x$, the sequence $\vtheta(x_n)$ converges in $\Omega(T, P)^{\an} \subset G^{\an}$ and that the limit depends only on $x$. Namely, for any $f \in k[\Omega(T, P)]$, we have $$|f|(\vtheta(x_n)) = \max_{\nu \in \NN^{\mbf \Phi}} |f_{\nu}|(\vtheta_Z(\ast))\prod_{(\alpha,i)\in {\mbf \Phi}}e^{\nu(\alpha,i)\langle \alpha, x_n \rangle} \underset{n \to \infty}{\longrightarrow} \max_{\nu \in \NN^{\mbf \Phi}} |f_{\nu}|(\vtheta_Z(\ast))\prod_{(\alpha,i)\in {\mbf \Phi}}e^{\nu(\alpha,i)\langle \alpha, x \rangle}.$$
	
	\textit{Step 2: }Now, consider an arbitrary sequence $(y_n) \in \mc B_{\QQ}(G, k)^{\NN}$ converging to $x$.  Consider $\mathbf{f} \subset \mc B(G, k)$ the facet of $\mc B(G, k)$ containing $x$, that is the only relatively open polysimplex of $\mc B(G, k)$ containing $x$. We may associate to it the parahoric $k^{\oo}$-group scheme $\mc G^{0}_{\mbf f}$. By \cite[Théorème 4.7]{Lou}, the \textit{star} of $\mbf f$, that is the set of facets of $\mc B(G, k)$ whose closure contains $\mbf f$, is in one-to-one correspondence with the set of parabolic subgroups of $\mc G^0_{\mbf f, \tilde k}$. In particular, because $k$ is local, the residue field $\tilde k$ is finite and the set of parabolic $\tilde k$-subgroups of $\mc G^0_{\mbf f, \tilde k}$ is finite and the star of $\mbf f$ is therefore finite. 
	
	Because the union of the facets in the star of $\mbf f$ is a neighbourhood of $x$ \cite[Lemma VI.3A]{Brown}, it follows that it contains the $y_n$ for all sufficiently large $n$. Because the star of $\mbf f$ is a finite collection, there must exist a facet $\mbf f'$ such that $\mbf f \subset \ol{\mbf f'}$ and a subsequence $(y_{n_j})$ such that $y_{n_j} \in \mbf f'$ for all $j \in \NN$. Moreover, by \cite[7.4.9]{BT1}, the group $\mc G^0_{\mbf f}(k^{\circ})$ acts transitively on the apartments containing $\mbf f$. Therefore, there exists $g \in \mc G^0_{\mbf f}(k^{\circ})$ such that $ \mbf f' \subset g\cdot A(G, T)$. Consequently, we may write $y_{n_j} = gx_{n_j},$ with $x_{n_j} \in A_{\QQ}(G, T)$ for all $j \in \NN$. Note that the sequence $(x_{n_j})$ also converges to $x$. It will suffice to prove that $\lim_{j \to \infty} \vtheta(y_{n_j})  = \lim_{j \to \infty} \vtheta(x_{n_j})$.
	
	Because $\vtheta$ is $G(k)$-equivariant, we have $\vtheta(y_{n_j}) = g\vtheta(x_{n_j})g^{-1}$ for all $j \in \NN$. By continuity of the action, we have $$\lim_{j \to \infty} \vtheta(y_{n_j}) = g \left( \lim_{j \to \infty} \vtheta(x_{n_j})\right) g^{-1}.$$
	
	By the first step, we know that the limit $\lim_{j \to \infty} \vtheta(x_{n_j})$ is independent of the choice of the sequence of points in $A_{\QQ}(G, T)$ converging to $x$. Moreover, by Proposition \ref{Density_rational}, we know that there exists a sequence $(u_n)$ of virtually special points in $\mbf f \cap \mc B_{\QQ}(G, k)$ that converges to $x$. Therefore, we have $$\lim_{j \to \infty} \vtheta(y_{n_j}) = g\left(\lim_{j \to \infty} \vtheta(u_j) \right) g^{-1} =  \lim_{j \to \infty} \vtheta(gu_j) = \lim_{j \to \infty} \vtheta(u_j).$$
	
	The exact same argument applies to any subsequence of $(y_n)$, which establishes the fact that from any subsequence of $(\vtheta(y_n))$, we may extract a subsequence that converges to the common limit of the sequences $(\vtheta(x_n))$, for $(x_n)$ an arbitrary sequence of points in $A_{\QQ}(G, T)$ that converges to $x$. Hence, $(\vtheta(y_n))$ converges to that same limit.
	
	This proves that, for any sequence of points $(x_n)$ in $\mc B_{\QQ}(G, k)$ converging to a point $x$ in $A(G, T)$, the limit $\lim_{n \to \infty} \vtheta(x_n)$ exists and depends only on $x$. The general case easily follows. Indeed, if $y \in \mc B(G, k)$, then there exists $g \in G(k)$ such that $gy \in A(G, T)$ and,  for any sequence of virtually special points $(y_n)$ converging to $y$, the sequence $(\vtheta(y_n))$ converges to a limit depending only on $y$ if and only if the same holds for the sequence $(\vtheta(gy_n)) = (g\vtheta(y_n)g^{-1})$.
\end{proof}

Note that the density of virtually special points and the continuity of $\vtheta$ yield the following as a direct corollary of Lemma \ref{lem:Compatibility_extensions}.

\begin{prop}\label{Compatibility_extensions}
	Let $K/k$ be a finite Galois extension. We have the commutative diagram 
	\begin{center}
		\begin{tikzcd}
			\mc B(G, K) \arrow[r, "\vtheta"] & G_K^{\an}\arrow[d, "\pr_{K/k}^{\an}"] \\ \mc B(G, k)\arrow[u, "p_{K/k}"] \arrow[r, "\vtheta"] & G^{\an}
		\end{tikzcd}.
	\end{center}
\end{prop}

\section{Embedding in pseudo-flag varieties} \label{sect:Compactification}

In this section, our goal is to prove that we can recover the family of polyhedral compactifications of $\mc B(G, k)$ constructed in \cite[Chap. 1]{Char} -- the abstract building-theoretic analog of Satake's compactifications of symmetric spaces -- by mapping $\mc B(G, k)$ in the $(G/P_G(\mu))^{\an}$ for various $\mu \in X_*(G)$. 

\subsection{Dependence on the pseudo-parabolic subgroup}

Let $P$ be a pseudo-parabolic subgroup. We check that the composition $$\vtheta_P: \mc B(G, k) \overset{\vtheta}{\longrightarrow} G^{\an} \overset{\lambda_P}{\longrightarrow} (G/P)^{\an}$$ depends only on the conjugacy class of $P$, in the following sense.

\begin{prop}\label{IndependanceP}
	Let $Q$ be a pseudo-parabolic subgroup of $G$ conjugate to $P$ and $h \in G(k)$ such that $Q = hPh^{-1}$.  Then, setting the isomorphism $$\iota_{Q,P}: \begin{array}{ccc} G/P & \to & G/Q \\ gP & \mapsto & gh^{-1}Q\end{array},$$ induced by the orbit map $g \mapsto gh^{-1}Q$, we have $\vtheta_Q = \iota_{Q,P}^{\an}\circ \vtheta_P$.
	
\end{prop}

\begin{proof}
	Because both sides of the equation are continuous maps on $\mc B(G, k)$, it is sufficient to prove that they agree on the dense subset of virtually special points. Moreover, Proposition \ref{Compatibility_extensions} allows us to reduce to proving the equality for special vertices in the case when $G$ is pseudo-split over $k$.
	
	Now, assume that $G$ is pseudo-split and consider a special vertex $x \in \mc B(G, k)$. Denote by $\mc G$ the parahoric $k^{\circ}$-group scheme $\mc G_x^{\circ}$ associated to $x$. Let $\mu$ be a cocharacter of $G$ such that $P = P_G(\mu)$. Then, there exists a maximal split torus $T \subset G$ such that $\mu$ factors through $T$. Now, by construction, $\mc G$ contains the split $k^{\circ}$-torus $\mc T = \Spec(k^{\circ}[X^*(T)])$ and $\mu$ descends to a cocharacter $\mu^{\circ}$ of $\mc T$.
	Let $\mc P$ be the subgroup of $\mc G$ defined by $\mc P = P_{\mc G}(\mu^{\circ})$. Then, the image $\vtheta_P(x) \in (G/P)^{\an}$ is the unique preimage of the generic point of $(\mc G/\mc P)_{\tilde k}$ under the reduction map $$\rho_{\mc P}: (\mc G/\mc P)^{\an} \to (\mc G/\mc P)_{\tilde k}.$$
	Indeed, the commutative diagram 
	\begin{center}
		\begin{tikzcd}
			G^{\an} \arrow[r, "\lambda_P^{\an}"] & (G/P)^{\an}\\
			\mc G^{\an} \arrow[u, "\subset"] \arrow[d, "\rho"] \arrow[r, "\lambda_{\mc P}^{\an}"] & (\mc G/\mc P)^{\an} \arrow[d, "\rho_{\mc P}"] \arrow[u, "\subset"]\\
			\mc G_{\tilde k} \arrow[r, "(\lambda_{\tilde P})_{\tilde k}"] & (\mc G/\mc P)_{\tilde k}
		\end{tikzcd}
	\end{center}
	and the surjectivity of the map $(\lambda_{\mc P})_{\tilde k}: \mc G_{\tilde k} \to (\mc G/\mc P)_{\tilde k}$ imply that $\vtheta(x)$, which by construction is the unique preimage of the generic point of $\mc G_{\tilde k}$ under the reduction map $\rho$, maps to a preimage of the generic point of $(\mc G/\mc P)_{\tilde k}$ in $(\mc G/\mc P)^{\an}$. To establish the uniqueness of the preimage, recall that, for any non-empty affine open subset $\mc U$ of $\mc G/\mc P$, the space $\mc U^{\an}$ contains a unique preimage of the generic point of $(\mc G/\mc P)_{\tilde k}$ by \cite[Proposition 2.4.4]{Ber}. The fact that, by construction, $(\mc G/\mc P)^{\an}$ is covered by subsets of the form $\mc U^{\an}$ and that any two of these meet along a subset of the same type allows one to conclude.
	
	Next, observe that due to the Iwasawa decomposition \cite[Theorem 3.12]{Sol}, we may write $h = kp$, with $k \in \mc G(k^{\circ})$ and $p \in P(k)$. Hence, we have $Q = kPk^{-1} = P_G(k \cdot \mu)$. Set $\mc Q = P_{\mc G}(k \cdot \mu^{\circ})$. Then, we have $\mc Q = k\mc P k^{-1}$ and the map $\iota_{Q, P}$ descends to a map $$\iota_{\mc Q,\mc P}: \begin{array}{ccc} \mc G/\mc P & \to & \mc G/\mc Q \\ g\mc P & \mapsto & gk^{-1}\mc  Q\end{array}.$$ Therefore, the diagram \begin{center}
		\begin{tikzcd}
			(G/P)^{\an} \arrow[r, "\iota_{Q,P}^{\an}"]&  (G/Q)^{\an}\\ (\mc G/\mc P)^{\an} \arrow[u, "\subset"] \arrow[r, "\iota_{\mc Q, \mc P}^{\an}"] \arrow[d,"\rho_{\mc P}"]& (\mc G/\mc Q)^{\an} \arrow[d, "\rho_{\mc Q}"] \arrow[u, "\subset"]\\
			(\mc G/\mc P)_{\tilde k} \arrow[r, "(\iota_{\mc Q, \mc P})_{\tilde k}"] & (\mc G/\mc Q)_{\tilde k}
			
		\end{tikzcd}
	\end{center}
	commutes. Because $(\iota_{\mc Q, \mc P})_{\tilde k}$ maps generic points to generic points, the map $\iota_{Q, P}^{\an}$ maps the unique preimage in $(\mc G/\mc P)^{\an}$ of the generic point of $(\mc G/\mc P)_{\tilde k}$ to the unique preimage in $(\mc G/\mc Q)^{\an}$ of the generic point of $(\mc G/\mc Q)_{\tilde k}$. In other words, $$\vtheta_Q(x) = \iota_{Q, P}^{\an}(\vtheta_P(x)).$$
	
\end{proof}

\begin{rmk}
    The map $\iota_{Q, P}$ does not depend on the choice of $h$. Indeed, if $h' \in G(k)$ also satisfies $Q = h'Ph'^{-1}$, then, by \cite[Proposition 3.5.7]{CGP}, we have $h' h^{-1} \in N_G(Q)(k) = Q(k)$, and thus $h^{-1}Q = h'^{-1}Q$.
\end{rmk}

\begin{cor}\label{cor:equiv_thetaP}
	The map $\vtheta_P$ is $G(k)$-equivariant.
\end{cor}

\begin{proof}
	Let $g \in G(k)$. We have the following commutative diagram  
	\begin{center}
		\begin{tikzcd}
			G \arrow[equals]{d} \arrow[r, "\cdot g^{-1}"] & G \arrow[r, "\lambda_P"] & G/P \\
			G \arrow[rr, "\lambda_{g^{-1}Pg}"] & & G/g^{-1}Pg \arrow[u, "\iota_{P, g^{-1}Pg}"]
		\end{tikzcd},
	\end{center}
	in which all maps are $G(k)$-equivariant for the action by left-translation.
	Now, recall from Proposition \ref{Equiv_theta} that, for each $x \in \mc B(G, k)$, we have $$\vtheta(gx) = g\vtheta(x)g^{-1}.$$
	Consequently, we have 
	$$\vtheta_P(gx) = \lambda^{\an}_P(g\vtheta(x)g^{-1}) = g\lambda^{\an}_P(\vtheta(x)g^{-1}).$$
	By the above diagram, we have $$\lambda_P^{\an}(\vtheta(x)g^{-1}) = \iota_{P, gPg^{-1}}^{\an} (\vtheta_{g^{-1}Pg}(x)) = \vtheta_P(x).$$
	Finally, we find $$\vtheta_P(gx) = g\vtheta_P(x),$$ which completes the proof.
\end{proof}

\subsection{Restricting to an apartment: An explicit formula}\label{sect:Formula_ThetaP}

The following sections will establish that, if $P$ is a suitable pseudo-parabolic subgroup of $G$, then $\vtheta_P$ is an embedding with relatively compact image. The argument hinges on an explicit formula for the restriction of $\vtheta_P$ to an apartment $A(G, T)$ of the building such that $T \subset P$ in the case where $G$ is pseudo-split. The argument is essentially identical -- yet not reducible -- to that of Proposition \ref{Formule_theta}.

Assume that $G$ is pseudo-split over $k$ and let $A$ be an apartment in $\mc B(G, k)$ associated to a maximal split torus $T$. We denote by $\Phi = \Phi(G, T)$ the root system of $G$ corresponding to $T$. Let $o$ be a special vertex in $A$ and denote by $\mc G$ the parahoric $k^{\circ}$-model of $G$ associated to $o$. For each $\alpha \in \Phi_{\nd}$, we denote by $\mc U_{\alpha} \subset \mc G$ the $k^{\circ}$-model of the root group $U_{\alpha}$ associated to $o$. 

Let $P = P_G(\mu)$ be a pseudo-parabolic subgroup containing $T$ and $U^- = U_G(-\mu)$. Arguing as in the proof of Proposition \ref{IndependanceP}, we may find a $k^{\circ}$-cocharacter $\mu^{\circ}: \mbb G_{m, k^{\circ}} \to \mc G$ such that $\mu^{\circ} \otimes_{k^{\oo}} k = \mu$. Then, let $\mc P = P_{\mc G}(\mu^{\oo})$ and $\mc U^- = U_{\mc G}(-\mu^{\oo})$. Let $$\Psi = \Phi(U^-, T) = \{\alpha \in \Phi, \langle \alpha, \mu \rangle < 0\}. $$ Then, we claim that the multiplication map $$m: \prod_{\alpha \in \Psi_{\nd}} \mc U_{\alpha} \to \mc U^-$$ is a $\mc T$-equivariant isomorphism of $k^{\oo}$-schemes.

For each non-divisible root $\alpha \in\, \Psi_{\nd}$, we have a $\mc T$-equivariant isomorphism of $k^{\circ}$-schemes: $$\mc U_{\alpha} \simeq \mc U_{2{\alpha}} \times \mc U_{\alpha}/\mc U_{2{\alpha}}.$$ Moreover, each $\mc U_{2{\alpha}}$ (resp. $\mc U_{\alpha}/\mc U_{2{\alpha}}$) is an affine space over $k^{\circ}$ on which $\mc T$ acts linearly via the character $2{\alpha}$ (resp. ${\alpha}$). 
For each $\alpha \in \Phi$, set $N_{\alpha} = \dim_k(U_{\alpha}/U_{2\alpha})$ and fix $\mc T$-equivariant isomorphisms $$u_{\alpha}: \mbb A^{N_{\alpha}}_{k^{\oo}} \overset{\sim}{\to} \mc U_{\alpha}/\mc U_{2{\alpha}} \text{ and } u_{2{\alpha}}: \mbb A^{N_{2{\alpha}}}_{k^{\oo}} \overset{\sim}{\to} \mc U_{2{\alpha}}$$ for each non-divisible root ${\alpha}$. These yield a $\mc T$-equivariant isomorphism of $k^{\circ}$-schemes $$\prod_{\alpha \in \Psi} \AA^{N_{\alpha}}_{k^{\circ}} \overset{\sim}{\to} \mc U^-.$$ 
In the following, we identify $k[U^-]$ to the polynomial algebra $k[(\xi_{\alpha, i})_{(\alpha, i) \in\, \Psi \times \NN, 1 \le i \le N_{\alpha}}]$ using this isomorphism.

\begin{rmk}\label{term:adapted_coordinates}
    We call a system of coordinates $(\xi_{\alpha, i})_{(\alpha, i) \in \mbf \Psi}$ on $U^-$ so constructed \textit{adapted to} the point $o$.
\end{rmk}

Denote by $\Omega(\mc T, \mc P)$ the image of the open immersion defined by the multiplication map \cite[Proposition 2.1.8 (3)]{CGP} $$m: \mc U^- \times \mc P \to \mc G.$$ and let $\Omega(T, P) = \Omega(\mc T, \mc P) \otimes_{k^{\oo}} k$.

\begin{prop} \label{Formula_ThetaP}
	We keep all the above notation. Let $x$ be a point in $A$. We let $\Psi = \Phi(U_G(-\mu), T)$ and $$\mbf \Psi = \{(\alpha, i) \in\, \Psi \times \NN, 1 \le i \le N_{\alpha}\}.$$
	
	\begin{enumerate}
		\item The image $\vtheta_P(x)$ lies in  $(U^-)^{\an} \simeq (\Omega(T, P)/P)^{\an} \subset (G/P)^{\an}$.
		\item For each $f =\sum_{\nu \in \NN^{\mbf \Psi}} f_{\nu}\xi^{\nu} \in k[U_G(-\mu)] = k\left[(\xi_{\alpha, i})_{(\alpha, i)\in {\mbf \Psi}}\right]$, we have $$|f(\vtheta_P(x))| = \max_{\nu \in \NN^{\mbf \Psi}} |f_{\nu}| \prod_{(\alpha, i)\in {\mbf \Psi}}e^{\nu(\alpha, i) \langle \alpha, x\rangle}.$$
	\end{enumerate}
\end{prop}

\begin{proof}
	The arguments are essentially the same as those of Proposition \ref{Formule_theta}. 
	
	The statement will follow from the fact that $\vtheta(x)$ lies in $\Omega(T, P)^{\an}$ for all $x \in A$. To establish this, we proceed in three steps, first proving it for the origin $o$, then extending to any virtually special point, and finally concluding by density. 
	\begin{enumerate}
		\item First of all, because $\Omega(\mc T, \mc P)$ is an open subset of $\mc G$ that meets the special fibre of $\mc G$, the special fibre $\Omega(\mc T, \mc P)_{\tilde k}$ contains the generic point of $\mc G_{\tilde k}$. Therefore, the analytification $\Omega(\mc T, \mc P)^{\an}$ contains $\vtheta(o)$. Now, because the diagram
		\begin{center}
			\begin{tikzcd}
				\Omega(\mc T, \mc P)^{\an} \arrow[d, "\rho"] \arrow[r, "\lambda_{\mc P}^{\an}"] & (\Omega(\mc T, \mc P)/\mc P)^{\an} \arrow[d, "\rho"] \\
				\Omega(\mc T, \mc P)_{\tilde k}  \arrow[r, "\lambda_{\mc P_{\tilde k}}"] & (\Omega(\mc T, \mc P)/\mc P)_{\tilde k} 
			\end{tikzcd}
		\end{center}
		commutes and the bottom arrow is surjective (and therefore dominant), the image $\vtheta_P(o)$ is the unique Shilov boundary of $(\mc G/\mc P)^{\an}$.
		Identifying $(\Omega(\mc T, \mc P)/\mc P)^{\an}$ to $(\mc U^-)^{\an}$ via $\lambda_{\mc P}^{\an}$, the point $\vtheta_P(o)$ may be understood as the only seminorm such that, for each $f \in k[U^-] = k[(\xi_{\alpha, i})_{(\alpha, i) \in {\mbf \Psi}}]$, we have $$|f(\vtheta_P(o))| \le 1 \iff f \in k^{\circ}[(\xi_{\alpha, i})_{(\alpha, i) \in {\mbf \Psi}}]$$ and $$|f(\vtheta_P(o))| < 1 \iff f \text{ maps to zero in } \tilde k[(\xi_{\alpha, i})_{(\alpha, i)\in \mbf \Psi}].$$ In other words, if $f = \sum_{\nu \in \NN^{\mbf \Psi}} f_{\nu} \xi^{\nu} \in k[(\xi_{\alpha, i})]$, we have $$|f(\vtheta_P(o))| = \max_{\nu \in \NN^{\mbf \Psi}} |f_{\nu}|.	$$
		\item Next, consider a virtually special point $x \in A_{\QQ}(T, k)$. By Corollary \ref{Transitivity_split}, we have at our disposal a finite Galois extension $K/k$ and a $t \in T(K)$ such that $x_K = t o_K$. Then, we have $\vtheta(x_K) = t\vtheta(o_K)t^{-1}$ and, because $\Omega(T, P)$ is normalised by $T$, we have $\vtheta(x_K) \in \Omega(T, P)_K^{\an}$ and, consequently, $\vtheta(x) = \pr_{K/k}^{\an}(\vtheta(x_K)) \in \Omega(T, P)^{\an}$, and finally $\vtheta_P(x) \in (\Omega(T, P)/P)^{\an}$. Moreover, for $f \in k[U^-]$, we have $$|f|(\vtheta_P(x)) = |f|(\vtheta_P(x_K)) = |\tau^*(f)|(\vtheta_P(o_K)),$$ where $\tau$ is the only map that makes the following diagram commute
		\begin{center}
			\begin{tikzcd}
				\prod_{{\alpha} \in \Psi}\mbb A^{N_{\alpha}}_K \arrow[r, "\tau"] \arrow[d, "\sim"] & \prod_{{\alpha} \in \Psi}\mbb A^{N_{\alpha}}_K  \arrow[d, "\sim"]\\ (U^-)_K \arrow[r, "\mathrm{int}(t_K)"] & (U^-)_K
			\end{tikzcd}.
		\end{center}
		In other words, $\tau^*$ is the only $k$-automorphism of $k[(\xi_{\alpha,i})_{(\alpha, i)\in {\mbf \Psi}}]$ that maps $\xi_{\alpha, i}$ to $\alpha(t) \xi_{\alpha, i}$ for each $\alpha \in \Psi$ and $i \in [\![1, N_{\alpha}]\!]$. We deduce that if $f = \sum_{\nu \in \NN^{\mbf \Psi}} f_{\nu} \xi^{\nu} \in k[(\xi_{\alpha,i})_{(\alpha, i)\in {\mbf \Psi}}]$, we have $$\tau^*(f) = \sum_{\nu \in \NN^{\mbf \Psi}} \left(f_{\nu} \prod_{(\alpha, i)\in {\mbf \Psi}} \alpha(t)^{\nu(\alpha, i)}\right) \xi^{\nu}$$ and thus $$|f|(\vtheta_P(x_K)) = \max_{\nu \in \NN^{\mbf \Psi}} |f_{\nu}| \prod_{(\alpha, i)\in {\mbf \Psi}} |\alpha(t)|^{\nu(\alpha, i)} = \max_{\nu \in \NN^{\mbf \Psi}} |f_{\nu}| \prod_{(\alpha, i)\in {\mbf \Psi}} e^{\nu(\alpha, i)\langle \alpha, x\rangle } $$ using point 1.
		
		\item If $x$ is an arbitrary point in $A$, then there exists a sequence $(x_n)$ of points in $A_{\QQ}(T, k)$ that converges to $x$. Then, it is clear from the continuity in $x$ of the formula in point 2. that the sequence $(\vtheta_P(x_n))$ converges to a seminorm on $k[U^-]$, that is that $\vtheta_P(x) \in \Omega(T, P)^{\an}$ and that the formula for $\vtheta_P(x)$ is as announced.
	\end{enumerate}
	
\end{proof}

\subsection{Injectivity of the map}

This section establishes the injectivity of the map $\vtheta_P$ under the condition that $P$ is \textit{of non-degenerate type}, a property we define below.

\begin{defn}\label{def:non-degenerate_type}
	Let $G$ be a quasi-reductive group over a field $k$ and $P$ be a pseudo-parabolic subgroup of $G$. We say that $P$ is \textit{of non-degenerate type} if it does not contain any minimal nontrivial $k$-isotropic perfect smooth connected normal subgroup of $G$.
\end{defn}

The property of being of non-degenerate type admits the following Lie-theoretic characterisations.

\begin{prop}\label{Non-degenerate_type}
	Let $G$ be a quasi-reductive group over a field $k$ and $P$ be a pseudo-parabolic subgroup of $G$. Fix a $k$-split maximal torus $S \subset P$ and $P_0$ a minimal pseudo-parabolic subgroup such that $S \subset P_0 \subset P$. Denote by $P = L_P \ltimes U$ the Levi decomposition of $P$ corresponding to $S$. Let $\Phi = \Phi(G, S)$ be the root system associated to $S$ and $\Delta \subset \Phi$ be the system of simple roots corresponding to $P_0$.
	The following are equivalent.
	\begin{enumerate}
		\item $P$ is of non-degenerate type.
		\item The type $I = \Phi(L_P, S) \cap \Delta$ of $P$ contains no connected components of the Dynkin diagram $\Delta$.
		\item The subset $-\Psi = \Phi(U, S)$ spans $V^*(G, S) = \Span_{\RR}(\Phi)$ as a vector space.
	\end{enumerate}
\end{prop}

\begin{proof}
	It follows from \cite[Proposition C.2.32]{CGP} that the minimal nontrivial $k$-isotropic perfect smooth connected normal subgroup of $G$ are the $N_{\Phi_i} = \langle U_{\alpha}, \alpha \in \Phi_i \rangle$ for $\Phi_i$ ranging in the set of irreducible components of the root system $\Phi$. Because $P$ contains the root group $U_{\alpha}$ if and only if $\alpha \in \Phi(P, S)$, we have the equivalences $$N_{\Phi_i} \subset P \iff \Phi_i \subset \Phi(P, S) \iff \Phi_i \subset \Phi(L_P, S) = \Phi(P, S) \cap -\Phi(P,S).$$
	The equivalence $1 \iff 2$ immediately follows as the connected components of the Dynkin diagram $\Delta$ are precisely the $\Delta \cap \Phi_i$ for $\Phi_i$ ranging among the irreducible components of $\Phi$.
	
	We now prove the equivalence $2 \iff 3$. Assume that $I$ contains a connected component $\Delta_i = \Delta \cap \Phi_i$ of the Dynkin diagram $\Delta$. Then, we have $\Phi_i \subset \Phi(L_P, S)$ and therefore $\Phi(U, S) = \Phi(P, S)\setminus \Phi(L_P, S) \subset \bigsqcup_{j \ne i} \Phi_j$ and therefore $\Span_{\RR}(\Phi(U, S))$ is contained in $\Span_{\RR}\left(\bigsqcup_{j \ne i} \Phi_j\right)$, which is a proper subspace of $\Span_{\RR}(\Phi)$ because the $\Phi_i$ are pairwise orthogonal. Conversely, if $I$ contains no connected component of $\Delta$, let $\Phi_i$ be an irreducible component of $\Phi$. By assumption, there exists an $\alpha \in \Phi_i$ such that $\alpha \in \Delta \setminus I$. Because $P_0 \subset P$, that is $\Delta \subset \Phi(P,S)$, we must have $\alpha \in \Phi(U, S)$. Then, for each $\beta \in I \cap \Phi_i$, the root $s_{\alpha}(\beta) = \beta - \underbrace{\langle \alpha, \beta \rangle}_{< 0} \alpha$ is a sum of elements of $\Delta$ and is not contained in the span of $I$, hence it lies in $\Phi(U, S)$. Therefore, the subspace $\Span_{\RR}(\Phi(U, S))$ contains $I \cap \Phi_i$. Because it contains $(\Delta \setminus I) \cap \Phi_i$, it actually contains $\Delta \cap \Phi_i$ and therefore contains $\Phi_i$. 
	
\end{proof}

\begin{prop}\label{Injectivity_ThetaP_apartment}
	The map $\vtheta_P$ is injective if and only if $P$ is of non-degenerate type
\end{prop}

\begin{proof}
	Observe that, because $\vtheta_P$ is $G(k)$-equivariant and because any two points lie in an apartment, the map $\vtheta_P$ is injective if and only if there exists an apartment $A$ such that $\vtheta_{P|A}$ is injective. 
	
Let $S$ be a maximal $k$-split torus contained in $P$. Let $o$ be a special point of $A(G, k)$ and $K$ be a finite Galois extension such that $G_K$ is pseudo-split and $o_K$ is special. Let $T$ be a maximal $K$-split torus that contains $S_K$ such that $p_{K/k}(A(G, S)) \subset A(G_K, T)$. By Propositions \ref{Compatibility_extensions} and \ref{Formula_ThetaP}, we have the commutative diagram 
	\begin{center}
		\begin{tikzcd}
			A(G_K, T) \arrow[r, "\vtheta_{P_K}"] & (U^-)_K^{\an}\arrow[d, "\pr_{K/k}^{\an}"] \\ A(G, S)\arrow[u, "p_{K/k}"] \arrow[r, "\vtheta_P"] & (U^-)^{\an}
		\end{tikzcd},
	\end{center}
	where $U^-$ is the $k$-split unipotent radical of the opposite pseudo-parabolic subgroup to $P$ with respect to $S$. With the notations of Section \ref{sect:Formula_ThetaP} associated to the point $o_K$, we have, for each $f = \sum_{\nu \in \NN^{\mbf \Psi}} f_{\nu}\xi^{\nu} \in K[U^-] = K[(\xi_{\alpha, i})_{(\alpha, i)\in {\mbf \Psi}}]$ and $x \in A(G_K, T)$: \begin{equation} \label{eq: RelFormulathetaP}
	    |f(\vtheta_{P_K}(x))| = \max_{\nu \in \NN^{\mbf \Psi}} |f_{\nu}|\prod_{(\alpha, i) \in {\mbf \Psi}}e^{\nu(\alpha, i)\langle \alpha, x - o_K \rangle}. 
	\end{equation}
	Note also that, for each character $\chi \in X^*(T)$ and each cocharacter $\lambda \in X_*(S)$, we have the following compatibility property between duality pairings $$\langle \chi, \lambda_K \rangle = \langle \chi_{|S}, \lambda \rangle .$$
	Now, let $a \in \Phi(U^-, S)$ be a relative weight of the Lie algebra of $U^-$, and $\phi \in \Lie(U^-)^* \subset k[U^-]$ be a nonzero weight vector of weight $-a$. Then, we may write $$\phi = \sum_{\substack{(\alpha, i) \in {\mbf \Psi}\\ \alpha_{|S} = a}}\phi_{\alpha, i} \xi_{\alpha, i},$$ with $\phi_{\alpha, i}\in K$ and not all $\phi_{\alpha, i}$ equal to zero. Hence, we have, for all $x \in A(G, S)$: $$|\phi(\vtheta_P(x))| = |\phi_K(\vtheta_{P_K}(x_K))| = \left(\max_{\substack{(\alpha, i) \in {\mbf \Psi}\\ \alpha_{|S} = a}} |\phi_{\alpha, i}|\right) e^{\langle a, x-o\rangle}.$$
	Therefore, if $x, y \in A(G, S)$ satisfy $\vtheta_P(x) = \vtheta_P(y)$, then $\langle a, x-o \rangle = \langle a, y-o \rangle $ for each $a \in \Phi(U^-, S)$. Conversely, if $x$ and $y$ in $A(G,S)$ satisfy $\langle a, x-y \rangle = 0$ for each $a \in \Phi(U^-, S)$, then the observation following Equation (\ref{eq: RelFormulathetaP}) that $\vtheta_{P_K}(x) = \vtheta_{P_K}(y)$, and therefore $\vtheta_P(x) = \vtheta_P(y)$. 
	In other words, for $x$ and $y$ in $A(G,S)$, we have $\vtheta_P(x) = \vtheta_P(y)$ if and only if $x-y$ lies in the annihilator $$\Phi(U^-, S)^0 = \{v \in V(G, S), \forall a \in \Phi(U^-, S), \langle a, v \rangle = 0\}.$$ Hence, the map $\vtheta_P$ is injective if and only if the subset $\Phi(U^-, S)$ spans the dual space $V^*(G, S)$, that is if and only if $P$ is of non-degenerate type by Proposition \ref{Non-degenerate_type}.
\end{proof}

\subsection{Relative compactness}

In this section, we establish that, if $P$ is a pseudo-parabolic subgroup of $G$, then the image of the map $\vtheta_P$ is relatively compact. As usual, it will suffice to consider the case when $G$ is pseudo-split. The main insight is that the Cartan decomposition \cite[Theorem 3.13 (d)]{Sol} allows one to restrict to proving that sequences taking values in a suitably chosen cone admit a convergent subsequence, which will follow from the previous section.

\begin{prop}\label{prop:relative_compactness}
	The map $\vtheta_P$ has relatively compact image. In particular, if $P$ is of non-degenerate type, then $\vtheta_P$ factors through a $G(k)$-equivariant compactification of $\mc B(G, k)$.
\end{prop}

\begin{proof}
	First of all, thanks to Proposition \ref{Compatibility_extensions}, we have the following commutative diagram of continuous maps
	\begin{center}
		\begin{tikzcd}
			\mc B(G, K) \arrow[r, "\vtheta_{P_K}"] & (G/P)_K^{\an} \arrow[d, "\pr_{K/k}^{\an}"] \\ 
			\mc B(G, k) \arrow[u, "p_{K/k}"] \arrow[r, "\vtheta_P"]  & (G/P)^{\an}
		\end{tikzcd}
	\end{center}	
	for any finite Galois extension $K/k$. It follows that, if $\vtheta_{P_K}$ has relatively compact image, then so does $\vtheta_P$. 
	
	It is therefore sufficient to prove our claim in the case where $G$ is pseudo-split over $k$, which we will assume in what follows. Let $T$ be a maximal split torus contained in $P$ and $\mu \in X_*(T)$ be a cocharacter such that $P = P_G(\mu)$. Set $U^- = U_G(-\mu)$. Fix a special vertex $o$ in $A(G, T)$, set $\Psi = \Phi(U^-, T)$, and consider the cone $$C(P) = \{x \in A(G, T)\ |\ \forall \alpha \in \Psi, \langle \alpha, x - o \rangle \le 0\}.$$
	Observe that we have $$A(G, T) = W(G, T)(k) \cdot C(P),$$ where the action  of $W(G, T)(k) = \dfrac{N_G(T)(k)}{Z_G(T)(k)}$ on $A(G, T)$ is given by the identification of $A(G, T)$ with $V(G, T)$ setting $o$ as origin.
	Moreover, because $o$ is special, the projection $$N_G(T)(k)_o \to W(G, T)(k)$$ is onto, hence we have $$A(G, T) = N_G(T)(k)_o \cdot C(P).$$ Moreover, by \cite[Theorem 3.11]{Sol}, we have $$\mc B(G, k)  = G(k)_o \cdot A(G, T),$$ and therefore $$\mc B(G, k) = G(k)_o \cdot C(P).$$
	Because $\vtheta_P$ is $G(k)$-equivariant and continuous and because $G(k)_o$ acts through its compact quotient $G(k)_o/Z(G)(k)$, proving that $\vtheta_P$ has relatively compact image reduces to proving that its restriction to $C(P)$ has relatively compact image. Now, let $(x_n)$ be a sequence in $C(P)$. Identify $A(G, T)$ to $V(G, T)$ using $o$ as origin and assume the notations of Section \ref{sect:Formula_ThetaP}. Then, by Proposition \ref{Formula_ThetaP}, we have $\vtheta_P(x_n) \in (U^-)^{\an}$ for each $n \in \NN$ and, for each $f = \sum_{\nu \in \NN^{\mbf \Psi}} f_{\nu}\xi^{\nu}\in k[U^-] = k[(\xi_{\alpha, i})_{(\alpha, i) \in {\mbf \Psi}}]$, we have $$|f|(\vtheta_P(x_n)) = \max_{\nu \in \NN^{\mbf \Psi}} |f_{\nu}| \prod_{(\alpha, i)\in {\mbf \Psi}}e^{\nu(\alpha, i) \langle \alpha, x_n\rangle}.$$
	For each $\alpha \in \Psi$, the sequence $(e^{\langle \alpha, x_n\rangle})$ is bounded. Consequently, there exists a subsequence $(x_{n_j})$ such that $(e^{\langle \alpha, x_{n_j}\rangle })$ converges for each $\alpha \in \Psi$. Then, for each $f \in k[U^-]$, the sequence $(|f|(\vtheta_P(x_{n_j})))$ converges and thus the subsequence $(\vtheta_P(x_{n_j}))$ converges in $(U^-)^{\an}$.
\end{proof}

The proof actually yields the following more precise statement.

\begin{prop}\label{prop:relative_compactness_conewise}
    Let $S$ be a maximal split torus in $P$ and $o$ a special vertex in $A(G, S)$. Consider the cone $$C(P) = \{x \in A(G, S), \forall \alpha \in \Phi(R_{u,k}(P), S), \langle \alpha, x-o \rangle \ge 0\}.$$
    Then, the image $\vtheta_P(C(P))$ is relatively compact in the open subspace $(\Omega(S, P)/P)^{\an} \subset (G/P)^{\an}$.
\end{prop}

\subsection{Openness of the corestriction}

In this section, we check that the corestriction 
$$\vtheta_P: \mc B(G, k) \to \ol{\vtheta_P(\mc B(G, k))}$$ is an open map. In particular, if $P$ is of non-degenerate type, we prove that  $\vtheta_P$ is an embedding.

\begin{prop}\label{prop:openness}
    The corestricted map $\vtheta_P: \mc B(G, k) \to \ol{\vtheta_P(\mc B(G, k))}$ is open. In particular, whenever $P$ is of non-degenerate type, the map $\vtheta_P$ is an embedding.
\end{prop}

The proof goes in three steps. First of all, we prove that the restriction of $\vtheta_P$ to an apartment $A(G, S)$ (and thus any apartment) induces an open map to the closure of its image.
Secondly, we prove that the action map $$G(k) \times \ol{\vtheta_P(A(G, S))} \longto \ol{\vtheta_P(\mc B(G, k))}$$ is a topological quotient map. Lastly, we deduce that the map $\vtheta_P: \mc B(G, k) \to \ol{\vtheta_P(\mc B(G, k))}$  is open.

\begin{lem}\label{lem:rel_openness_img_apt}
	Let $S$ be a maximal split torus of $G$. The image $\vtheta_P(A(G, S))$ is open in its closure $\ol{\vtheta_P(A(G, S))}$.
\end{lem}

\begin{proof}
	Without loss of generality, we may assume $S$ to be contained in $P$.
	
	Let $o$ be a special vertex in $A(G, S)$. Let $T$ be a maximal torus of $G$ and $K/k$ a finite Galois extension such that 
	\begin{enumerate}
		\item The torus $T_K$ is split.
		\item The apartment $A(G_K, T_K)$ contains $p_{K/k}(A(G, S))$.
		\item The point $o_K \in A(G_K, T_K)$ is a special vertex.
	\end{enumerate}
	Then, as in the proof of Proposition \ref{Non-degenerate_type}, we have the diagram 
	\begin{center}
		\begin{tikzcd}
			A(G_K, T) \arrow[r, "\vtheta_{P_K}"] & (U^-)_K^{\an}\arrow[d, "\pr_{K/k}^{\an}"] \\ A(G, S)\arrow[u, "p_{K/k}"] \arrow[r, "\vtheta_P"] & (U^-)^{\an}
		\end{tikzcd},
	\end{center}
	where $U^-$ is the $k$-split unipotent radical of the pseudo-parabolic subgroup of $G$ opposite to $P$ with respect to $S$.
	We may rewrite the explicit formula in Proposition \ref{Formula_ThetaP} as the statement that, for each $x \in \vtheta_P(A(G, S))$ and $f \in K[U^-] = K[(\xi_{\alpha, i})_{(\alpha, i) \in {\mbf \Psi}}]$, we have $$|f(x)| = \max_{\nu \in \NN^{\mbf \Psi}} |f_{\nu}| \prod_{(\alpha, i) \in {\mbf \Psi}} |\xi_{\alpha, i}(x)|^{\nu(\alpha, i)},$$
	where, for each $\alpha \in \Phi(U^-, S)$, the family $(\xi_{\alpha, i})$ is a choice of coordinates on $U_{\alpha}$ given by $o_K$ through the procedure described in the introduction to Section \ref{sect:Formula_ThetaP}.
	It follows from the proof of Proposition \ref{prop:relative_compactness} that this formula extends to all points $x \in \ol{\vtheta_P(A(G, S))} \cap (U^-)^{\an}$ and that a point $x \in \ol{\vtheta_P(A(G, S))} \cap (U^-)^{\an}$ lies in $\vtheta_P(A(G, S))$ if and only if $|\xi_{\alpha, i}(x)| > 0$ for all $(\alpha, i) \in {\mbf \Psi}$.
	
	It follows that the image $\vtheta_P(A(G, S))$ is open in $\ol{\vtheta_P(A(G, S))} \cap (\Omega(S, P)/P)^{\an}$.
	Note that the Bruhat decomposition \cite[Theorem C.2.8]{CGP} implies that $G/P = \bigcup_{w \in W} \dot w\Omega(S, P)/P$, where, for each $w \in W$, we choose an arbitrary lift $\dot w$ of $w$ in $N_G(S)(k)$, see \cite[II.1.9.7]{Jantzen}.
	By $N_G(S)(k)$-equivariance of $\vtheta_P$, we deduce that $\vtheta_P(A(G, S))$ is open in $\ol{\vtheta_P(A(G, S))}\cap (\dot w\Omega(S, P)/P)^{\an}$ for each $w \in W$, and therefore open in $\ol{\vtheta_P(A(G, S))}$.
	
\end{proof}
    
We now conclude our first step.

\begin{lem}\label{lem:rel_openness_rest_apt}

    Let $S$ be a maximal split torus of $G$. The map $$\vtheta_P: A(G, S) \to \ol{\vtheta_P(A(G, S))}$$ is open.
\end{lem}

\begin{proof}
    As in the previous lemma, we may assume that $S$ is contained in $P$. Let $U^-$ be the $k$-split unipotent radical of the pseudo-parabolic subgroup opposite to $P$ with respect to $S$. 
    Let $o$ be a special vertex in $A(G, S)$. Recall from Proposition \ref{Injectivity_ThetaP_apartment} that $\vtheta_P$ factors through the quotient map $$A(G, S) \to A(G, S)/\langle \Phi(U^-, S) \rangle ^{0},$$ a surjective affine map, hence open. We may therefore assume that $P$ is of non-degenerate type, ie. that we have the implication $$(\forall a \in \Phi(U^-, S), \langle a, x \rangle = 0) \implies x = 0$$ for each $x \in V(G, S)$.
    Under this assumption, the open half-spaces $$H_{a, r}^+ = \{x \in A(G, S), \langle a, x - o\rangle > r\} \text{ and } H_{a,r}^- = \{x \in A(G, S), \langle a, x - o\rangle < r\}$$ for $a \in \Phi(U^-, S) $ form a subbase of open sets for the topology of $A(G, S)$. Because $\vtheta_P$ is injective, we need only check that these half-spaces have an open image.
    Consider a system of coordinates $(\xi_{\alpha, i})_{(\alpha, i)\in \mbf \Psi}$ on $U^-_K$ as in Section \ref{sect:Formula_ThetaP}. Let $\phi = \sum_{\substack{(\alpha, i) \in {\mbf \Psi}\\ \alpha_{|S} = a}}\phi_{\alpha, i} \xi_{\alpha, i} \in k[U^-]$ be a nonzero weight vector of weight $-a$ in $\Lie(U^-)^*$. Then, it follows from Proposition \ref{Formula_ThetaP} that, for each $r \in \RR$, we have $$\vtheta_P(H_{a,r}^+) = \vtheta_P(A(G, S)) \cap \{x \in (U^-)^{\an}, |\phi(x)| > Ce^r\}$$ and $$ \vtheta_P(H_{a,r}^-) = \vtheta_P(A(G, S)) \cap \{x \in (U^-)^{\an}, |\phi(x)| < Ce^r\},$$ where $C = \max_{\substack{(\alpha, i) \in {\mbf \Psi}\\ \alpha_{|S} = a}} |\phi_{\alpha, i}|$.  Consequently, for each $r \in \RR$, the images $\vtheta_P(H_{a, r}^+)$ and $\vtheta_P(H_{a, r}^-)$ are open in $\vtheta_P(A(G, S))$, and thus open in $\ol{\vtheta_P(A(G, S))}$ by Lemma \ref{lem:rel_openness_img_apt}.
\end{proof}

Now, let $o$ be a special vertex in $A(G, S)$. By \cite[Theorem 3.12]{Sol}, we have a coarse Cartan decomposition $$\mc B(G, k) = G(k)_o \cdot A(G, S), $$ and therefore $$\ol{\vtheta_P(\mc B(G, k))} = G(k)_o \cdot \ol{\vtheta_P(A(G, S))}.$$
Because $G(k)_o \times \ol{\vtheta_P(A(G, S))}$ and $\ol{\vtheta_P(\mc B(G, k))}$ are compact Hausdorff by Proposition \ref{prop:relative_compactness}, the surjective action map $$\pi_P: G(k)_o \times \ol{\vtheta_P(A(G, S))} \to \ol{\vtheta_P(\mc B(G, k))}$$ is in fact a quotient map.

In order to conclude that $\vtheta_P$ is an open map, we need to make one last observation.

\begin{lem}\label{lem:saturation_piP}
    Let $S$ be a maximal split torus of $G$ and $x \in \ol{\vtheta_P(A(G, S))}$. If, for some $g \in G(k)$, we have $gx \in \vtheta_P(\mc B(G, k))$, then $x \in \vtheta_P(A(G, S))$.
\end{lem}

\begin{proof}
    The proof is completely analogous to that of \cite[Lemma 3.33]{RTW1}. The key point is the fact that, using Proposition \ref{Formula_ThetaP}, we may characterise the points of $\vtheta_P(A(G, S))$ as the points $x$ of $\ol{\vtheta_P(A(G, S))}$ that induce norms on their local ring $\mc O_{(G/P)^{\an}, x}$, a property that is clearly preserved by the action of $G(k)$.
    
    First of all, reasoning as in the proof of Lemma \ref{lem:rel_openness_img_apt}, we observe that there exists a $\dot w \in N_G(S)(k)$ such that $\dot wx \in (\Omega(S, P)/P)^{\an}$. Moreover, the lemma holds for $x$ if and only if it does for $\dot wx$. We may therefore assume that $x$ lies in $\ol{\vtheta_P(A(G, S))} \cap (\Omega(S, P)/P)^{\an}$. 	
	Let $o$ be a special vertex in $A(G, S)$. Let $T$ be a maximal torus of $G$ and $K/k$ a finite Galois extension such that 
	\begin{enumerate}
		\item The torus $T_K$ is split.
		\item The apartment $A(G_K, T_K)$ contains $p_{K/k}(A(G, S))$.
		\item The point $o_K \in A(G_K, T_K)$ is a special vertex.
	\end{enumerate}
    Denote by $U^-$ the $k$-split unipotent radical of the opposite pseudo-parabolic to $P$ with respect to $S$. 	Identify $A(G, S)$ with $V(G, S)$ (resp. $A(G_K, T_K)$ with $V(G_K, T_K)$) setting $o$ (resp. $o_K$) as origin. Under this identification, the map $p_{K/k}: A(G, S) \to A(G_K, T_K)$ is the map induced by base change on cocharacters.
	Consider the sets $\Psi = \Phi(U^-, S)$ and $\Psi_K = \Phi(U-, T_K)$. Let $(\xi_{\alpha, i})_{(\alpha, i) \in \mbf \Psi_K}$ be a system of coordinates on $U^-_K$ that is adapted to $o_K$ (Remark \ref{term:adapted_coordinates}). 
	Let $(x_n)$ be a sequence in $A(G, S)$ such that $x = \lim_{n \to \infty} \vtheta_P(x_n)$. The sequence $(\vtheta_{P_K}((x_n)_K))$ converges to a point $x_K$ in $\ol{\vtheta_{P_K}(A(G_K, T_K))} \cap (\Omega(T, P)/P)_K^{\an}$ by Proposition \ref{Formula_ThetaP}.
	Moreover, we have the equivalences
	$$\begin{array}{ccl} x_K \in \vtheta_{P_K}(A(G_K, T_K)) & \iff &  \forall (\alpha, i) \in \mbf \Psi_K, |\xi_{\alpha, i}(x_K)| > 0 \\
	& \iff & x_K \text{ defines a norm on } \mc O_{(G/P)_K^{\an}, x_K}\end{array}.$$
	To conclude, first observe that, by Proposition \ref{Formula_ThetaP}, for each weight $\alpha \in \Psi_K$ and $a = \alpha_{|S} \in \Psi$ its restriction to $S$, we have $$|\xi_{\alpha, i}(x_K)| = \lim_{n \to \infty} e^{\langle a, x_n \rangle}.$$
	Therefore, we have the equivalence $$x_K \in \vtheta_{P_K}(A(G_K, T_K)) \iff x \in \vtheta_P(A(G, S)),$$ and in particular the implication $$x \in \vtheta_P(A(G, S)) \implies x \text{ defines a norm on } \mc O_{(G/P)^{\an}, x},$$ as the seminorm defined by $x$ on $\mc O_{(G/P)^{\an}, x}$ is the restriction of the one defined by $x_K$ on $\mc O_{(G/P)_K^{\an}, x_K}$, hence a norm.
	On the other hand, if $a \in \Psi$ and $\phi \in \Lie(U^-)^*$ is a nonzero weight vector of weight $-a$, then as observed in the proof of Lemma \ref{lem:rel_openness_rest_apt}, there exists $C > 0$ such that $$|\phi(x)| = |\phi_K(x_K)| = \lim_{n \to \infty} C e^{\langle a, x_n \rangle}.$$
	It follows that, if $x$ defines a norm on $\mc O_{(G/P)^{\an}, x}$, then by the above computation, we have $$\forall (\alpha, i) \in \mbf \Psi_K, |\xi_{\alpha, i}(x_K)| > 0,$$ and therefore $x \in \vtheta_P(A(G, S))$, by the above observations.
	
\end{proof}

We now carry out the proof of Proposition \ref{prop:openness}.

\begin{proof}[Proof of Proposition \ref{prop:openness}]
    Let $U$ be an open subset of $\mc B(G, k)$. Consider the diagram 
    \begin{center}
        \begin{tikzcd}
            G(k)_o \times A(G, S) \arrow[r, "\id \times \vtheta_P"] \arrow[d, "\pi"] & G(k)_o \times \ol{\vtheta_P(A(G, S))} \arrow[d, "\pi_P"] \\
            \mc B(G, k) \arrow[r, "\vtheta_P"] & \ol{\vtheta_P(\mc B(G, k))}
        \end{tikzcd}.
    \end{center}
    Then, the subset $V = (\id\times \vtheta_P)(\pi^{-1}(U))$ is open in $G(k)_o \times \ol{\vtheta_P(A(G, S))}$ by Lemma \ref{lem:rel_openness_rest_apt}. As we established above, the rightmost map $\pi_P$ is a quotient map. Therefore, in order to prove that $\vtheta_P(U) = \pi_P(V)$ is open in $\ol{\vtheta_P(\mc B(G, k)}$, it is sufficient to check that $V$ is $\pi_P$-saturated, which is a direct consequence of Lemma \ref{lem:saturation_piP}.
\end{proof}
\section{Description of the boundary}\label{sect:Boundary}

In this section, we establish that the family of compactifications of $\mc B(G, k)$ constructed above corresponds to the one constructed in \cite{Char} and specialises to the one constructed in \cite{RTW1} in the reductive case. More precisely, we will prove that, if $P$ is a pseudo-parabolic subgroup of $G$ of non-degenerate type and $S$ a maximal split torus in $P$, then the compactification     
$\ol{\vtheta_P(\mc B(G, k))} \subset (G/P)^{\an}$ is $G(k)$-equivariantly homeomorphic to the compactification $\ol{\mc B(G, k)}$ defined in \cite[Section 5]{Char} for a suitable choice of fan.
We deduce that the compactifications are stratified and interpret the strata in terms of group-theoretic data.

\subsection{Polyhedral compactifications of affine buildings}

In this section, we introduce terminology and results from \cite{Char} that will be in constant use in what follows.

The article \cite{Char} associates to each affine building $\mc I$, modeled on an affine Coxeter complex $\Sigma$, and each fan $\mc F$ in the direction $\vec \Sigma$ of $\Sigma$ subject to some compatibility relations \cite[2.3]{Char} a canonical compactification $\ol I^{\mc F}$ of $\mc I$.
We give an overview of the construction, first compactifying each apartment and then glueing the compactifications of the various apartments together to get a compactification of the building.

\subsubsection*{Fans and compactifications of affine spaces}

We first describe the compactifications of apartments in \cite[Section 3]{Char}.

The setting is that of a Euclidean space $V$ whose dual is endowed with a root system $\Phi \subset V^*$. Denote by $W$ the Weyl group of $\Phi$. 
Recall that a cone in $V$ is said to be \textit{polyhedral} if it can be defined by a finite number of linear inequalities. We call a cone \textit{relatively open} if it is open in its linear span. A polyhedral cone is relatively open if and only if it can be defined by a finite number of linear equations and strict linear inequalities.
Then, a fan in the sense of \cite{Char} is a family $\mc F$ of relatively open polyhedral cones of $V$ satisfying the following conditions:
\begin{enumerate}
	\item The family $\mc F$ is a finite partition of $V$.
	\item $\{0\} \in \mc F$.
	\item For each $f \in \mc F$, the boundary $\partial f = \ol{f} \setminus f$ is a union of cones of $\mc F$, called the \textit{proper faces} of $f$ (we also consider $g$ to be a face of itself).
	\item For all cones $ f$ and $g$ in $\mc F$, if $ f$ is a face of $g$, then $\ol{f} = \Span_{\RR}(f) \cap \ol{g}$.
\end{enumerate}

\begin{ex}
	The root system $\Phi \subset V^*$ determines a hyperplane arrangement in $V$. The facets, in the sense of \cite[Ch. V, §1.3]{Bou}, then form a fan $\mc F$ in $V$ if the root system $\Phi$ is essential. In the following, we call this fan the Weyl fan. Its facets will be called Weyl facets, and its open facets called Weyl chambers.
\end{ex}

\begin{rmk}\label{def:closed_fan}
The above definition differs slightly from the one coming from toric geometry adopted in e.g. \cite[1.1]{Fulton} or \cite[Appendix B]{RTW1}, which we shall call a \textit{closed fan}, namely a finite set $\mc F$ of closed strictly convex polyhedral cones (ie. containing no lines) such that any two cones in $\mc F$ meet along a face, and such that any face of a cone in $\mc F$ belongs to $\mc F$. In that context, a face of a closed polyhedral  convex cone $f$ is defined as the intersection of $f$ with a supporting hyperplane.
The operations of taking closures and relative interiors allow one to pass from fans on the one hand to closed fans covering $V$ and vice versa. These operations also establish a correspondence between both notions of faces, as we check below.
\end{rmk}

\begin{lem}\label{lem:strict_convexity_fan}
    Let $\mc F$ be a fan in $V$. For each cone $g$ in $\mc F$, the closure $\ol g$ is strictly convex.
\end{lem}

\begin{proof}
    Let $g \in \mc F$ be a cone of minimal dimension such that there exists $x\in V\setminus\{ 0\}$ such that $x \in \ol g$ and $-x \in \ol g$.
    Note that $x$ is in $g$. Indeed, otherwise, there exists a proper face $f$ of $g$ such that $x \in f$. Then, by Condition 4, we have $-x  \in \ol f$, which contradicts the minimality of $ g$.
    Now, since $g$ is relatively open and $x \in g$, it follows from \cite[Theorem 6.1]{Rockafellar} that $0 = \frac{1}{2}(x+ (-x)) \in g$, which contradicts Condition 2.
    
\end{proof}

\begin{lem}\label{lem:faces_in_fans}
    Let $\mc F$ be a fan in $V$ and $g$ a cone in $\mc F$. Let $f$ be a relatively open polyhedral cone in $V$. The following are equivalent:
    \begin{enumerate}
        \item The cone $f$ is a face of $g$.
        \item There exists a linear form $\alpha \in V^*$ such that $\alpha(x)  \ge 0$ for each $x \in g$ and such that $\ol f  = (\ker \alpha) \cap \ol g$.
    \end{enumerate} 
\end{lem}

\begin{proof}
    \textit{(Step 1: $(i) \implies (ii)$)} Assume that $f$ is a face of $g$. Let us first assume that $g$ is open. Then because $\ol g$ is strictly convex, the dual cone $$g^{\vee} = \left\{\alpha \in V^*,\forall x \in g, \alpha(x) \ge 0\right\}$$ has non-empty interior by \cite[Corollary 14.6.1]{Rockafellar}. Because the annihilator $$f^0 = \{\alpha \in V^*, f \subset \ker \alpha\}$$ meets the interior of $g^{\vee}$ (that is, because $f$ lies in a supporting hyperplane of $g$, see \cite[Theorem 11.6]{Rockafellar}), the intersection $f^{0} \cap g^{\vee}$ has non-empty relative interior. Let $\alpha \in f^0 \cap g^{\vee}$ be a relative interior point. Let $x \in \ol g \setminus \Span_{\RR}(f)$. Because $\alpha$ is in the relative interior of $f^{0}\cap g^{\vee}$, it lies in the relative interior of the half-space $$f^{0} \cap \{x\}^{\vee} = \{\alpha \in f^0, \alpha(x) \ge 0\}$$ by \cite[Corollary 6.5.2]{Rockafellar}, that is the (relatively) open half-space $\{\alpha \in f^{0}, \alpha(x) > 0\},$ hence $\alpha(x) > 0$. Finally, we have $$(\ker \alpha) \cap \ol g = \Span_{\RR}(f) \cap \ol g = \ol f.$$
    In the general case, let $H_0$ be a supporting hyperplane of $g$ in $\Span_{\RR}(g)$ such that $H_0 \cap \ol g = \ol f$. Then, the subspace $H = H_0 \oplus \Span_{\RR}(g)^{\perp}$ is a supporting hyperplane of $g$ in $V$ such that $H \cap \ol g = \ol f$.
    
    \textit{(Step 2: $(ii) \implies (i)$)} Let $\alpha \in V^*$ be a linear form such that $\alpha(x) \ge 0$ for each $x \in g$ and such that $\ol f = (\ker \alpha)\cap \ol g$. Because each face of $g$ is relatively open, the form $\alpha$ is either identically zero or only assumes positive values on it. Consequently, the intersection $(\ker \alpha)\cap \ol g$ is a union of faces of $g$, say $$(\ker \alpha )\cap \ol g = f_1 \cup \dots \cup f_r.$$
    Because the intersection is closed, we in fact have $$(\ker \alpha )\cap \ol g = \ol f_1 \cup \dots \cup \ol f_r.$$
    One of the $\ol f_i$ must then have non-empty interior in $\Span_{\RR}((\ker \alpha) \cap \ol g)$, from which we deduce that $$\ol f_i = \Span_{\RR}(f_i) \cap \ol g = \Span_{\RR}((\ker \alpha) \cap \ol g) \cap \ol g  = (\ker \alpha) \cap \ol g = \ol f.$$
    Because both $f_i$ and $f$ are convex and relatively open, we conclude that $f = f_i$, hence $f$ is indeed a face of $g$.
\end{proof}

We now assume that we are given an affine space $A$ directed by $V$. Two affine convex cones $f = x+\vec f$ and $g = y + \vec g$ in $A$, directed by cones $\vec f$ and $\vec g$ in $V$, are said to be \textit{parallel} ($f\parallel g$) if they have the same direction, that is if $\vec f = \vec g$. The cones $f$ and $g$ are said to be \textit{equivalent} ($f \sim g$) if they are parallel and have non-empty intersection or, equivalently, if they share a parallel subcone, ie. if there exists a cone $h \subset f \cap g$ that is parallel to $f$ and $g$. Call $f \subset A$ an \textit{affine $\mc F$-cone} if its direction $\vec f$ belongs to $\mc F$. Denote by $\mc F_A$ the set of affine $\mc F$-cones in $A$.

Then, the compactification of $A$ associated to a fan $\mc F$ in $V$ is the quotient of the set of affine $\mc F$-cones in $A$ under the equivalence relation defined above: $$\ol A^{\mc F} = \underbrace{\{x + \vec f, x \in A, \vec f \in \mc F\}}_{\mc F_A}/\sim.$$ The equivalence classes can be grouped by direction, thereby partitioning the compactification: $$\ol A^{\mc F} = \bigsqcup_{\vec f \in \mc F} A_{\vec f}.$$ For each $\vec f \in \mc F$, the subspace $A_{\vec f} = \{[x + \vec f], x \in A\}$ is called the \textit{façade} of $\ol A^{\mc F}$ of type $\vec f$. This façade is naturally an affine space under $V/\Span_{\RR}(\vec f)$.
Geometrically, setting an origin at some point $o \in A$, we compactify each open cone $o + \vec f = \{x \in A, \forall i \in I, \alpha_i(x-o) > 0\}$ by allowing the $\alpha_i(x-o)$ for $i\in I$, which we think of as the distance from $x$ to the walls of the cone $o+\vec f$, to assume infinite values.

Precisely, we have the following statement
\begin{prop}[{Corollary of \cite[Proposition 3.2.4]{Char}}]\label{prop:convergence_in_Charignon}
	Let $(x_n)$ be a sequence of points in $A$, and $l = [a + \vec f] \in \ol A^{\mc F}$. Then, the sequence $(x_n)$ converges to $l$ if and only if, for each linear form $\alpha \in V^*$ such that $\vec f \subset \ker \alpha$, we have $$\alpha(x_n - a) \underset{n \to +\infty}{\longto} 0$$ and, for each $\alpha \in V^*$ with $\alpha(x) > 0$ for each $x$ in $\vec f$, we have $$\alpha(x_n - a ) \underset{n \to +\infty}{\longto} +\infty.$$
\end{prop}

These façades form a stratification of $\ol A^{\mc F}$, specifically we have $$A_{\vec g} \subset \ol{A_{\vec f}} \iff \vec f \text{ is a face of } \vec g.$$

\subsubsection*{Fans associated to types of pseudo-parabolic subgroups}

We now define the family of fans that will concern us in the following. 

Weyl chambers are in one-to-one correspondence with bases of the root system: Given a basis $\Delta$ of the root system, the cone $$C(\Delta) = \{x \in V, \forall \alpha \in \Delta, \alpha(x) > 0\}$$ is a Weyl chamber. Conversely, given a Weyl chamber $C$, the intersection $C^{\vee}\cap \Phi$ of the dual cone of $C$ with $\Phi$ is a positive system of roots that gives rise to a unique basis. Moreover, given a Weyl chamber $C$ corresponding to a basis $\Delta$, the faces of $C$ are in one-to-one correspondence with the subsets of $\Delta$: If $J \subset \Delta$, then the set $$\{x \in V, \forall \alpha \in J, \alpha(x) = 0 \text{ and } \forall \alpha \in \Delta \setminus J, \alpha(x) > 0\}$$ is a face of $C$, said to be of type $J$. 

\begin{ex}
	Fix a Weyl chamber $C$ corresponding to a basis $\Delta$ of $\Phi$, and a subset $J \subset \Delta$ containing no connected component of $\Delta$, considered as a Dynkin diagram. Say a face of $C$ is \textit{admissible} if its type contains no connected component of $\Delta$. Then, for each admissible face $f$ of $C$ of type $I$, we let $\ol f^J$ denote the union of $f$ and its faces of type included in $I \cup (J \cap I^{\perp})$. 
	Then, the set $$J \cdot f = \bigcup_{w \in \langle s_{\alpha}, \alpha \in J \cap I^{\perp}\rangle } w\ol f^J$$ is a relatively open polyhedral cone, and the set $$\mc F^J = \bigcup_{w \in W} w \cdot \{J \cdot f, f \text{ admissible face of } C\}$$ is a fan in $V$ by \cite[Proposition 2.5.3]{Char}.
\end{ex}

When the root system is given by a maximal split torus $S$ in a quasi-reductive group $G$ over a field $k$, we have a natural group-theoretic interpretation for this fan.

Let $P$ is a pseudo-parabolic subgroup of non-degenerate type in a quasi-reductive group $G$ over a field $k$. Let $P_0$ be a minimal pseudo-parabolic subgroup of $G$ and $S$ a maximal split torus such that $S \subset P_0 \subset P$. Set $V = V(G, S)$, and $\Phi = \Phi(G, S)$. Let $W$ be the Weyl group of $\Phi$.
Let $\Delta$ be the basis of the root system $\Phi$ corresponding to $P_0$ and $J = \Delta \cap -\Phi(P, S)$ be the type of $P$.
Set $$C(P) = \{x \in V(G,S), \forall \alpha \in \Phi(\mc R_{u,k}(P), S), \langle \alpha, x \rangle \ge 0\}, $$ and let $$C = \{x \in V(G, S), \forall \alpha \in \Delta, \langle \alpha, x \rangle > 0\}$$ be the Weyl chamber associated to $P_0$.
These cones are related as follows.
\begin{prop}\label{prop:gp_interpretation_FJ}
    With the above notation, we have $$C(P) = \ol{J \cdot C}.$$
    Consequently, the open cones of $\mc F^J$ are the $\overset{\circ}{\arc{C(Q)}}$, where $Q$ ranges among the pseudo-parabolic subgroups containing $S$ conjugate to $P$. Moreover, the correspondence $P \mapsto C(P)$ is $W$-equivariant.
\end{prop}

\begin{proof}
    Observe that $$\ol{J \cdot C} = \bigcup_{w \in \langle s_{\alpha}, \alpha \in J \rangle } w(\ol C),$$ which is the union of the closed Weyl chambers associated to the bases $w(\Delta)$ for $w$ in $W_J = \langle s_{\alpha}, \alpha \in J \rangle$.
    Let $L_P$ be the Levi subgroup of $P$ corresponding to $S$. Then, the reflection group $W_J$ is the Weyl group $W(L_P, S)$. In particular, because the conjugacy action of $L_P$ on $P$ preserves $\mc R_{u,k}(P)$, the group $W_J$ stabilises the cone $C(P)$. Hence, because $\ol C \subset C(P)$, we must in turn have $\ol{J \cdot C} \subset C(P)$. 
    
    On the other hand, the cone $C(P)$ is a union of Weyl facets with non-empty interior and is therefore the closure of the union of the Weyl chambers that it contains. Further, a Weyl chamber is contained in $C(P)$ if and only if it corresponds to a minimal pseudo-parabolic subgroup of $G$ containing $S$ and contained in $P$ by \cite[Proposition 3.5.14]{CGP}. Now, by \cite[Proposition 2.2.10]{CGP}, the map $$Q \mapsto Q \cap L_P$$  establishes an $L_P(k)$-equivariant and order-preserving bijection between pseudo-parabolic subgroups of $P$ containing $S$ and pseudo-parabolic subgroups of $L_P$ containing $S$. By \cite[Theorem C.2.3]{CGP} and \cite[Theorem C.2.5]{CGP}, the normaliser $N_{L_P}(S)(k)$ acts transitively on the minimal pseudo-parabolic subgroups of $L_P$ containing $S$, hence it acts transitively on the minimal pseudo-parabolic subgroups of $P$ containing $S$. Hence, the Weyl chambers contained in $C(P)$ are precisely the $w(\Delta)$ for $w \in W(L_P, S)(k) = W_J$. We conclude that $C(P)$ is in fact equal to $\bigcup_{w \in W_J} w(\ol C) = \ol{J \cdot C}$, which completes the proof.
\end{proof}

\subsubsection*{Compactifications of affine buildings}

To conclude these preliminary considerations, we extend the above construction to yield compactifications of affine buildings. In \cite{Landvogt}, where the affine buildings are Bruhat-Tits buildings of reductive groups, this is obtained by way of a glueing procedure, much like the building itself: stabilisers are defined for points of the compactification of a given apartment with respect to the Weyl fan, then the compactification of $\mc B(G, k)$ is defined to be the quotient of $G(k) \times \ol A^{\mc F}$ by an equivalence relation similar to that defining the building $\mc B(G, k)$.

A salient feature of $\cite{Char}$ is that it performs the same construction while solely relying on the internal geometry of the building. For this purpose, the author introduces the notion of a cone in the building, which in some sense is a glueing of affine cones in various apartments containing a common subcone of a Weyl facet, called its \textit{core}.

Consider an apartment $A$ in an affine building $\mc B$. We denote by $V$ the direction of $A$ and by $\Phi \subset V^*$ its vectorial root system. We also assume that we are given a fan $\mc F$ in $V$ that is invariant under the action of the Weyl group $W = W(\Phi)$. Under this assumption, the notion of an affine $\mc F$-cone can be defined in the whole building: An affine $\mc F$-cone is a subset $f \subset \mc B$ such that there exist an apartment $B$ containing $f$ and an isomorphism of affine Coxeter complexes $\phi_B: A \to B$ such that $\phi_B^{-1}(f)$ is an affine $\mc F$-cone in $A$. For each apartment $B$, denote by $\mc F_B$ the set of affine $\mc F$-cones contained in $B$.

Then, for each cone $f \in \mc F$, the \textit{core} of $f$ is the convex cone $$\delta(f) = f \cap \Fix(\Stab_W(f)).$$

\begin{ex}\label{ex:cores_in_FJ}
	In the case when $\mc F = \mc F^J$ with the notations of the previous example, then, for each face $f$ of the Weyl cone $C$ of admissible type $I$, the core of $J \cdot f$ is the face of $C$ of type $I \cup (J \cap I^{\perp})$.
\end{ex}

Likewise, given any affine $\mc F$-cone $f = x + \vec f$ in $A$, the core of $f$ is defined as $\delta(f) = x + \delta(\vec f)$. 

We then define the cone of $\mc B$ associated to an affine cone $f$ in $A$ to be $$\tilde f = \bigcup_{\delta(f) \subset B} \phi_B(f),$$ where $B$ ranges among all apartments containing the core $\delta(f)$, and $\phi_B: A \to B$ is an isomorphism of Coxeter complexes fixing the intersection $A \cap B$. Denote by $\mc F_{\mc B}$ the set $$\mc F_{\mc B} = \{\tilde f, f \in \mc F_A, A \text{ apartment of } \mc B\}.$$ It can then be checked that the $\phi_B(f)$ all have the same core, which we also call the core of the cone $\tilde f$ and denote by $\delta(\tilde f)$.

Finally, the compactification of $\mc B$ is constructed in the same way as that of $A$, ie. as equivalence classes of cones. The only problem is to define the notion of parallelism for cones in $\mc B$, which is where cores play a crucial role.
Indeed, it follows from \cite[Corollaire 4.5.9]{Char} that, given any two cones $F$ and $G$ in $\mc B$, there exists an apartment $A$ that contains parallel subcones of $\delta(F)$ and $\delta(G)$. We then say that $F$ and $G$ are \textit{parallel} ($F \parallel G$) if there exist parallel subcones $f \subset \delta(F)$ and $g \subset \delta(G)$ and an apartment $A$ containing $f$ and $g$ such that $f \parallel g$ in $A$ (this parallelism does not depend on the specific $A$).
Likewise, we say that $F$ and $G$ are \textit{equivalent} ($F \sim G$) if $F$ and $G$ are parallel and have a non-empty intersection. Then, we set $$\ol{\mc B}^{\mc F} = \mc F_{\mc B}/\sim.$$
We call $\ol{\mc B}^{\mc F}$ the \textit{polyhedral compactification} of $\mc B$ with respect to the fan $\mc F$.

As in the previous case, the equivalence classes can be grouped in parallelism classes: $$\ol{\mc B}^{\mc F} = \bigcup_{F \in \mc F_{\mc B}} \mc B_F.$$ For each $F \in \mc F_{\mc B}$, the subspace $\mc B_F = \{G \in \mc F_{\mc B}, G \parallel F\}$ is called the (outer) \textit{façade} of $\mc B$ of type $F$. 
The façade $\mc B_F$ is naturally an affine building whose apartments are the façades $A_f$, for apartments $A$ containing affine $\mc F$-cones $f$ such that $\tilde f \parallel F$. These façades are either equal or disjoint, and form the strata of a topological stratification of $\ol{\mc B}^{\mc F}$ \cite[§8.2.2]{Char}.

\subsection{A comparison}\label{Comparison_compactifications}

Because the building $\mc B(G,k)$ arises from valuated root data in $G(k)$ \cite[Section 3.2]{Sol}, Proposition 7.4 of \cite{Char} applies. It is thus sufficient, in order to identify the aforementioned compactifications, to prove that the procedure of taking the closure of a given apartment in one space and the other yields homeomorphic spaces.

\begin{prop}\label{prop:comparison_compactifications}
	Let $P$ be a pseudo-parabolic subgroup of non-degenerate type, let $P_0 \subset P$ be a minimal pseudo-parabolic subgroup, and let $S \subset P_0$ a maximal split torus. Denote by $\Delta$ the basis of $\Phi(G, S)$ associated to $P_0$ and let $J \subset \Delta$ be the type of $P$.
	
	The map $\vtheta_P$ extends to a continuous map $$\ol{A(G, S)}^{\mc F^J} \to (G/P)^{\an}$$ that is a homeomorphism onto its image.
\end{prop}

\begin{proof}
	By \cite[Lemma 3.28]{GJT}, it suffices to prove that, for any two sequences $(x_n)$ and $(x'_n)$ of points in $A(G, S)$, the following are equivalent:
\begin{enumerate}[(i)]
	\item The sequences $(x_n)$ and $(x'_n)$ converge and have the same limit in $\ol{A(G, S)}^{\mc F^J}$.
	\item The sequences $(\vtheta_P(x_n))$ and $(\vtheta_P(x'_n))$ converge and have the same limit in $(G/P)^{\an}$.
\end{enumerate} 
We use Propositions \ref{Formula_ThetaP} and \ref{prop:convergence_in_Charignon} to characterise convergence in one space and the other. Both characterisations can be translated into convergence conditions on the sequences $(\langle a, x_n\rangle)$ and $(\langle a, x'_n\rangle)$ for certain roots $a$.

Fix a special vertex $o \in A(G, S)$. Identify $A(G, S)$ with its direction $V(G, S)$ setting $o$ as origin. 
Let 
$$\begin{array}{ccl} 
C(P) & = & \{x \in A(G, S), \forall a \in \Phi(\mc R_{u, k}(P), S),  \langle a, x\rangle \ge 0\}
\end{array}.$$ 

\textit{(Step 1: (i) $\implies$ (ii))}
Let $(x_n)$ and $(x'_n)$ be two sequences of points in $A(G, S)$. Assume that both $(x_n)$ and $(x'_n)$ converge in $\ol{A(G, S)}^{\mc F^J}$ and have the same limit $l$ lying in the façade $A(G, S)_{\vec f}$, where $\vec f \in \mc F^J$. 

Without loss of generality, we may assume that $\vec f$ is a face of $\interior{C(P)}$. Indeed, by construction, $\vec f$ is a face of an open cone of $\mc F^J$ and, by Proposition \ref{prop:gp_interpretation_FJ}, the Weyl group $W(G, S)(k)$ acts transitively on these open cones. If $w \in W(G, S)(k)$ maps $\vec f$ to a face of $\interior{C(P)}$, and if $\dot w \in N_G(S)(k)$ is a lift of $w$, then $(\dot wx_n)$ and $(\dot wx'_n)$ both converge and have the same limit $\dot w l$ lying in $A(G, S)_{w \vec f}$, and $w \vec f$ is a face of $\interior{C(P)}$. Furthermore, $(\vtheta_P(x_n))$ and $(\vtheta_P(x'_n))$ converge and have the same limit in $(G/P)^{\an}$ if and only if the same holds for $(\vtheta_P(\dot w x_n))$ and $(\vtheta_P(\dot w x'_n))$ by $G(k)$-equivariance of $\vtheta_P$ and continuity of the action of $G(k)$ on $(G/P)^{\an}$.

Let us therefore assume that $\vec f$ is a face of $\interior{C(P)}$. Each $a \in \Phi(\mc R_{u,k}(P), S)$, assumes only nonnegative values on $\vec f$. In fact, because $\vec f$ is relatively open, each $a$ is either identically zero or positive on $\vec f$. Therefore, by Proposition \ref{prop:convergence_in_Charignon}, the limits $$\lim_{n \to \infty} \langle a, x_n \rangle \text{ and } \lim_{n \to \infty} \langle a, x'_n \rangle$$ both exist in $\RR \cup \{+\infty\}$ and coincide for each $a \in \Phi(\mc R_{u,k}(P), S)$.

We now deduce that $(\vtheta_P(x_n))$ and $(\vtheta_P(x'_n))$ converge in $(G/P)^{\an}$ and have the same limit. We use Proposition \ref{Formula_ThetaP}.
Let $T$ be a maximal torus of $G$ containing $S$ and $K/k$ be a finite Galois extension such that:
\begin{enumerate}
	\item The torus $T_K$ is split.
	\item The apartment $A(G_K, T_K)$ contains $p_{K/k}(A(G, S))$.
	\item The point $o_K \in A(G_K, T_K)$ is a special vertex.
\end{enumerate}
Identify $A(G_K, T_K)$ with its direction $V(G_K, T_K)$ by setting the origin at $o_K$. Let $\mu: \mbb G_m \to S$ such that $P = P_G(\mu)$ and set $U^- = U_G(-\mu)$. By Proposition \ref{Formula_ThetaP}, both $(\vtheta_P(x_n))$ and $(\vtheta_P(x'_n))$ take their values in $(\Omega(T, P)/P)^{\an}$. Moreover, identifying $\Omega(T, P)/P$ with $U^-$ and choosing a system of coordinates $(\xi_{\alpha, i})_{(\alpha, i)\in \mbf \Psi_K}$ on $(U^-)_K$ adapted to $o_K$, we have, for each $f = \sum_{\nu \in \NN^J} f_{\nu}\xi^{\nu} \in K[U^-]$: $$|f(\vtheta_{P_K}((x_n)_K))| = \max_{\nu \in \NN^{\mbf \Psi_K}} |f_{\nu}|\prod_{(\alpha, i) \in \mbf \Psi_K }e^{\nu(\alpha, i)\langle \alpha, (x_n)_K \rangle},$$
which we may rewrite $$|f(\vtheta_{P_K}((x_n)_K))| = \max_{\nu \in \NN^{\mbf \Psi_K}} |f_{\nu}|\prod_{a \in \Psi} \prod_{\substack{(\alpha, i) \in \mbf \Psi_K\\ \alpha_{|S} = a}}e^{\nu(\alpha, i)\langle a, x_n \rangle}.$$

Evidently, for each $f \in K[U^-]$, the quantities $|f(\vtheta_{P_K}((x_n)_K))| $ and $|f(\vtheta_{P_K}((x'_n)_K))| $ converge in $\RR$ with the same limit, hence the sequences $(\vtheta_{P_K}((x_n)_K))$ and $(\vtheta_{P_K}((x'_n)_K))$ converge in $(U^-)^{\an}_K$ and have the same limit. By continuity of $\pr_{K/k}^{\an}$, so do the sequences $(\vtheta_P(x_n))$ and $(\vtheta_P(x'_n))$, by commutativity of the diagram 
\begin{center}
	\begin{tikzcd}
		A(G_K, T) \arrow[r, "\vtheta_{P_K}"] & (U^-)_K^{\an}\arrow[d, "\pr_{K/k}^{\an}"] \\ A(G, S)\arrow[u, "p_{K/k}"] \arrow[r, "\vtheta_P"] & (U^-)^{\an}
	\end{tikzcd}.
\end{center}

\textit{(Step 2: (ii) $\implies$ (i))} Let $(x_n)$ and $(x'_n)$ be two sequences of points of $A(G, S)$. Assume that $(\vtheta_P(x_n))$ and $(\vtheta_P(x'_n))$ both converge in $(G/P)^{\an}$. We first prove that the common limit of the two sequences may be assumed to lie in $(\Omega(S, P)/P)^{\an}$. 

Recall from Proposition \ref{prop:gp_interpretation_FJ} that $$A(G, S) = \bigcup_{\substack{Q \text{ conjugate to } P\\ S \subset Q}} C(Q).$$
Because only finitely many subgroups $Q$ conjugate to $P$ contain $S$, there exists such a $Q$ and a subsequence $(x_{n_j})_{j \in \NN}$ such that $\vtheta_P(x_{n_j}) \in C(Q)$ for each $j \in \NN$. In particular, we have $\lim_{n \to \infty} \vtheta_P(x_n) \in \ol{\vtheta_P(C(Q))}$.
By Proposition \ref{prop:relative_compactness_conewise}, for each pseudo-parabolic subgroup $Q$ containing $S$ conjugate to $P$, the image $\vtheta_Q(C(Q))$ is relatively compact in $(\Omega(S, Q)/Q)^{\an}$, ie. $\ol{\vtheta_Q(C(Q))} \subset (\Omega(S, Q)/Q)^{\an}$. 
Let $w \in W(G, S)(k)$ be such that $w C(Q) = C(P)$ and $\dot w$ be a lift of $w$ in $N_G(Q)(k)$.	Then, we have $P = \dot w Q \dot w^{-1}$, and therefore, by Proposition \ref{IndependanceP}: $$\ol{\vtheta_P(C(Q))} = \iota_{P, Q}^{\an}(\vtheta_Q(C(Q))) \subset \iota_{P, Q}^{\an}((\Omega(S, Q)/Q)^{\an}) = \dot w (\Omega(S, P)/P)^{\an}.$$
Consequently, the sequences $(\vtheta_P(\dot w^{-1} x_n))$ and $(\vtheta_P(\dot w^{-1} x'_n))$ converge to a point in $(\Omega(S, P)/P)^{\an}$. Moreover, the sequences $(x_n)$ and $(x'_n)$ converge to the same limit if and only if the same holds for $(\dot w^{-1} x_n)$ and $(\dot w^{-1}x'_n)$.

Let us therefore reduce to the case when $(\vtheta_P(x_n))$ and $(\vtheta_P(x'_n))$ both converge in $(\Omega(S, P)/P)^{\an} \simeq (U^-)^{\an}$. Choose a special vertex $o$ in $A(G, S)$, a maximal torus $T$ of $G$ containing $S$ and a finite Galois extension $K/k$ as in Step 1 and make the same choice of notations. Let $a \in -\Phi(\mc R_{u,k}(P), S)$ and let $\phi = \sum_{\substack{(\alpha, i) \in {\mbf \Psi_K}\\ \alpha_{|S} = a}} \phi_{\alpha, i} \xi_{\alpha, i} \in k[U^-] \subset K[U^-]$ be a $k$-linear form on the root group $U_a$. Then, we have $$|\phi(\vtheta_P(x_n))| = |\phi_K(\vtheta_{P_K}((x_n)_K))| = \left(\max_{\substack{(\alpha, i) \in {\mbf \Psi_K}\\ \alpha_{|S} = a}} |\phi_{\alpha, i}|e^{\langle \alpha, (x_n)_K \rangle}\right),$$ which rewrites as $$|\phi(\vtheta_P(x_n))| = \left(\max_{\substack{(\alpha, i) \in {\mbf \Psi_K}\\ \alpha_{|S} = a}} |\phi_{\alpha, i}| \right) e^{\langle a, x_n\rangle}.$$
Because the sequences $(|\phi(\vtheta_P(x_n))|)$ and $(|\phi(\vtheta_P(x'_n))|)$ converge with the same limit, the sequences $(\langle a, x_n\rangle)$ and $(\langle a, x'_n\rangle)$ must converge in $\RR \cup \{- \infty\}$ with the same limit.
We deduce that, for each $a \in \Phi(\mc R_{u,k}(P), S)$, we have $$\lim_{n \to \infty} \langle a, x_n\rangle = \lim_{n \to \infty} \langle a, x'_n \rangle \in \RR \cup \{+\infty\}.$$
Let $$I_c = \left\{a \in \Phi(\mc R_{u,k}(P), S), \lim_{n \to \infty} \langle a, x_n \rangle \in \RR\right\}$$ and $$I_d = \left\{a \in \Phi(\mc R_{u,k}(P), S), \lim_{n \to \infty} \langle a, x_n \rangle = +\infty \right\}.$$
Then, denoting by $\vec f$ the set $$ \vec f = \{x \in V(G, S), \forall a \in I_c, \langle a, x\rangle = 0 \text{ and } \forall a \in I_d, \langle a, x \rangle > 0\},$$ the cone $\vec f$ is a face of $C(P)$ by Lemma \ref{lem:faces_in_fans} and it follows from Proposition \ref{prop:convergence_in_Charignon} that $(x_n)$ and $(x'_n)$ converge to a common limit lying in the façade $A(G, S)_{\vec f}$, which completes the proof of the equivalence.
\end{proof}

By the criterion \cite[Proposition 7.4]{Char}, the compactification $\ol{\vtheta_P(\mc B(G, k))}$ is thus $G(k)$-equivariantly homeomorphic to the polyhedral compactification of $\mc B(G, k)$ associated to the fan $\mc F^J$.

\subsection{Identifying the strata}

It now follows from \cite[Corollaire 8.2.11]{Char} that the compactification $\ol{\vtheta_P(\mc B(G, k))}$ is stratified by affine buildings. Precisely, we have a $G(k)$-equivariant homeomorphism $$\ol{\vtheta_P(\mc B(G, k))} = \bigcup_{F \in (\mc F^J)_{\mc B(G, k)}} \mc B(G, k)_{F},$$ where the right-hand side is indexed by cones in $\mc B(G, k)$ in the sense of \cite[Sections 5.2]{Char}. The stratum $\mc B(G, k)_{F}$ corresponding to a cone $F$ is in turn defined as the set of equivalence classes of cones parallel to $F$ for the relation \cite[Définition 5.4.1]{Char}.

We begin by proving that these strata -- which are affine buildings -- can be interpreted as Bruhat-Tits buildings of certain Levi subgroups of pseudo-parabolic subgroups of $G$.

\begin{thm}\label{Stratification}
	Let $F$ be a cone in $(\mc F^J)_{\mc B(G, k)}$. There exists a unique pseudo-parabolic subgroup $Q$ of $G$ such that $\mc B(G, k)_F$ is stable under the action of $Q(k)$. Moreover, this stratum identifies $Q(k)$-equivariantly to the Bruhat-Tits building of the quasi-reductive group $Q/\mc R_{us,k}(Q)$.
\end{thm}

The proof of this theorem will take up most of this section and will therefore be broken down into several steps, namely Propositions \ref{Zariski_density_fixator}, \ref{Action_on_facade}, and \ref{Identifying_facade}. The outline goes as follows: We begin by associating to an apartment $A(G, S)$ containing the core of $F$ a pseudo-parabolic subgroup $Q$ containing $S$. We prove that this pseudo-parabolic subgroup $Q$ depends only on the parallelism class of $F$ (thus neither on the specific representative $F$ nor the specific apartment containing $\delta(F)$ that we choose). We deduce from these observations that $Q(k)$ stabilises the façade $\mc B(G, k)_F$ and that the $k$-split unipotent radical $\mc R_{us,k}(Q)(k)$ acts trivially on that same façade.
Lastly, we check that, given a Levi factor $L_Q \subset Q$, the projections $A(G, S) \to A(G, S)_F$ with $S \subset L_Q$ glue together to a $L_Q(k)$-equivariant map $L_Q(k) \cdot A(G, S) \epi \mc B(G, k)_{F}$ that factors through an equivariant, simplicial homeomorphism $\mc B(L_Q, k) \to \mc B(G, k)_{F}$, which will complete the proof.

\subsubsection*{Associating a pseudo-parabolic subgroup to a façade}

Let us now carry out the first step, namely identifying the pseudo-parabolic subgroup $Q$ and establishing its sole dependency on the parallelism class of $F$. The argument hinges on the following preliminary observation.

\begin{prop}\label{Zariski_density_analytic_open}
	Let $X$ be a smooth irreducible variety over a field $k$ complete with respect to a non-trivial absolute value. Then, any non-empty open subset of $X(k)$ for its analytic topology is Zariski-dense in $X$.
\end{prop}

\begin{proof}
	A proof is given in \cite[Proposition 3.5.75 and Remark 3.5.76]{Poonen} in the case of local fields, but a careful inspection shows this proof only to depend on the completeness of $k$ (specifically in the use of the implicit function theorem over $k$ in the proof of \cite[Proposition 3.5.73 (i)]{Poonen}).
\end{proof}

We take note of the following special case.

\begin{cor}\label{Zariski_density_int_points}
	Let $k$ be a complete discretely valued field. Let $\mf X$ be a smooth irreducible variety over the valuation ring $k^{\circ} = \{x \in k, \omega(x) \ge 0\}$ and let $X = \mf X_k$. Then, the set of integral points $\mf X(k^{\circ})$ is either empty or Zariski-dense in $X$.
\end{cor}

From this analytic observation, we draw the following consequence.

\begin{prop}\label{Zariski_density_fixator}
	Let $Q$ be the Zariski closure in $G$ of the pointwise stabiliser $P_{\delta(F)}$ of $\delta(F)$ in $G(k)$. Then $Q$ is a pseudo-parabolic subgroup in $G$ and, for any cone $F'$ in $\mc B(G, k)$ parallel to $F$, the pointwise stabiliser $\hat P_{\delta(F')}$ is Zariski-dense in $Q$.
\end{prop}

\begin{proof}
	We give an indirect proof, namely we first construct a suitable pseudo-parabolic subgroup $Q$ then prove that, for any cone $F'$ parallel to $F$, the pointwise stabiliser $$\hat P_{\delta(F')} = \{g \in G(k), \forall x \in \delta(F'), gx = x\}$$ is Zariski-dense in $Q$.
	
	Let $S$ be a maximal split torus of $G$ such that the apartment $A(G, S)$ contains the core $\delta(F)$ of $F$ \cite[Définition 3.3.2]{Char}. Because the direction $\vec \delta(F)$ is a Weyl facet in $V(G, S)$ (Example \ref{ex:cores_in_FJ}) , it corresponds to a unique parabolic subset $\Psi$ of $\Phi(G, S)$, namely the intersection of the root system with the dual cone of $\vec \delta(F)$ $$\Psi = \Phi(G, S) \cap \vec \delta(F)^{\vee} = \{\alpha \in \Phi(G, S)\ |\ \forall x \in \vec \delta(F), \langle \alpha, x\rangle \ge 0\}.$$
	This subset is indeed parabolic as it is clear that, whenever $\alpha, \beta \in \Psi$ and $\alpha + \beta \in \Phi(G, S)$, we have $\alpha + \beta \in \Psi$ and, if $\alpha \in \Phi(G, S)$, then because the hyperplane $\{x \in V(G, S), \langle \alpha, x \rangle = 0\}$ is part of the hyperplane arrangement that defines $\Sigma(G, S)$, the root $\alpha$ assumes a constant sign on each facet of $\Sigma(G, S)$, hence either $\alpha \in \Psi$ or $-\alpha \in \Psi$. By \cite[Theorem C.2.15]{CGP}, the subset $\Psi$ then determines a unique pseudo-parabolic subgroup $Q$ of $G$ containing $S$ such that $\Psi$ is the set of $S$-weights in the Lie algebra of $Q$, namely $Q = P_G(\mu)$ for any cocharacter $\mu \in \vec \delta(F)$.
	
	We now prove that the pointwise stabiliser $\hat P_{\delta(F)}$ is Zariski-dense in $Q$. Fix a special vertex $o$ in $A(G, S)$ and the associated valuation $$(\varphi_{\alpha}: U_{\alpha}(k) \to \RR)_{\alpha \in \Phi(G,S)}$$ of the generating root datum $(Z_G(S)(k), (U_{\alpha}(k))_{\alpha \in \Phi(G,S)})$. Denote by $x$ the vertex of the cone $\delta(F)$. By \cite[Définition 7.1.8 and Proposition 7.4.4]{BT1}, the pointwise stabiliser $\hat P_{\delta(F)}$ is spanned by the subgroup $\hat N_{\delta(F)} = N_G(S)(k)\cap \hat P_{\delta(F)}$ and the root subgroups $U_{\alpha,f_{\delta(F)}(\alpha)	}$ for $\alpha \in \Phi$, where $$f_{\delta(F)}(\alpha) = \inf\{k \in \RR\,|\, \forall  z \in \delta(F), \langle \alpha, z-o \rangle + k\ge 0 \}  .$$ First of all, observe that if $n \in N_{\delta(F)}$, then $n$ fixes the vertex $x$ of $\delta(F)$ as well as the point $x+\mu$. Hence, the linear part of $n$ fixes $\mu$, that is $n$ commutes with $\mu$. Hence, we have the inclusion $\hat N_{\delta(F)} \subset Z_G(\mu)(k) \subset Q(k)$. Next, recall from \cite[Proposition 7.4.5]{BT1} that, if $\alpha \in \Phi(G,S)$ and $u \in U_{\alpha}(k)\setminus \{1\}$, then the set of points of $A(G, S)$ that are fixed under $u$ is the affine root $$\{z \in A(G, S), \langle \alpha, z-o \rangle + \varphi_{\alpha}(u) \ge 0\}.$$
	It follows that, if $\langle \alpha, \mu \rangle < 0$ and $u \in U_{\alpha}(k)\setminus \{1\}$, there exists $t > 0$ such that $x + t\mu$ is not fixed under $u$ and therefore $u$ does not fix $\delta(F)$. Hence, if $\langle \alpha, \mu \rangle < 0$, we have $U_{\alpha, \delta(F)} = \{1\}$. Therefore $\hat P_{\delta(F)}$ is spanned by a subgroup of $Z_G(\mu)(k)$ and subgroups of $U_{\alpha}(k)$ for $\langle \alpha, \mu \rangle \ge 0$, hence $\hat P_{\delta(F)} \subset Q(k)$. Consequently, the Zariski closure of $\hat P_{\delta(F)}$ lies in the Zariski closure of $Q(k)$, ie. $Q$ by Proposition \ref{Zariski_density_analytic_open}.
	On the other hand, if $\langle \alpha, \mu \rangle \ge 0$, we have $\langle \alpha, x-o \rangle = \inf \{ \langle \alpha, z-o\rangle, z \in \delta(F)\}$ and therefore $f_{\delta(F)}(\alpha) = \inf(\Gamma'_{\alpha} \cap [-\langle \alpha, x-o\rangle , +\infty[ ) < \infty$. Hence, for each nondivisible $\alpha \in \Phi(Q, S)$, the stabiliser $\hat P_{\delta(F)}$ contains the subgroup $U_{\alpha, f_{\delta(F)}(\alpha)}U_{2\alpha, f_{\delta(F)}(2\alpha)}$, which is the set of $k^{\circ}$-points of a smooth connected $k^{\circ}$-model of $U_{\alpha}$ by \cite[Lemme 4.4]{Lou}, and therefore Zariski-dense in $U_{\alpha}$ by Corollary \ref{Zariski_density_int_points}. Moreover, because the maximal compact subgroup $Z_G(S)(k)_b$ of $Z_G(S)(k)$ acts trivially on the apartment $A(G, S)$, it fixes $\delta(F)$ and is thus contained in $\hat N_{\delta(F)}$. Because $Z_G(S)(k)_b$ is the set of integral points of a smooth connected $k^{\circ}$-model of $Z_G(S)$ by \cite[Proposition 4.3]{Lou}, it is Zariski-dense in $Z_G(S)$. Consequently, the Zariski closure of $\hat P_{\delta(F)}$ contains $U_{\alpha}$ for each nondivisible $\alpha \in \Phi(Q, S)$ as well as $Z_G(S)$, hence it contains $Q$, which completes the proof.
	
	It is clear that the above argument is independent of the specific vertex $x \in A(G, S)$ of the affine cone $\delta(F)$. Hence, for any affine cone $F' \subset A(G, S)$ parallel to $F$, the pointwise stabiliser $\hat P_{\delta(F')}$ is Zariski-dense in $Q$.
	
	Lastly, if $\mf F \subset \mc B(G, k)$ is an affine cone that is parallel to $F$, then there exist parallel subcones $F' \subset F$ and $\mf F' \subset \mf F$ and an apartment $A(G, S')$ that contains both $\delta(F')$ and $\delta(\mf F')$. Then, several applications of the above result yield that the Zariski closures of the pointwise stabilisers of $\delta(F), \delta(F'), \delta(\mf F)$, and $\delta(\mf F')$ are all equal to $Q$.
	
\end{proof}

\begin{cor}\label{Transitivity_on_facade}
	For any maximal split torus $S'$ in $G$ such that $A(G, S')$ contains $\delta(F')$ for some cone $F'$ in $\mc B(G, k)$ parallel to $F$, we have the inclusion $S' \subset Q$.
\end{cor}

\begin{proof}
	The pointwise stabiliser $\hat P_{\delta(F')}$ contains the maximal compact subgroup $Z_G(S')(k)_b$ of $Z_G(S')$, which is the set of $k^{\circ}$-points of a smooth connected $k^{\circ}$-model of $Z_G(S')$ by \cite[Lemme 4.4]{Lou} and therefore Zariski-dense in $Z_G(S')$ by Corollary \ref{Zariski_density_int_points}. It follows that the Zariski closure $Q$ of $\hat P_{\delta(F')}$ contains $Z_G(S')$, and therefore $S'$.
\end{proof}

\subsubsection*{Action of the pseudo-parabolic subgroup on the façade}

We now deduce that $Q(k)$ acts on the façade $\mc B(G, k)_{F}$ through its quotient $(Q/\mc R_{us,k}(Q))(k) = Q(k)/\mc R_{us,k}(Q)(k)$, the first step towards identifying the façade with the Bruhat-Tits building of $Q/\mc R_{us,k}(Q)$.

\begin{prop}\label{Action_on_facade}
	The group $Q(k)$ is the stabiliser of the façade $\mc B(G, k)_{F}$ in $G(k)$. Moreover, the subgroup $\mc R_{us,k}(Q)(k)$ fixes $\mc B(G, k)_{F}$ pointwise.
\end{prop}

\begin{proof}
	It follows from the previous proposition that, if $F'$ is an affine cone in an apartment $A(G, S')$ such that $\delta(F')\parallel\delta(F)$, then we have $S' \subset Q$. We check that $Q(k)$ maps the façade $A(G, S')_{F'}$ into $\mc B(G, k)_{F}$. To this end, it is sufficient to check that the centraliser $Z_G(S')(k)$ and the root groups $U_{\alpha}(k)$ for $\alpha \in \Phi(Q, S')$ all map $A(G, S')_{F'}$ into $\mc B(G, k)_{F}$. Let $\mu \in X_*(S')$ be such that $Q = P_G(\mu)$, so that $$\Phi(Q, S') = \{\alpha \in \Phi(G, S')\,|\, \langle \alpha, \mu \rangle \ge 0\}.$$ First of all, the subgroup $Z_G(S')(k)$ acts on $A(G, S')$ by translations, and therefore permutes affine cones in $A(G, S')$ parallel to $\delta(F')$, hence $Z_G(S')(k)$ maps $A(G, S')_{F'}$ to itself. Next, we deal with root groups. Let $\alpha \in \Phi(G, S')$ such that $\langle \alpha, \mu \rangle \ge 0$ and $u \in U_{\alpha}(k)$. By \cite[Proposition 7.4.5]{BT1}, there exists a half-space $H$ in $A(G, S')$ with direction $$\vec H = \{x \in A(G, S')\,|\, \langle \alpha, x\rangle \ge 0\}$$ that $u$ fixes pointwise. Denote by $x'$ the vertex of $F'$. There exists $v \in X_*(S')$ such that $x+v \in H$ and therefore $x+v+\vec\delta(F') \subset H$. Then, because $u$ induces an affine map from $A(G, S')$ to $A(uS'u^{-1}, k)$ that fixes $H$, the cone $u(\delta(F')) = \delta(u(F'))$ is parallel to $u(x+v+\vec \delta(F')) = x+v+\vec \delta(F')$, and therefore to $\delta(F')$. Therefore, $u(F')$ is parallel to $F'$ and we have $u([F']) \in \mc B(G, k)_{F}.$ Moreover, if $\langle \alpha, \mu \rangle > 0$, then we may choose the cocharacter $v$ of the form $n \mu$ for some sufficiently large $n \in \NN$. Then, the cone $x+n\mu+\vec \delta(F')$ is a parallel subcone of $F'$ that is fixed by $u$, hence, in that case, $u$ fixes the equivalence class $[F']$. 
	
	We have thus proved that the group $Q(k)$ leaves $\mc B(G, k)_{F}$ globally invariant and that the subgroup $\mc R_{us,k}(Q)(k)$, which is spanned by the $U_{\alpha}(k)$ for $\langle \alpha, \mu \rangle > 0$, fixes $\mc B(G, k)_{F}$ pointwise.
	
	Assume that the stabiliser $\Stab_{G(k)}(\mc B(G, k)_F)$ properly contains $Q(k)$. Then, by \cite[Theorem C.2.23]{CGP}, there exists a pseudo-parabolic subgroup $Q' \supsetneq Q$ such that $$Q'(k) = \Stab_{G(k)}(\mc B(G, k)_F).$$
	Then, there exists a root $\alpha \in \Phi(Q', S)$ such that $\alpha \not \in \Phi(Q, S)$, ie. $\langle \alpha, \mu \rangle < 0$. Then, we have $-\alpha \in \Phi(Q, S) \subset \Phi(Q', S)$, hence $Q'(k)$ contains the root groups $U_{\alpha}(k)$ and $U_{-\alpha}(k)$. Consequently, by \cite[Proposition C.2.24]{CGP}, the subgroup $Q'(k)$ contains an element $m \in N_G(S)(k)$ that lifts the reflection $s_{\alpha} \in W(G,S)(k)$ with respect to the hyperplane $H_{\alpha} = \{x \in V(G, S)\, |\,  \langle \alpha, x \rangle = 0\}$. But the cone $\vec \delta(F)$ is contained in the open half-space $\{x \in V(G, S)\, |\, \langle \alpha, x\rangle < 0\}$ and is thus not invariant under $s_{\alpha}$, hence $m$ maps $\delta(F)$ to a cone with direction $s_{\alpha}(\vec \delta(F)) \ne \vec \delta(F)$, hence $m$ maps the equivalence class $[F] \in \mc B(G, k)_F$ outside of the façade $\mc B(G, k)_F$, a contradiction.
\end{proof}

\subsubsection*{Recovering the façade}

Lastly, we identify the façade $\mc B(G, k)_{F}$ with the Bruhat-Tits building of $Q/\mc R_{us,k}(Q)$ (under the guise of a Levi subgroup $L_Q$ of $Q$).

\begin{prop}\label{Identifying_facade}
	Let $S$ be a maximal split torus such that $\delta(F) \subset A(G, S)$ and $L_Q$ be the Levi subgroup of $Q$ that contains $S$. Let $f = F \cap A(G, S)$.
	\begin{enumerate}
		\item There exists a unique continuous $L_Q(k)$-equivariant surjection $$p: L_Q(k) \cdot A(G, S) \epi \mc B(G, k)_{F}$$ that extends the projection $A(G, S) \epi A(G, S)_f$.
		\item The map $p$ factors through a $L_Q(k)$-equivariant simplicial homeomorphism $$\bar p: \mc B(L_Q, k) \overset{\sim}{\to} \mc B(G, k)_{F}.$$
	\end{enumerate}
\end{prop}

\begin{proof}
	\begin{enumerate}
		\item Uniqueness follows directly from the $L_Q(k)$-equivariance of $p$. Our goal is now to prove that the $L_Q(k)$-equivariant map $$ \begin{array}{ccc}
		p: L_Q(k) \cdot A(G, S) & \longto & \mc B(G, k)_F \\
		x & \longmapsto & \lim\limits_{t \to +\infty} \left( x + t\mu  \right)
		\end{array}$$ is well-defined, where the operator $+$ denotes the action of the vector space $$V = \{x \in V(G, S)\, |\, \forall \alpha \in \Phi(G,S), \langle \alpha, \mu \rangle = 0 \implies \langle \alpha, x\rangle = 0\} $$ on $L_Q(k) \cdot A(G, S)$ defined in \cite[Proposition 7.6.4 (iii)]{BT1} and described in the Appendix.

        Let $S'$ be a maximal split torus in $L_Q$ and $x \in A(G, S')$. For each root $\alpha \in \Phi(G, S')$, we have $$\begin{array}{ccc}\langle\alpha, x + t\mu \rangle \underset{t \to +\infty}{\longto} +\infty &\text{ if } &\langle \alpha, \mu \rangle > 0\end{array}$$ and $$\begin{array}{ccc}\langle \alpha, x+t\mu \rangle \underset{t \to +\infty}{\longto}  \langle \alpha, x\rangle &\text{ if } & \langle \alpha, \mu\rangle =0\end{array}$$ hence, by \cite[Proposition 3.2.4]{Char} , the limit $p(x) = \lim_{t \to +\infty} (x+t\mu)$ exists and lies in the façade $A(G, S')_{\vec f'}  \subset \mc B(G, k)_F$, where $\vec f' \subset X_*(S')_{\RR}$ is the cone in the fan $\mc F^P$ that contains $\vec \delta(F)$. More precisely, for each $S' \subset L_Q$, the restriction of $p$ to the apartment $A(G, S')$ is the projection $x \mapsto [x+\vec f']$ of $A(G, S')$ onto its façade of type $\vec f'$ \cite[Définition 3.1.3]{Char}.
		In particular, the image of the map $p$ is the union $\bigcup_{S' \subset L_Q} A(G, S')_{\vec f'}$ over maximal split tori $S' \subset L_Q$. To prove that this union is the entire façade $\mc B(G, k)_F$, recall from \cite[Proposition 8.1.5]{Char} that the façade $\mc B(G, k)_F$ is the union of the façades $A(G, S')_{f'}$ for apartments $A(G, S')$ containing a cone $f'$ such that $\tilde f'\parallel F$. Let $A(G, S')$ be an apartment containing a cone $f'$ such that $\tilde f'\parallel F$. By Corollary \ref{Transitivity_on_facade}, there exists $q \in Q(k)$ such that $A(G, S') = qA(G, S)$. Then, because $Q(k)$ maps $\mc B(G, k)_F$ to itself, we must have $A(G, S')_{f'} = qA(G, S)_f$. Write $q = mu$ with $m \in L_Q(k)$ and $u \in \mc R_{u,k}(Q)(k)$. Because $\mc R_{u,k}(Q)(k)$ acts trivially on the $\mc B(G, k)_F$ by Proposition \ref{Action_on_facade}, we then have $A(G, S')_{f'} = mA(G, S)_f$. Hence, we have $$\mc B(G, k)_F = L_Q(k) \cdot A(G, S)_f$$ and the map $p$ is therefore surjective.
		Finally, the map $p$ is $L_Q(k)$-equivariant because the $V$-action on $L_Q(k) \cdot A(G, S)$ is continuous and commutes with the action of $L_Q(k)$.
		\item Recall from \cite[Proposition 7.6.4 (i)]{BT1} that we have a canonical $L_Q(k)$-equivariant map $\tilde \pi: L_Q(k) \cdot A(G, S) \epi \mc B(L_Q, k)$ that extends the projection $A(G, S) \epi A(L_Q, S)$. Moreover, the map $\tilde \pi$ is invariant under the action of $V$ on $L_Q(k) \cdot A(G, S)$ and it follows from \cite[Proposition 7.6.4 (iii)]{BT1} that any $V$-invariant map with domain $L_Q(k)\cdot A(G, S)$ factors through $\tilde \pi$. Now, observe that, if $S' \subset L_Q$ is a maximal split torus, $x \in A(G, S')$, $\mu \in X_*(S')$ such that $Q = P_G(\mu)$,  and $v \in V $, then for each $\alpha \in \Phi(G, S')$, we have $$\begin{array}{ccc}\langle\alpha, x + v + t\mu \rangle \underset{t \to +\infty}{\longto} +\infty &\text{ if } &\langle \alpha, \mu \rangle > 0\end{array}$$ and $$\begin{array}{ccc}\langle \alpha, x+v+t\mu \rangle \underset{t \to +\infty}{\longto}  \langle \alpha, x\rangle &\text{ if } & \langle \alpha, \mu\rangle =0\end{array},$$ and therefore $p(x+v) = p(x)$. The map $p$ is thus invariant under $V$ and, by the above discussion, there exists a map $$\ol p: \mc B(L_Q, k) \to \mc B(G, k)_F$$ such that $p$ admits the factorisation
		\begin{center}
			\begin{tikzcd}
				L_Q(k) \cdot A(G, S) \arrow{dr}{\tilde \pi} \arrow{rr}{p} && \mc B(G, k)_F\\
				& \mc B(L_Q, k) \arrow{ur}{\ol p}
			\end{tikzcd}.
		\end{center}
		The map $\ol p$ is surjective because $p$ is, and $L_Q(k)$-equivariant because $p$ and $\tilde \pi$ are $L_Q(k)$-equivariant and surjective. The restriction of $\ol p$ to the apartment $A(L_Q, S)$ is a surjective affine map between affine spaces of the same dimension, and therefore one-to-one. The $L_Q(k)$-equivariance of $\ol p$ then implies that its restriction to each apartment is one-to-one, and therefore that $\ol p$ is globally one-to-one -- as any two points in $\mc B(L_Q, k)$ are contained in an apartment -- and finally that $\ol p$ is bijective.
		To prove that $\ol p$ is a homeomorphism, we first observe that the restriction of $\ol p$ to $A(L_Q, S)$, a bijective affine map, is a homeomorphism onto its image. Next, we recall that if $x \in A(L_Q, S)$, then, by \cite[Theorem 3.11 	(d) and (e)]{Sol}, we have $$\mc B(L_Q, k) = L_Q(k)_x \cdot A(L_Q, S).$$ Further, because $Z(L_Q)(k)$ acts trivially on the building $\mc B(L_Q, k)$ and is contained and cocompact in the stabiliser $L_Q(k)_x$ by \cite[Proposition 3.9 (a) and (b)]{Sol}, we may write $$\mc B(L_Q, k) = \underbrace{\frac{L_Q(k)_x}{Z(L_Q)(k)}}_{\text{compact}} \cdot A(L_Q, S).$$
		Now, because $\bar p$ is $L_Q(k)$-equivariant, the action of $Z(Q)(k)$ on $\mc A(G, S)_f$ is also trivial and we also have $$\mc B(G, k)_F = \frac{L_Q(k)_x}{Z(L_Q)(k)} \cdot A(G, S)_f.$$
		Let $(u_n)_{n \in \NN}$ be a sequence of points in $\mc B(L_Q, k)$ that converges to a point $u \in \mc B(L_Q, k)$ and, for each $n \in \NN$, write $u_n = k_n x_n$ with $k_n \in L_Q(k)_x/Z(L_Q)(k)$ and $x_n \in A(G, S)$. Let $\varphi: \NN \to \NN$ be an increasing function. Then, because $L_Q(k)_x/Z(L_Q)(k)$ is compact, we may extract from $k_{\varphi(n)}$ a subsequence $k_{\varphi\circ \psi(n)}$ that converges to a point $k \in L_Q(k)_x/Z(L_Q)(k)$. Then, the sequence $x_{\varphi \circ \psi(n)}$ converges to $x = k^{-1} u$ and, by continuity of the $L_Q(k)$-action and of the restriction of $\ol p$ to $A(G, S)$, we may write $$\ol p(u_{\varphi \circ \psi(n)}) = k_{\varphi\circ \psi(n)} \ol p(x_{\varphi \circ \psi(n)}) \underset{n \to \infty}{\longto} k \ol p(x) = \ol p(u).$$
		We have thus proved that, from any subsequence of $(\ol p(u_n))$ we may extract a subsubsequence that converges to $\ol p(u)$, hence $(\ol p(u_n))$ converges to $\ol p(u)$. Hence, the map $\ol p$ is continuous. Because the above argument only relied on $L_Q(k)$-equivariance and the continuity of the restriction to an apartment, it also applies to $(\ol p)^{-1}$, hence $\ol p$ is a homeomorphism.
		
		Lastly, we check that $\ol p$ is simplicial, which amounts to proving that it maps walls in $\mc B(L_Q, k)$ to walls in $\mc B(G, k)_F$. By $L_Q(k)$-equivariance, it is sufficient to check this property in our distinguished apartment $A(L_Q, S)$. By \cite[Section 3.3.4]{Char}, the walls in $A(G, S)_f$ are the images under $p$ of the walls in $A(G, S)$ whose direction contains the cone $\vec f$ or, equivalently the subspace $V =  \Span_{\RR}(\vec f)$. Now, if $H \subset A(G, S)$ is a wall, then we have $\ol p(H) = p(\tilde \pi^{-1}(H))$. By \cite[Proposition 7.6.4 (ii)]{BT1}, the affine subspace $\tilde \pi^{-1}(H)$ is a wall of $A(G, S)$ and its direction contains $V$, as the linear part of $\tilde \pi$ is the quotient map $V(G, S) \epi V(L_Q, S) = V(G, S)/V$. Hence, the affine subspace $\ol p(H)$ is a wall in $A(G, S)_f$, which completes the proof.
	\end{enumerate}
\end{proof}

\subsubsection*{Relevant pseudo-parabolic subgroups}
We have thus proved that the stratification of $\ol{\vtheta_P(\mc B(G, k))}$ is indexed by pseudo-parabolic subgroups of $G$. In this last section, we characterise those pseudo-parabolic subgroups that give rise to a stratum in terms of $P$, thereby completing our description.

Fix a maximal split torus $S$ of $G$ and a minimal pseudo-parabolic subgroup $P_0$ such that $S \subset P_0 \subset P$. Let $\Phi = \Phi(G, S)$ be the root system associated to $S$ and $\Delta \subset \Phi$ be the system of simple roots corresponding to $P_0$. We endow $\Delta$ with its Dynkin diagram structure. All subsets of $\Delta$ thus inherit an induced graph structure. Denote by $J$ the type of $P$, that is the set of simple roots $ J = \Delta \cap -\Phi(P, S)$. We have proved in Section \ref{Comparison_compactifications} that $\ol{\vtheta_P(\mc B(G, k))}$ is the polyhedral compactification in the sense of \cite{Char} associated with the fan $\mc F^J$ of \cite[Définition 2.5.2]{Char}.

In keeping with the terminology of \cite{RTW1}, we give the following definition:

\begin{defn}\label{def:relevant_pseudo-par}
	Let $Q$ be a pseudo-parabolic subgroup in $G$. By \cite[Theorem 5.1.3 and Proposition 5.1.4 (iii)]{CP2}, there exists a unique pseudo-parabolic subgroup $Q$ in the conjugacy class of $Q$ such that $P_0 \subset Q$. Then, we say that $Q$ is \textit{$J$-relevant} if there exists $I \subset \Delta$, none of whose connected components is contained in $J$, such that the type of $Q'$ is $I \sqcup (J \cap I^{\perp})$.
\end{defn}

Then, we have 

\begin{prop}
	Let $Q$ be a pseudo-parabolic subgroup in $G$. The following are equivalent:
	\begin{enumerate}
		\item The subgroup $Q$ is $J$-relevant.
		\item There exists a cone $F$ in $(\mc F^J)_{\mc B(G, k)}$ such that $Q(k)$ is the stabiliser in $G(k)$ of the stratum $\mc B(G, k)_F$.
	\end{enumerate}
\end{prop}

\begin{proof}
	Because both statements are invariant under $G(k)$-conjugacy, we may assume that $Q$ contains $P_0$. Then, the equivalence essentially follows from the second example in \cite[Section 3.3.2]{Char}. Namely, the proof of Proposition \ref{Zariski_density_fixator} and Proposition \ref{Action_on_facade} imply that $Q$ is the stabiliser of a façade $\mc B(G, k)_F$ for some $F \in (\mc F^J)_{\mc B(G, k)}$ if and only if, for any  apartment $A(G, S')$ with $S' \subset Q$, the Weyl facet $$\vec \delta  = \{x \in V(G, S')\, |\, \forall \alpha \in \Phi(Q, S'), \langle \alpha, x \rangle \ge 0 \}$$ is the direction of the core of an affine cone in $(\mc F^J)_{A(G, S')}$.
	
	Now, by \cite[Section 3.3.2]{Char}, the directions of the cores of the cones in $(\mc F^J)_{A(G, S)}$ are precisely the Weyl facets whose type is of the form $I \cup (J \cap I^{\perp})$ for some $I \subset \Delta$ none of whose connected component is contained in $J$, from which we deduce both implications.
\end{proof}

We have thus proved the following theorem.

\begin{thm}\label{thm:stratification}
	The compactification $\ol{\vtheta_P(\mc B(G, k))}$ is $G(k)$-equivariantly homeomorphic to the disjoint union $$\ol{\vtheta_P(\mc B(G, k))} = \bigsqcup_{\substack{Q \subset G\\Q \,J\text{-relevant pseudo-parabolic}}} \mc B(Q/\mc R_{us,k}(Q), k).$$
	For each $J$-relevant pseudo-parabolic subgroup $Q$ of $G$, the group $Q(k)$ is the stabiliser of the stratum $\mc B(Q/\mc R_{us,k}(Q), k)$.
\end{thm}

For the sake of completeness, we now include the comparison between our definition of $J$-relevant and the characterisation given in \cite[Proposition 3.24 (ii)]{RTW1}, thereby proving that the terminology is consistent in the reductive case. 

\begin{prop}
	Let $Q$ be a pseudo-parabolic subgroup of $G$ containing $P_0$. Denote by $T \subset \Delta$ its type. Endow $T$ with the graph structure induced by the Dynkin diagram structure on $\Delta$. Let $\tilde T$ be the union of the connected components of $T$ that meet $\Delta \setminus J$. The following are equivalent:
	\begin{enumerate}
		\item $Q$ is $J$-relevant
		\item For any simple root $\alpha \in \Delta$, if $\alpha \in J$  and $\alpha \perp \tilde T$, then $\alpha \in T$.
	\end{enumerate}

\end{prop}

\begin{proof}
	Assume that $Q$ is $J$-relevant. Then, there exists $I$ such that $T = I \sqcup (J \cap I^{\perp})$. Observe that $I$ and $J \cap I^{\perp}$ are disconnected for the graph structure on $T$, so that the connected components of $T$ are precisely those of $I$ and those of $J \cap I^{\perp}$. Then, because no connected component of $I$ is contained in $J$, the connected components of $T$ that meet $\Delta \setminus J$ are precisely the connected components of $I$. Therefore, a simple root $\alpha$ that is in $J$ and orthogonal to $\tilde T = I$ is in $J \cap I^{\perp}$ and thus in $T$.
	
	Conversely, if $T$ satisfies condition 2, then setting $I = \tilde T$, no connected component of $I$ is contained in $J$ by construction, and we have $T = I \sqcup (J \cap I^{\perp})$. Indeed, condition 2 immediately implies that $J \cap I^{\perp} \subset T$. On the other hand, if $\alpha \in T \setminus I$ then, because $I$ is a union of connected components, we have $\alpha \in I^{\perp}$. Moreover, by construction, $T \setminus I$ is the union of the connected components of $T$ that are contained in $J$, hence $\alpha \in J$, hence $\alpha \in J \cap I^{\perp}$, hence $T \subset I \sqcup (J \cap I^{\perp})$, which completes the proof.
\end{proof}

\section{Appendix: Levi subgroups and their inner façades}

In the following, we fix a quasi-reductive group $G$ over a discretely valued field $k$ with a perfect residue field. Let $S$ be a $k$-split torus and $G_1 = Z_G(S)$ be its centraliser.
In this section, we relate the Bruhat-Tits building of $G_1$ to that of the ambient group $G$, which essentially amounts to rephrasing the results of \cite[§7.6]{BT1} in our particular context.

\subsection*{The inner façade}

Observe that the apartments in the Bruhat-Tits building $\mc B(G_1, k)$ are in one-to-one correspondence with the maximal $k$-split tori of $G_1$, that is the maximal $k$-split tori of $G$ that contain $S$. Therefore, a maximal $k$-split torus $T$ of $G$ containing $S$ gives rise to two apartments, namely $A(G, T)$ and $A(G_1, T)$, and two root systems $\Phi = \Phi(G, T)$, and $\Phi_1 = \Phi(G_1, T)$.  Moreover, these apartments are related by the map $$\pi: A(G, T) \epi A(G_1, T)$$ that associates to each valuation $\varphi = (\varphi_{\alpha})_{\alpha \in\Phi}$ in $A(G, T)$ the family $\varphi_1 = (\varphi_{\alpha})_{\alpha \in \Phi_1}$, which is established in \cite[Proposition 7.6.3]{BT1} to be a valuation of the generating root data $(Z_G(T)(k), (U_{\alpha}(k))_{\alpha \in \Phi_1})$ in $G_1(k)$. The map $\pi$ is well-defined as it preserves equipollence, affine, and its linear part is the quotient map $$V(G, T) = \frac{X_*(T)_{\RR}}{\Span_{\RR}(\Phi)^{\perp}} \epi V(G_1, T) = \frac{X_*(T)_{\RR}}{\Span_{\RR}(\Phi_1)^{\perp}}.$$
Further, if $T$ is a maximal $k$-split torus in $G_1$, then the subspace $\mc B (G_1, G, k)$ defined as $$\mc B(G_1, G, k) = G_1(k) \cdot A(G, T) =  \bigcup_{\substack{T' \text{ max. }k\text{-split torus of } G\\S \subset T'}} A(G, T')$$ may be thought of as a Euclidean extension of the Bruhat-Tits building of $G_1$, in the following sense.
By \cite[Proposition 7.6.4]{BT1}, there exists a unique $G_1(k)$-equivariant map $$\tilde \pi: \mc B(G_1, G, k) \to \mc B(G_1, k)$$ that extends $\pi$. Moreover, it enjoys the following properties:
\begin{enumerate}
	\item We have $\tilde \pi^{-1}(A(G_1, T)) = A(G, T)$, and the inverse image of each apartment of $\mc B(G_1, k)$ is an apartment in $\mc B(G, k)$.
	\item There exists a unique action of the vector space $$V = \ker\big( V(G, T) \to V(G_1, T)\big) = \frac{\Span_{\RR}(\Phi_1)^{\perp}}{\Span_{\RR}(\Phi)^{\perp}}$$ on $\mc B(G_1, G, k)$ that extends the action of $V$ on $A(G, T)$ by translation and commutes with the action of $G_1(k)$ in the sense that, for each $x \in \mc B(G_1, G, k)$, $g \in G_1(k)$, and $v \in V$, we have $$g \cdot (x+v) = g\cdot x + v.$$ Then, the map $\tilde \pi$ is constant on $V$-orbits and factors through a bijection $\mc B(G_1, G, k)/V \overset{\sim}{\longto} \mc B(G_1, k)$.
\end{enumerate}

To understand the action of $V$ more concretely, note that we have $$V = \frac{X_*(Z(G_1)) \otimes_{\ZZ} \RR}{X_*(Z(G)) \otimes_{\ZZ} \RR}.$$ Hence, $V$ naturally sits inside $V(G, T')$ for each maximal $k$-split torus $T'$ containing $S$. In particular, the vector space $V$ acts on each $A(G, T')$ by translation. Then, the action of $V$ on $\mc B(G_1, G, k)$ restricts to the translation action on $A(G, T')$ for each maximal $k$-split torus $T'$ in $G_1$.

\begin{rmk}
	With the terminology of \cite{Rou2}, the subspace $\mc B(G_1, G, k)$ endowed with the polysimplicial structure given by the affine roots with direction in $\Phi_1$ is an inessential affine building \cite[2.3.4]{Rou2}. It is the \textit{inner façade} associated to the sector face support $V$ \cite[3.3.14]{Rou2}. Point 2 translates in that language to the statement that the map $\tilde \pi$ realises the essentialisation map of $\mc B(G_1, G, k)$ \cite[2.1.12]{Rou2}.
\end{rmk}

\subsection*{Inner façades and field extensions}

For each discretely valued extension $K/k$ with perfect residue field, we set $$\mc B(G_1, G, K) := \mc B((G_1)_K, G_K, K) = G_1(K) \cdot A(G_K, T)$$ for $T$ any maximal $K$-split torus of $G_K$ containing $S_K$.

In the case where $K/k$ is Galois, we establish that the action of $\Gal(K/k)$ on $\mc B(G, K)$ stabilises $\mc B(G_1, G, K)$, and that both the essentialisation map and the $V$-action are compatible with the Galois action.

\begin{prop}\label{Galois_action_inner_facade}
	Let $K/k$ be a Galois extension of discretely valued fields with perfect residue fields. Let $T$ be a maximal $K$-split torus containing $S_K$. Let $$V^K = \ker(V(G_K, T) \to V((Z_G(S))_K, T)).$$
	Then
	\begin{enumerate}
		\item The action of $\Gal(K/k)$ on $\mc B(G, K)$ stabilises $\mc B(G_1, G, K)$.
		\item The map $\tilde \pi_K: \mc B(G_1, G, K) \to \mc B(G_1, K)$ is $\Gal(K/k)$-equivariant.
		\item The action of $V^K$ on $\mc B(G_1, G, K)$ commutes with the action of $\Gal(K/k)$ in the following sense: For each $v \in V^K$, each $x \in \mc B(G_1, G, K)$, and each $\sigma \in \Gal(K/k)$, we have $$\sigma(x+v) = \sigma(x) + \sigma(v).$$
	\end{enumerate}
\end{prop}

\begin{proof}
	\begin{enumerate}
		\item This is clear as the action of $\Gal(K/k)$ stabilises the set of maximal $K$-split tori containing $S_K$.
		\item Let $\sigma \in \Gal(K/k)$ and consider the automorphisms $$p_{\sigma}: \mc B(G, K) \to \mc B(G, K)$$ and $$p^{G_1}_{\sigma}: \mc B(G_1, K) \to \mc B(G_1, K).$$
		Then, the map $p_{\sigma}^{G_1}\circ \tilde \pi^K \circ p_{\sigma}^{-1}$ is $G(K)$-equivariant and restricts to an affine map from $A(G_K, T)$ to $A((G_1)_K, T)$. Moreover, its restriction to $A(G_K, T)$ is the map $$\pi^K: A(G_K, T) \to A((G_1)_K, T)$$ described in the previous section. Indeed, if we fix a valuation $(\varphi_{\alpha})_{\alpha \in \Phi(G_K, T)}$ corresponding to a point in $A(G_K, T)$, then it follows from Proposition \ref{Galois_action_valuated_root_data} that $p_{\sigma}\circ \tilde \pi^K \circ (p^{G_1}_{\sigma})^{-1}$ maps $(\varphi_{\alpha})$ to $(\varphi_{\alpha})_{\alpha \in \Phi((G_1)_K, T)}$.
		Moreover, it follows from Point 2 of Theorem 1 that the linear part of $p_{\sigma}\circ \tilde \pi^K \circ (p^{G_1}_{\sigma})^{-1}$ is the same as that of $\pi^K$, which completes the proof. We thus find that $$p_{\sigma}^{G_1} \circ \tilde \pi^K = \tilde \pi^K \circ p_{\sigma}.$$
		\item This follows directly from the fact that the action of $\Gal(K/k)$ on each apartment is affine and that $V^K$ acts on each apartment by translation.
	\end{enumerate}
\end{proof}

\bibliographystyle{halpha-abbrv}
\bibliography{biblio}

\end{document}